%% file: arxiv_main.tex
\documentclass[mnsc,nonblindrev]{informs3}

\OneAndAHalfSpacedXI 

\input{preamble}

\usepackage{cases}

\usepackage{nomencl}
\makenomenclature

\usepackage{algpseudocode}
\usepackage{tikz}
\usepackage[inline]{enumitem} 

\usepackage{natbib}
 \bibpunct[, ]{(}{)}{,}{a}{}{,}%
 \def\bibfont{\small}%
\TheoremsNumberedThrough     
\ECRepeatTheorems

\newtheorem{observation}[theorem]{Observation}

\EquationsNumberedThrough    

\begin{document}


\RUNAUTHOR{Gao et al.}

\RUNTITLE{Sequential Fair Allocation and Routing in Nonprofit Operations}


\TITLE{Sequential Fair Allocation and Routing in Nonprofit Operations}

\ARTICLEAUTHORS{%
\AUTHOR{Haiqing Gao~\textsuperscript{1}, Seyed M.R. Iravani~\textsuperscript{2}, Sean R. Sinclair~\textsuperscript{3}}
\AFF{Department of Industrial Engineering and Management Sciences, Northwestern University, Evanston, IL 60208 \\ \textsuperscript{1} \EMAIL{haiqinggao2027@u.northwestern.edu}; \textsuperscript{2} \EMAIL{s-iravani@northwestern.edu}; \textsuperscript{3} \EMAIL{sean.sinclair@northwestern.edu} }
} 

\ABSTRACT{%
\input{parts/abstract}
}%




\KEYWORDS{Sequential Fair Allocation, Simultaneous Allocation and Routing, Dynamic Programming, Fairness Objectives, Nonprofit Operations} 


\maketitle


\input{parts/introduction}
\input{parts/contributions}
\input{parts/related_work}

\input{parts/model_prelim}

\input{parts/allocation}

\input{parts/routing}

\input{parts/experiments}

\input{parts/conclusions}

\medskip

\noindent\textbf{Acknowledgments.} 
The authors would like to thank Will Ma, Vahideh Manshadi, and Irem Sengul Orgut for their helpful comments on a preliminary version of this manuscript.
The authors would additionally like to thank Dawson Ren for their help in developing the code for the simulations.

\bibliographystyle{informs2014} 
\bibliography{references} 


%
%
%

\makeatletter

\AtBeginEnvironment{APPENDICES}{%
  \renewcommand{\theHsection}{appendix.\Alph{section}}%
  \renewcommand{\theHsubsection}{\theHsection.\arabic{subsection}}%
  \renewcommand{\theHsubsubsection}{\theHsubsection.\arabic{subsubsection}}%
  \renewcommand{\theHfigure}{\theHsection.\arabic{figure}}%
  \renewcommand{\theHtable}{\theHsection.\arabic{table}}%
  \renewcommand{\theHequation}{\theHsection.\arabic{equation}}%
  \renewcommand{\theHtheorem}{\theHsection.\arabic{theorem}}%
}
\makeatother

\newpage

\crefalias{section}{appendix}
\begin{APPENDICES}
\OneAndAHalfSpacedXI 
\input{parts/appendix/dp_properties}
\input{parts/appendix/Threshold_properties}
\input{parts/appendix/decreasing_CV_proofs}
\input{parts/appendix/distirbution_dict}
\end{APPENDICES}








\end{document}

%% file: preamble.tex
\usepackage{amsmath,amsfonts,cancel}
\usepackage[shortlabels,inline]{enumitem}
\usepackage{algorithm}
\usepackage{mathtools}
\usepackage{tcolorbox}
\usepackage{booktabs}
\usepackage{xfrac}
\usepackage{xspace}

\usepackage{subcaption}

\usepackage{tensor}
\usepackage{bm}
\usepackage{tikz}
\usepackage{titlesec}
\usepackage{placeins}
\usepackage{changepage}
\usepackage[colorlinks=true,breaklinks=true,bookmarks=true,urlcolor=magenta,citecolor=magenta,linkcolor=magenta,bookmarksopen=false,draft=false]{hyperref}
\usepackage[capitalize]{cleveref}
\usepackage{pifont}

\newcommand{\Deff}{\Delta_{\text{eff}}\xspace}
\newcommand{\Dfair}{\Delta_{\text{fair}}\xspace}
\newcommand{\DfairEA}{\Dfair^{\text{XA}}}
\newcommand{\DfairEP}{\Dfair^{\text{XP}}}

\newcommand{\Vsig}{\boldsymbol{\sigma}}
\newcommand{\Vpi}{\boldsymbol{\pi}}
\newcommand{\Vmu}{\boldsymbol{\mu}}
\newcommand{\Vu}{\boldsymbol{u}}

\newcommand{\E}{\mathbb{E}}

\newcommand{\D}{\mathcal{D}}
\newcommand{\T}{\mathcal{T}}

\newcommand{\distF}{\overset{\rightarrow}{\mathcal{F}}}

\newcommand{\itaS}{\mathcal{S}}
\newcommand{\permu}[1]{\textsc{Perm}(#1)}

\newcommand{\lb}{d^{\min}}
\newcommand{\ub}{d^{\max}}
\newcommand{\hist}{\beta_{\min}}

\newcommand{\scarce}{\underline{c}}
\newcommand{\abundant}{\overline{c}}

\newcommand{\hqexpost}{\textsc{Forward}\xspace}
\newcommand{\mexpost}{\textsc{Ex-Post}\xspace}
\newcommand{\routing}{\text{routing}\xspace}

\usepackage{multirow}
\usepackage{thm-restate}

\mathchardef\mhyphen="2D 
\ifdefined \argmin
\else \DeclareMathOperator*{\argmax}{arg\,min}
\fi
\ifdefined \argmax
\else \DeclareMathOperator*{\argmax}{arg\,max}
\fi

\let\originalleft\left
\let\originalright\right
\renewcommand{\left}{\mathopen{}\mathclose\bgroup\originalleft}
\renewcommand{\right}{\aftergroup\egroup\originalright}

\renewcommand{\paragraph}[1]{\noindent {\bf #1.}}
\renewcommand{\subparagraph}[1]{\noindent {\em #1.}~~}

\definecolor{edits}{rgb}{1,0,0}

\usepackage{pifont}

\allowdisplaybreaks

\renewenvironment{proof}[1][Proof]{%
  \par\noindent{\itshape #1.} \ignorespaces
}{%
  \hfill\Halmos\par
}

\usepackage{color-edits}
\addauthor{srs}{orange}
\addauthor{hqg}{blue}
\addauthor{djr}{cyan}
\addauthor{si}{teal}
\addauthor{todo}{pink}
\addauthor{done}{violet}



%% file: parts/abstract.tex
We study a dynamic fair sequential allocation problem in which a central planner distributes a divisible resource across multiple locations under demand uncertainty. Motivated by applications such as humanitarian relief and food distribution, we incorporate routing decisions into the planner’s problem and  jointly optimize allocation and visitation order under two max-min fairness objectives, {\em ex-post} and {\em forward}. 
We first reveal an {\em equating property} of the optimal allocation for any fixed visitation order. This property is crucial to establish that the optimal allocation follows a {\em threshold structure}: at each location, demand is fully satisfied when sufficiently low and otherwise met proportionally.
We then characterize the optimal routing policy and show that, under certain conditions, visiting locations in decreasing order of their demand's coefficient of variation (CV) is optimal. Building on this insight, we propose a simple heuristic, PPA-deCV, which closely approximates the fairness-efficiency frontier of the jointly optimal policy. 
Next, through extensive numerical experiments, we compare multiple {\em fairness objectives} across different {\em fairness metrics}, demonstrating that improvements in surrogate max-min objectives do not necessarily translate into improvements in fairness metrics. Furthermore, we identify which objective is aligned with which fairness metric, providing practical guidance on objective selection to achieve fairness-efficiency trade-offs.

%% file: parts/introduction.tex
\section{Introduction}
\label{sec:introduction}
The non-profit sector, including food bank operations, vaccine distribution, and post-disaster relief, has encountered escalating challenges in recent years. For instance, after a decade of progress in reducing food insecurity in the United States, the COVID-19 pandemic reversed this trend. In 2024, 27\% of nonelderly adults reported experiencing food insecurity, a notable increase from 22\% in 2019~\citep{urban2024foodinsecurity}. These services are especially critical in rural areas where food insecurity is disproportionately high, since 84\% of U.S. counties with the highest rates of food-insecure children are rural~\citep{ginsburg2019unreliable,feeding_america}. Similar resource allocation challenges emerged during the early stages of the pandemic. In January 2021, the average weekly number of reported COVID-19 cases in the United States surged to approximately 8 million~\citep{nytimes_covid_data}. In response, the U.S. vaccination program was estimated to have reduced COVID-19 related deaths by a factor of 3.2 and hospitalizations by a factor of 4.9~\citep{commonwealth_vaccine_impact}. However, the Federal Emergency Management Agency (FEMA) faced criticism for its opaque and inconsistent distribution of scarce medical supplies from the Strategic National Stockpile~\citep{washingtonpost_2020}. 

These examples illustrate a common operational dilemma in stochastic resource allocation: how to balance the immediate fulfillment of demand with the need to reserve resources for future needs. While this trade-off is already present in conventional profit-maximizing settings, the non-profit context introduces an additional layer of complexity. Decision-makers must also consider fairness, so the objective becomes striking a balance between {\em efficiency} (maximizing the impact of limited resources) and {\em equity} (treating all recipients fairly)~\citep{varian1973equity}. For instance, a first-come-first-served policy may achieve full efficiency but result in highly inequitable outcomes for later recipients who may receive no resources~\citep{bertsimas2011price, sinclair2022sequential, vardi2024price}.

Stochastic resource allocation problems are often modeled as online sequential allocation problems, where demand is revealed over time, and decisions must be made without knowledge of future requests. Consider, for example, a mobile food pantry with a fixed supply of donated goods visiting a sequence of distribution sites. At each stop, the planner observes the local demand without knowledge of future demands, and must decide how much to allocate before moving on.
Unlike profit objectives, which are typically well-defined, fairness objectives lack a universal standard due to the subjectivity of fairness definitions and the dual goals of efficiency and equity. Prior work has proposed {\em surrogate single-objective functions} that balance these goals and developed heuristic policies to approximately optimize them~\citep{manshadi2021fair, lien2014sequential, ma2022fairness}. However, few studies have systematically compared these objectives or provided guidance for practitioners on selecting the most appropriate one based on their definitions of fairness.

Moreover, a critical operational dimension remains underexplored: the ability of the central planner to determine the \emph{visitation order} of locations. In many real-world operations, planners must simultaneously determine both the quantity of resources to allocate and the sequence of locations to visit. For example, when distribution points are adequately staffed, a driver may have the flexibility to choose the visitation order of sites rather than following a rigid schedule set weeks in advance. Even in cases of limited staffing, it is still possible to optimize a fixed schedule to improve overall fairness and efficiency. For instance, prioritizing a location with high demand variability early in the route can make subsequent demand more predictable, thereby enabling more informed and effective decision-making.

Motivated by the bi-criteria nature of fairness and efficiency, and the challenges of joint allocation and routing, this study seeks to address the following research questions:

\begin{adjustwidth}{2em}{2em}
\begin{center}
    \textit{What is the impact of the choice of a surrogate single objective function on the trade-off between fairness and efficiency? What are the structural properties of the optimal allocation and routing policies under different objective functions?}
\end{center}
\end{adjustwidth}

While our research is motivated by food banks and vaccine distribution, its implications extend to a wide range of domains, including post-disaster resource allocation~\citep{ma2022fairness}, electric vehicle charging~\citep{gerding2019fair}, distributed power systems~\citep{alyami2014adaptive}, refugee assignment~\citep{freund2023group}, hospital operations~\citep{canellas2025granular}, and federated cloud computing~\citep{ghodsi2011dominant}. 

%% file: parts/contributions.tex
\subsection{Main Contributions} \label{susbec: contri}

This paper presents dynamic programming formulations that jointly optimize allocation and routing decisions, and offers a systematic comparison of max-min fairness notions to inform the selection of fairness objectives across diverse operational contexts.
Our contributions are structured to address the research questions outlined earlier:

\paragraph{Impact of Objective} 
We measure the utility of each customer's allocation in terms of their fill rate (percentage of demand satisfied). Since the manager (henceforth ``principal'') aims to ensure an equitable distribution, we quantify fairness as the disparity between the maximum and minimum fill rates across distribution points~\citep{orgut2016modeling}. To balance fairness and efficiency, we compare two surrogate single-objectives, namely \mexpost and \hqexpost, proposed by \citet{manshadi2021fair} and \citet{lien2014sequential}, respectively. Our experimental results show that the \mexpost objective significantly improves ex-post fairness while maintaining nearly the same level of efficiency, whereas \hqexpost exhibits the opposite trade-off on the ex-ante frontier. This insight provides practical guidance that \mexpost is better suited for one-time allocation settings like post-disaster relief, while \hqexpost is more appropriate for recurring distribution contexts such as weekly food bank operations.

\paragraph{Structure of Optimal Allocation Policies} 
Our theoretical analysis establishes structural properties of optimal allocation policies under both fairness objectives. Given a fixed route, we show that the optimal policy follows a {\em threshold structure}: demands below a critical level are fully satisfied, while higher demands are partially satisfied (\Cref{thm:hq_threshold,thm:m_threshold}). Moreover, each objective satisfies an {\em equating property} (\Cref{lemma:hq_equate_property_cont,lemma:m_equate_property}), which equates the current fill rate with the (anticipated) future minimum fill rate. This property prevents overly conservative rationing for the future or overly aggressive allocations in the present.

\paragraph{Structure of Optimal Routing Policy} 
We next characterize routing decisions under the \mexpost objective implemented via the Projected Proportional Allocation (PPA) policy introduced by \cite{manshadi2021fair}. Under PPA, at each node the principal allocates inventory proportionally to the remaining expected demand to equalize ex-post fill rates across realized sample paths. We first show that it is optimal to visit any node with deterministic demand last (\Cref{thm:det_PPA}). We then prove under restrictions on the number of agents and their demand distributions that the optimal static routing policy orders nodes by decreasing coefficient of variation (CV) order under two canonical regimes: (i) when demands share a common mean but differ in variance (\Cref{thm:decreasing_var_two_node}), and (ii) when demands share a common variance but differ in mean (\Cref{thm:increasing_mean_routing_two_node}). In both cases, nodes with higher CV should be visited earlier. Motivated by this insight, we propose a heuristic called PPA-deCV, which combines the PPA allocation policy with the decreasing CV routing policy.

\paragraph{Numerical Simulations} 
We use extensive numerical experiments to distinguish between \emph{fairness objectives} and \emph{fairness metrics} by evaluating optimal allocation policies under different routing schemes. Although joint optimization yields only marginal improvements in objective values, it leads to substantial improvements in both ex-post and ex-ante fairness metrics. A key insight is that these gains are driven primarily by optimizing the routing policy, rather than by the allocation policy alone. We further characterize distinct fairness-efficiency frontiers induced by the \mexpost and \hqexpost objectives, providing actionable guidance for selecting objectives that align with specific operational priorities. Finally, we show that the proposed PPA-deCV heuristic achieves fairness and efficiency outcomes comparable to those of the jointly optimal policy, while significantly reducing computational complexity.

\paragraph{Paper Organization} 
\Cref{sec:model_preliminaries} defines the fairness objectives and presents their dynamic programming formulations. \Cref{sec:allocation,sec:sequencing} analyze the structural properties of optimal allocation and routing policies, respectively. \Cref{sec:experiments} evaluates the performance of different policies across a range of demand distributions and capacity levels. Finally, \Cref{sec:conclusion} concludes and outlines directions for future research.

%% file: parts/related_work.tex
\subsection{Related Work}
\label{sec: literature}
Nonprofit operations and fairness have garnered increasing attention in recent literature. On the application side, work on healthcare scheduling and resource allocation~\citep{shylo2013stochastic, liu2024patient} and humanitarian logistics and rescue operations~\citep{ouyang2024dynamic} highlights the operational challenges faced in nonprofit settings. Complementing this, a growing body of research focuses explicitly on fairness considerations in operations, including algorithmic pricing~\citep{cohen2022price}, disparity in service systems~\citep{cheng2025investigation}, fair allocation through elective information acquisition~\citep{cai2020fair}, sequential search with fairness constraints~\citep{aminian2025markovian}, online combinatorial optimization with group fairness constraints~\citep{golrezaei2024online}, fair scheduling~\citep{tassiulas2002maxmin, doulamis2007fair}, incentive design~\citep{freund2025fair}, and online contention resolution schemes capturing forward-looking fairness~\citep{ma2025forward}. We highlight the most closely related works below, including inventory routing, and static and stochastic fair resource allocation problems, especially as they relate to {\em sequential} fair allocation with the inclusion of {\em routing} decisions.

\paragraph{Inventory Routing Problem (IRP)} 
IRP models study sequential resource allocation in for-profit settings, integrating inventory management, vehicle routing, and delivery scheduling to minimize operational costs by optimizing the distribution process from a supplier to geographically dispersed customers~\citep{coelho2014thirty}.
The origins of IRP trace back to \citet{bell1983improving}, who first merged inventory control with vehicle routing in gas distribution, minimizing transportation costs while maintaining customer inventory levels. Subsequent work expanded IRP to incorporate trade-offs between distribution, inventory, and production setup costs \citep{federgruen1984combined, burns1985distribution, blumenfeld1985analyzing}. Researchers also developed dynamic allocation strategies for stochastic demand, proving the threshold structure of optimal policies \citep{bassok1995dynamic, kumar1995risk, berman2001deliveries}. These models align with our sequential resource allocation problem, where instead of profit-based objectives we consider fairness objectives. Though the objective is fundamentally different (usually fairness objectives are non-additive), we show that the optimal allocation policy preserves the threshold structure. Additionally, 
we do not consider replenishment due to the time-sensitive nature of demand in non-profit contexts, and so our initial supply is fixed.

Recent advances to IRPs leverage machine learning in developing approximation methods to address real-world complexities. Especially for direct delivery (where vehicle capacity restricts visits to one customer per route), researchers apply reinforcement learning to facilitate decision-making \citep{kleywegt2002stochastic, archibald2009indexability, bertazzi2013stochastic, hasturk2025constrainedreinforcementlearningdynamic, ortega2024stochastic}. As the primary focus of this paper is to identify structural properties, we omit detailed discussion on the computational research in IRP.

\paragraph{Deterministic Fair Resource Allocation} Fair allocation of resources has been extensively studied in deterministic settings, where the demands for all customers are fully known in advance, with little consideration for routing.
Seminal work by \citet{bertsimas2011price} systematically characterizes the efficiency-equity trade-off frontier for divisible goods under \textit{envy-free}, \textit{equitable}, \textit{proportional}, and \textit{max-min} fairness. Later \citet{lyu2025price} tightens the upper bound by introducing the worst-case service level. \citet{chaudhury2020efx} proves the existence of EF1 (an analog to envy-freeness for indivisible goods) allocations for three agents under deterministic demand. \citet{grigoryan2021effective} designs a matching algorithm for fair vaccine distribution. Other works address fairness in scheduling, where resources are shared rather than consumed~\citep{vardi2024price}, or study proportional fairness in indivisible allocations \citep{babaioff2022best}. Notably, these studies focus solely on allocation policies, assuming predetermined routes. 
Lastly, \citet{eisenhandler2019humanitarian} integrate routing and allocation decisions in a mixed supplier-customer setting with deterministic demand. Our work generalizes this by introducing stochastic demand and shifting focus from heuristic design to deriving structural properties of optimal policies.

\paragraph{Stochastic Fair Resource Allocation} In contrast to deterministic models, stochastic resource allocation settings assume resources or agents arrive sequentially.  A large body of work focuses on scenarios where resources arrive online while agents and their demands are fixed~\citep{bertsimas2013fairness, kaazempur2024data, celdir2025dynamic}.  More relevant to our work are models where agents and their demand arrive over time.
\citet{donahue2020fairness} consider a model where decisions are made upfront, prior to observing any demand realizations. They formulate a two-stage stochastic program, defining fairness as the probability of receiving a service.

Most relevant to this work are \citet{lien2014sequential, manshadi2021fair, ma2022fairness}.
 \citet{lien2014sequential} analyzes optimal allocation and \routing policies for food pantry operations in a two-node setting with symmetric demand distributions. We relax these conditions, derive broader structural insights, and rigorously compare fairness objectives and their dynamic programming (DP) formulations, where they provided only empirical observations. 
 \citet{manshadi2021fair} study resource allocation with correlated demands, providing competitive ratio performance guarantees of the projected proportional allocation (PPA) heuristic policy. In contrast, we focus on deriving structural results on the \textit{optimal} allocation and integrating \routing into decision-making. We also numerically study the performance of the PPA policy and compare it to our optimal allocation policy. \citet{ma2022fairness} formulate a bi-criteria objective that linearly combines fairness and efficiency. While their analysis centers on max-min fairness using shortfall as the performance metric, their two-threshold structure extends to the fairness definition in \citet{manshadi2021fair}. Although both \citet{ma2022fairness} and our work 
 characterize threshold structures, their discussion emphasizes the role of historical performance, whereas we focus on the influence of demand observations. Furthermore, we explore the fairness-efficiency trade-offs under the \hqexpost objective proposed by \citet{lien2014sequential} and the \mexpost objective proposed by \citet{manshadi2021fair}.

 {\citet{balcik2014multi} and \citet{alkaabneh2023routing} generalize the single-truck framework to multi-truck under max-min fairness objectives. Unlike their two-stage stochastic approach with predetermined routes, we develop a dynamic programming formulation where routing and allocation decisions occur post-demand observation. Furthermore, one can easily extend our work to a multi-truck setting by partitioning agents before observing demand. After the partitioning, the problem reduces to our single truck with a limited capacity setting. }

Beyond max-min fairness, \citet{sinclair2022sequential,banerjee2023online} explores multi-objective allocation based on the Nash Social Welfare (NSW) criterion. However, their solution is scale-invariant to demand and thus not applicable to the fill rate utilities considered in our study. \citet{fadaki2025sequential} also adopts the NSW objective, modeling the allocation problem as a Markov Decision Process (MDP) and solving it using Approximate Dynamic Programming (ADP). In contrast, \citet{si2022enabling} focuses on the $\alpha$-fairness metric within a horizon-fairness framework for sequential allocation. They propose an approximately optimal online algorithm and identify conditions under which performance guarantees can be established.

Our work advances this literature in three key aspects: (i) We provide both theoretical and empirical comparisons of max-min fairness variants, offering a broader perspective than prior studies that focused on single fairness objectives. (ii) We develop a systematic framework that characterizes the complete theoretical structure of optimal policies under different fairness criteria. This yields actionable insights for practitioners when selecting allocation objectives in diverse operational contexts. (iii) We extend the scope of fairness-aware decision-making by incorporating both allocation and routing decisions, which are critical in many real-world systems.

%% file: parts/model_prelim.tex
\section{Model and Preliminaries} \label{sec:model_preliminaries}

\paragraph{Technical Notation}  
For $N \in \mathbb{N}_+$, we let $[N] = \{1, 2, \ldots, N\}$.  We use script letters (e.g. $\D$) to denote distributions, lower-case letters (e.g. $d$) to denote their realizations, and $\E_\D[\cdot]$ to denote the expectation of a random variable when the randomness is dictated by $\D$.  Bold symbols denote vectors, e.g., $\Vpi = (\pi_1, \dots, \pi_N)$. We let $a \land b$ denote the minimum of $a$ and $b$, i.e. $\min\{a, b\}$. Lastly, we use $\permu{N}$ to denote the set of permutations of $[N]$.

\subsection{Problem Set-up} \label{subsec:prob_set_up}
\input{preliminary/prob_set_up}

\subsection{Dynamic Programming Formulations} \label{subsec:DP_formulation}
\input{preliminary/DP_formulation}

\subsection{Dynamic Programming Properties} \label{subsec:DP_properties}

\input{preliminary/DP_properties}

%% file: preliminary/prob_set_up.tex
\paragraph{Model Primitives} We consider a sequential resource allocation problem involving simultaneous allocation and \routing decisions (the selection of the order through which to traverse the different nodes). A truck departs from the depot with $c$ divisible units of a single type of resource\footnote{In practice, our results extend to the multi-resource setting where $c$ instead denotes a vector of $K$ different resources, and locations have {\em independent} demand for each type of resource.}. The goal is to allocate the scarce resource fairly over $N$ distribution points (or nodes). Each node $i \in [N]$ has a demand sampled independently from a known probability distribution $\D_i$ supported on the positive interval $[d_i^{\min}, d_i^{\max}] \subset \mathbb{R}_{>0}$, however, the realized demand is only revealed upon arrival. 
Initially, the driver decides to visit a node $\sigma_1$ in the set of nodes $[N]$, and the remaining unvisited set of nodes is $\itaS_1 := [N] \backslash \{\sigma_1\}$.  Then, over a series of stages $n = 1, \ldots, N$, the following occurs: 
\begin{itemize}
    \item The driver travels to node $\sigma_{n}$,
    and observes demand $d_{n}$ drawn from the probability distribution $\D_{\sigma_{n}}$.
    \item The driver makes an {\em irrevocable} allocation $\pi_{n}$ given the remaining capacity $c_n$ such that $0 \leq \pi_n \leq \min\{c_n, ~d_n\}$ and selects the next node $\sigma_{n+1}$ from the unvisited set of nodes $\itaS_n$.
    \item The remaining capacity is updated to $c_{n+1} = c_n - \pi_n$ and the unvisited set of nodes is updated to $\itaS_{n+1} = \itaS_n \backslash \{\sigma_{n+1}\}$.
\end{itemize}
This process continues until all nodes are visited at stage $n = N$. We denote $\Vsig$ as the \routing policy, $\Vpi$ as the allocation policy, and $\Vmu = (\Vsig, ~\Vpi)$ as the joint allocation and \routing policy.

\paragraph{Assumptions} 
We assume that demand is revealed only upon arrival, reflecting real-world nonprofit operations where volunteer-run services typically observe demand shortly before delivery. In contexts like vaccine distribution, demand fluctuates significantly and is difficult to estimate in advance. Our analysis focuses on a single-truck setting with fixed capacity, motivated by applications requiring immediate distribution without time for replenishment.

\paragraph{Metrics} We measure the utility for each node $i$'s allocation in terms of their {\bf fill-rate} $\beta_i = \frac{\pi_i}{d_i} \wedge 1$. It measures the percentage of node $i$'s demand satisfied under the allocation $\pi_i$. 
The principal's 
overarching goal is to design a sequence of allocation and \routing decisions that jointly optimize both {\em fairness} and {\em efficiency}. We start by measuring (un)fairness or (in)equity (we use these two terms interchangeably) via the maximum disparity in fill-rates across nodes: $\max_{i \in [N]} \beta_i - \min_{i \in [N]} \beta_i$.  
Since this is a random variable, we aggregate it either ex-post or ex-ante, each inducing a distinct notion of fairness:
\begin{definition}[Ex-post Fairness] \label{def:one_time_fair}
Given an initial capacity $c$ and a policy tuple $\Vmu = (\Vsig, ~\Vpi)$, the {\bf ex-post equity (fairness)} is:
\begin{equation} 
\DfairEP \triangleq \E_{\distF}[\max_i \beta_i - \min_i \beta_i],
\end{equation}
where $\distF = \left(\D_{\sigma_{1}}, \dots, \D_{\sigma_{N}} \right)$ denotes the joint demand distribution 
under the \routing policy $\Vsig = (\sigma_1, \dots, \sigma_{N})$, and $d_n \sim \D_{\sigma_{n}}$.
\end{definition}
This metric captures {\em short-term fairness}, as it evaluates the disparity in each realization of demand. In contrast, ex-ante or {\em long-term fairness} considers the disparity in expected fill rates across nodes:
\begin{definition}[Ex-ante Fairness] \label{def:long_term_fair}
Given an initial capacity $c$ and a policy tuple $\Vmu = (\Vsig, ~\Vpi)$, the {\bf ex-ante equity (fairness)} is:
\begin{equation} 
    \DfairEA \triangleq \max_i \E_{\distF}[\beta_i] - \min_i \E_{\distF}[\beta_i].
\end{equation}
\end{definition}
In applications such as post-disaster relief where resource distribution is a one-time event, ex-post fairness is more appropriate, as it captures fairness across all possible demand realizations. Conversely, in recurring distribution settings such as food bank operations, ex-ante fairness is more relevant, since stakeholders have a long relationship and are more concerned with average resource access over time rather than fairness in any single instance.
Next, we define the expected efficiency.
\begin{definition}[Efficiency] \label{def:efficiency}
Given an initial capacity $c$ and a policy tuple $\Vmu = (\Vsig, ~\Vpi)$, the expected {\bf efficiency} is defined as \begin{equation}
    \Deff = \dfrac{\E_{\distF} \left[ \sum \limits_{i = 1}^N  \pi_i \right]}{c}.
\end{equation}
\end{definition}

\paragraph{Objectives} 
 As written, this is a multi-objective optimization problem, trying to balance between fairness and efficiency. A common approach in the literature is to select a single objective that encompasses both terms simultaneously, such as \textit{egalitarian social welfare} (or \textit{Rawlsian fairness}), which maximizes the minimum utility among agents, namely $\min_i \beta_i$. Maximizing this quantity not only improves fairness (by narrowing the gap between fill rates), but also tends to allocate more resources, thereby reducing waste.  

For the principal we are interested in two objectives, the \mexpost objective in \citet{manshadi2021fair} and the \hqexpost objective in \citet{lien2014sequential}. Later in \Cref{sec:experiments}, we show that though these two objectives look similar, their equity-efficiency frontiers are fundamentally different. Note that the two objectives differ in how they handle the minimum fill rate among previous agents: the \mexpost objective tracks this quantity, whereas the \hqexpost objective does not. 

The first objective, \mexpost, incorporates the minimum fill rate over visited nodes into the decision-making process.
Given a policy tuple $\Vmu = (\Vsig, ~\Vpi)$, where $\Vsig$ is the \routing policy and $\Vpi$ is the allocation policy, it maximizes the expected minimum fill rate over sample paths, i.e.,
\begin{equation} \label{obj:m} 
\tag{\mexpost} 
    W_0^{\Vmu} (c) = \E_{\distF} \left[\min \limits_{n \in [N]} \frac{\pi_{n}}{d_{n}} \right].
\end{equation}
By maximizing the minimum fill rate across sample paths, \mexpost reserves resources upfront to avoid allocation failures that could occur in a single path and cannot be compensated later. However, this reservation behavior systematically assigns lower fill rates to later nodes, and thus results in worse long-term fairness compared to the \hqexpost objective, which is validated by our numerical results in \Cref{sec:experiments}.

In contrast, the second objective, \hqexpost, makes decisions without relying on the minimum fill rate over visited nodes.
Given a policy tuple $\Vmu = (\Vsig, ~\Vpi)$, where $\Vsig$ is the \routing policy and $\Vpi$ is the allocation policy, the \hqexpost objective maximizes the minimum of the current fill rate and the average of the future minimum fill rate, i.e.,
    \begin{equation} \tag{\hqexpost} \label{obj:hq}
        Z_0^{\Vmu} (c) = \E_{\D_{\sigma_1}} \left[ \E_{\D_{\sigma_2}} \left[\frac{\pi_1}{d_{1}} \land \E_{\D_{\sigma_3}} \left[\frac{\pi_2}{d_{2}} \land \dots \land \E_{\D_{\sigma_{N}}} \left[\frac{\pi_{N-1}}{d_{N-1}} \land \frac{\pi_N}{d_N} \right] \dots \right] \right] \right].
    \end{equation}

Despite its apparent complexity, \hqexpost aligns more closely with traditional dynamic programming (DP) frameworks, as we will show in \Cref{subsec:DP_formulation}. It balances immediate allocation decisions with expectations over future outcomes, and crucially, it reduces the dimension of the problem by relaxing the dependence on the minimum fill rate over visited nodes. This flexibility allows for the correction of suboptimal past decisions. As we will demonstrate in \Cref{sec:experiments}, while \hqexpost may compromise short-term fairness, it can lead to improved long-term fairness. 

The ultimate goal is to find the optimal policy $\Vmu^{*, ~w/z} = (\Vsig^{*, ~w/z}, ~\Vpi^{*, ~w/z})$ that maximizes \mexpost and \hqexpost respectively, i.e., 
\begin{align}
    W_0^{*} (c) &= \max_{\Vmu} ~W_0^{\Vmu} (c), \label{eqn:m_max}\\
    Z_0^{*} (c) &= \max_{\Vmu} ~Z_0^{\Vmu} (c), \label{eqn:hq_max}
\end{align}
where $\Vmu^{*, ~w}$, $\Vmu^{*, ~z}$ are optimal policy tuples for \mexpost and \hqexpost, respectively.
We begin by formulating \Cref{eqn:m_max,eqn:hq_max} as DPs and present basic properties of DP models to prepare for proving structural properties later.

%% file: preliminary/DP_formulation.tex
We start by presenting dynamic programming (DP) models of both the \mexpost and the \hqexpost objectives. Let $\hist^{n-1}$ 
be the minimum fill rate over the previous customers visited before stage $n$. Notice that $\beta_{\min}^0 = 1$ by default, so $\beta_{\min}^{n-1} \leq 1$, $\forall n$. We start off by presenting the dynamic program for the \mexpost objective.

\paragraph{\mexpost}
At each stage $n \in N$, the current state $\Vu_n = \left(c_n, ~\sigma_n, ~d_n, ~\itaS_n \right)$ contains the remaining capacity $c_n$, current node $\sigma_n$, their demand realization $d_n$, the set of unvisited nodes $\itaS_n$, and additionally the minimum fill rate over previously visited nodes $\beta_{\min}^{n-1}$.  
For an arbitrary policy tuple $\Vmu$, we have the Bellman equations for \mexpost,
\begin{equation} \label{eqn:m_bellman} 
    \begin{split}
        &W_n^{\Vmu}\left(\Vu_n, ~\beta_{\min}^{n-1} \right) = \E_{d_{n+1} \sim \D_{\sigma_{n+1}}} \left[W_{n+1}^{\Vmu}\left(\Vu_{n+1}, ~\beta_{\min}^{n-1} \land \frac{\pi_n}{d_n} \right) \right],\\
        &W_N^{\Vmu}\left(\Vu_N, ~\beta_{\min}^{N-1} \right) = \beta_{\min}^{N-1} \land \frac{\pi_N}{d_N},
    \end{split}
\end{equation}
where $\Vu_{n+1} = \left(c_n - \pi_n, ~\sigma_{n+1}, ~d_{n+1}, ~\itaS_{n} \backslash \{\sigma_{n+1}\} \right)$. 
The optimal policy $\Vmu^{*, ~w}$ for \mexpost additionally satisfies
\begin{equation} \label{eqn:m_dp} 
    \begin{split}
        &W_n^{*}\left(\Vu_n, ~\beta_{\min}^{n-1} \right) = \max \limits_{\scriptstyle{
        0 \leq \pi_n \leq c_n \land d_n, ~ \sigma_{n+1} \in \itaS_n}} \E_{d_{n+1} \sim \D_{\sigma_{n+1}}} \left[W_{n+1}^{*}\left(\Vu_{n+1}, ~\beta_{\min}^{n-1} \land \frac{\pi_n}{d_n} \right) \right],\\
       &W_N^{*}\left(\Vu_N, ~\beta_{\min}^{N-1} \right) = \beta_{\min}^{N-1} \land \frac{c_N}{d_N}, \quad \text{where $\pi_N^{*,~w} = c_N \land d_N$}.
    \end{split}
\end{equation}

We emphasize the dependence of \mexpost on the minimum fill rate over visited nodes by keeping $\beta_{\min}^{n-1}$ outside the state $\Vu_n$. We will later see that \hqexpost is memoryless. 
At stage $n = 0$, when the truck loads $c$ units of resource but has not yet visited any nodes, 
\begin{equation}
    W_0^{*} \left(c \right) = \max \limits_{\sigma_1 \in [N]} \E_{d_1 \sim \D_{\sigma_1}} \left[W_1^{*}\left(\Vu_1, ~1 \right) \right].
\end{equation}

\paragraph{\hqexpost}
At each stage $n \in N$, the current state $\Vu_n = \left(c_n, ~\sigma_n, ~d_n, ~\itaS_n \right)$ is defined the same as in the \mexpost dynamic program.
For an arbitrary policy tuple $\Vmu$, we have the Bellman equations for \hqexpost,
\begin{equation} \label{eqn:hq_bellman}       Z_n^{\Vmu}\left(\Vu_n\right) = \E_{d_{n+1} \sim \D_{\sigma_{n+1}}} \left[\frac{\pi_n}{d_n} \land Z_{n+1}^{\Vmu}\left(\Vu_{n+1}\right) \right] \text{ for $n < N$}, \quad \text{and} \quad
        Z_N^{\Vmu}\left(\Vu_N \right) = \frac{\pi_N}{d_N} \land 1,
\end{equation}
where $\Vu_{n+1} = (c_n - \pi_n, ~\sigma_{n+1}, ~d_{n+1}, ~\itaS_{n} \backslash \{\sigma_{n+1}\})$.  
The optimal policy $\Vmu^{*, ~z}$ for \hqexpost additionally satisfies
\begin{equation} \label{eqn:hq_dp} 
    \begin{split}
    &Z^{*}_{n} \left(\Vu_n \right) = \max \limits_{{\scriptstyle
    0 \leq \pi_n \leq c_n \land d_n,~\sigma_{n+1} \in \itaS_n }} 
    \mathbb{E}_{d_{n+1} \sim \mathcal{D}_{\sigma_{n+1}} }\left[ \frac{\pi_n}{d_n} \land Z^{*}_{n+1} \left(\Vu_{n+1} \right) \right],\\
    &Z^{*}_{N} \left( \Vu_N \right) = \frac{c_N}{d_N} \land 1, \quad \text{where $\pi_N^{*,~z} = c_N \land d_N$.}
    \end{split}
\end{equation}
At stage $n = 0$, when the truck loads $c$ units of resource and decides the first node to visit, 
\begin{equation}
    Z_0^{*} \left(c \right) = \max \limits_{\sigma_1 \in [N]} \E_{d_1 \sim \D_{\sigma_1}} \left[Z_1^{*}\left(\Vu_1 \right) \right].
\end{equation}

We identify two key distinctions between the \mexpost and \hqexpost models. First, the \mexpost model incorporates the previous minimum fill rate $\beta_{\min}^{n-1}$, into the system state. Second, it tracks the exact minimum fill rate across all sample paths. These features make the \mexpost model more conservative than \hqexpost, as it constrains current allocations based on historical performance and computes the precise minimum fill rate. An intermediate objective, referred to as SRA-e in \citet{lien2014sequential}, includes the minimum fill rate over previous nodes $\hist^{n-1}$ but does not record the exact expected minimum. However, we do not consider SRA-e further in this paper.

%% file: preliminary/DP_properties.tex
To prepare for analyzing the optimal allocation and \routing policies, we begin by presenting fundamental properties of the dynamic programming models under both objective functions.  Detailed proofs are deferred to \Cref{appendix:dp_properties}.

We first establish the monotonicity of the \mexpost and \hqexpost objectives in the capacity and demand. Intuitively, with more resources and less demand, one has the freedom to allocate more resources and thus obtains a higher minimum fill rate. A key distinction between the two objectives lies in how they incorporate past decisions. The \mexpost objective incorporates the prior minimum fill rate $\beta_{\min}^{n-1}$, reflecting performance over the entire sample path; consequently, it is constrained by past allocations and always bounded above by the previously attained minimum fill rate. In contrast, the \hqexpost objective excludes the prior minimum fill rate $\beta_{\min}^{n-1}$ and evaluates only current and future allocations.

To formalize our results, we first distinguish between static and dynamic routing policies.
\begin{definition}[Static and Dynamic Routing Policies]
    A \routing ~policy $\Vsig$ is said to be {\bf static} if at any stage $n$ and for any state $\Vu_n$, $\sigma_n(\Vu_n) = \sigma_n$, i.e., the \routing policy is independent of the state $\Vu_n$.  Otherwise, the \routing policy $\Vsig$ is said to be {\bf dynamic}.
\end{definition}
\noindent Under a static routing policy, the visitation order $(\sigma_1, \ldots, \sigma_N)$ is fixed in advance, before any demand is realized. In contrast, a dynamic policy allows the route to adapt to observed states.

We now establish monotonicity properties under a fixed (static) routing order:
\begin{restatable}[Monotonicity]{lemma}{Monotone} \label{lemma:monotone}
Given any static route $\Vsig \in \permu{N}$, for any stage $0 < n \leq N$, we have
\begin{enumerate}[label=(\alph*)]
    \item $W^{(\Vsig,~*)}_n (\Vu_n, ~\beta_{\min}^{n-1})$ and $Z^{(\Vsig,~*)}_n \left( \Vu_n \right)$ are non-decreasing in the capacity $c_n$,
    \item $W^{(\Vsig,~*)}_n (\Vu_n, ~\beta_{\min}^{n-1})$ is non-decreasing in the minimum fill rate over visited nodes $\beta_{\min}^{n-1}$,
    \item $W^{(\Vsig,~*)}_n (\Vu_n, ~\beta_{\min}^{n-1})$ and $Z^{(\Vsig,~*)}_n \left( \Vu_n \right)$ are non-increasing in the demand $d_n$,
    \item $W^{(\Vsig,~*)}_n (\Vu_n, ~\beta_{\min}^{n-1})$ is no greater than the minimum fill rate over the visited nodes, i.e., $W^{(\Vsig,~*)}_n (\Vu_n, ~\beta_{\min}^{n-1}) \leq \beta_{\min}^{n-1}$.
\end{enumerate}
\end{restatable}

Building on monotonicity, we next show that, under a fixed routing order, \mexpost is concave in both the prior minimum fill rate and capacity, and \hqexpost is concave in capacity. This concavity generally fails under dynamic routing, where the objectives are the maximum over concave functions.

\begin{restatable}[Concavity]{lemma}{Concave}
\label{lemma:concave_in_pi}
Given a static \routing ~policy $\Vsig \in \permu{N}$, for any stage $0 < n \leq N$, we have
\begin{enumerate}[label=(\alph*)]
    \item $W^{(\Vsig,~*)}_n (\Vu_n, ~\beta_{\min}^{n-1})$ is concave in the minimum fill rate over visited nodes $\beta_{\min}^{n-1}$,
    \item $W^{(\Vsig,~*)}_n (\Vu_n, ~\beta_{\min}^{n-1})$ and $Z^{(\Vsig,~*)}_n (\Vu_n)$ are concave in the capacity $c_n$.
\end{enumerate} 
\end{restatable}

We next introduce the equating property, a structural feature shared by both the \hqexpost and \mexpost objectives. We show that the optimal allocation policy always satisfies this property, which reduces the search space by ruling out non-equating allocations. Since all preceding results are established under static routing, we continue in this setting and first present the \hqexpost result, which is conceptually simpler.

\paragraph{\hqexpost} We say that an allocation policy has the ``equating property" if it equates the fill rate at the current node with the future minimum fill rate to go function.
\begin{definition}[Equating Property under \hqexpost]
    Given a static \routing policy $\Vsig \in \permu{N}$, for any stage $0 < n < N$ and state $\Vu_n$ containing the current demand $d_n$, an allocation policy $\Vpi$ is said to satisfy the {\bf equating property} under \hqexpost if it equates the fill rate at the current node and the expected future fill rate to go for some future demand, that is, there exists a demand $d_{n+1} \sim \D_{\sigma_{n+1}}$ such that
    \begin{align} \label{eqn:hq_equating_current_future}
        \frac{\pi_n}{d_n} = Z^{(\Vsig, ~\Vpi)}_{n+1} (c_n - \pi_n, ~\sigma_{n+1}, ~d_{n+1}, ~\itaS_{n+1}).
    \end{align}
\end{definition}

We now show that the optimal allocation policy $\pi_n^{*,~z}(\Vu_n)$ satisfies the equating property. Intuitively, if $\pi_n^{*,~z}(\Vu_n)$ failed to equate the two values, shifting resources toward either the present or the future would strictly improve the fill rate. The result holds for continuous distributions with bounded, connected support, though the property itself is defined more generally.

\begin{restatable}[Equating Property for Optimality under \hqexpost]{proposition}{HqEquatingProperty} \label{lemma:hq_equate_property_cont}
    Suppose that the demand distributions $\D_i$ for all $i \in [N]$ are continuous with connected support. Given a static \routing policy $\Vsig \in \permu{N}$, for any stage $0 < n < N$ and state $\Vu_n$ with $c_n > 0$, any optimal allocation policy $\pi_n^{*,~z}(\Vu_n)$ under the \hqexpost objective satisfies the equating property. 
\end{restatable}

\paragraph{\mexpost} Similarly, we have the equating property defined under the \mexpost objective. Unlike \Cref{eqn:hq_equating_current_future}, it is defined sample path-wise due to \mexpost depending on the minimum fill rate from prior allocations.
\begin{definition}[Equating Property under \mexpost]
    Given a static \routing policy $\Vsig \in \permu{N}$, for any stage $0 < n < N$ and state $\Vu_n$ containing the current demand $d_n$, an allocation policy $\Vpi$ is said to satisfy the {\bf equating property} under \mexpost if it equates the current fill rate and the future minimum fill rate over a specific sample path, that is, there exists a sequence of demands $(d_{n+1}, \ldots,d_N)$ where $d_i \sim \D_{\sigma_{i}}$, for all $n+1 \leq i \leq N$ such that
    \begin{align} \label{eqn:m_equating_current_future}
        \frac{\pi_n}{d_n} = \frac{\pi_{n+1}}{d_{n+1}} = \cdots = \frac{\pi_N}{d_N}.
    \end{align}
\end{definition}
The same intuition in \Cref{lemma:hq_equate_property_cont} applies to \Cref{lemma:m_equate_property}, indicating that the optimal allocation policy should at least equate the fill rates on one possible sample path.
\begin{restatable}[Equating Property for Optimality under \mexpost]{proposition}{MaEquatingProperty} \label{lemma:m_equate_property}
    Suppose that the demand distributions $\D_i$ for all $i \in [N]$ are continuous with connected supports. Given a static \routing policy $\Vsig \in \permu{N}$, for any stage $0 < n < N$, state $\Vu_n$ and minimum fill rate over previously visited nodes $\beta_{\min}^{n-1}$, any optimal allocation policy $\pi_n^{*,~w}(\Vu_n, ~\beta_{\min}^{n-1})$ under the \mexpost objective satisfies the equating property.
\end{restatable}

Finally, for a fixed policy $\Vmu = (\Vsig, ~\Vpi)$, allowing for dynamic \routing, \hqexpost upper bounds \mexpost. This implies that \hqexpost overestimates the true ex-post minimum fill rate. 

\begin{restatable}[Upper Bound]{proposition}{Jensen} \label{lemma:upper_bound}
Given a routing and allocation policy $\Vmu = (\Vsig, \Vpi)$, we have
\begin{equation} \label{eqn:upper_bound}
    W^{\Vmu}_0 \left(c \right) \leq Z^{\Vmu}_{0} \left(c \right).
\end{equation}
\end{restatable}
\noindent While the two objectives share several structural properties, they differ in important ways. We examine these differences numerically in \Cref{sec:experiments}.

%% file: parts/allocation.tex
\section{Optimal Allocation Policy under Static Routing Policy}
\label{sec:allocation}
In this section, we study the optimal allocation policy under static \routing policies, which are operationally easier to implement as they allow nonprofits to coordinate volunteers, donors, drivers, and distribution points around a predetermined schedule. They also simplify the analysis, since static \routing policies are more analytically tractable. We show that an optimal allocation policy exhibits a threshold structure for both objectives. Specifically, for each node $n$, there exists a threshold $t_n$ such that if $d_n \leq t_n$, the optimal policy fully satisfies the demand, whereas if $d_n > t_n$, the optimal allocation is strictly less than the demand. Combined with the equating property (\Cref{lemma:hq_equate_property_cont}), this result restricts attention to policies that satisfy both the equating property and the threshold structure, thereby enabling a tractable characterization of optimal solutions. We also provide guidance for identifying these thresholds.
All proofs are provided in \Cref{appendix:allocation}.

\input{parts/our_threshold}

\input{parts/Manshadi_threshold}

%% file: parts/our_threshold.tex
\subsection{\hqexpost Objective} \label{sec: our_threshold}

We begin by establishing the threshold structure of an optimal allocation policy under the \hqexpost objective. Towards this, we introduce additional notation. 

Let $Z^{(\Vsig, ~*)}_n (\Vu_n)$ denote the optimal expected future minimum fill rate to go function under a static \routing policy $\Vsig$ in our DP model (\Cref{eqn:hq_dp}). We define
\begin{equation} \label{eqn:def_h}
    h_n(\pi_n; ~d_n) = \E_{d_{n+1}\sim \D_{\sigma_{n+1}}} \left[\frac{\pi_n}{d_n} \land Z^{(\Vsig, ~*)}_{n+1} (\Vu_{n+1}) \right],
\end{equation}
and the marginal profit on the boundary 
\begin{equation} \label{eqn:def_marginal_profit}
    \varphi_n(d_n) \triangleq \frac{\partial h_n}{\partial \pi_n} (\pi_n; ~d_n) \Big{|}_{\pi_n = d_n}.
\end{equation}
Our main result is the following.
\begin{restatable}[Threshold Policy under \hqexpost]{theorem}{HqThreshold} \label{thm:hq_threshold}
    Suppose that the distributions $\D_i$ for all $i \in [N]$ are continuous with connected support. Given a static routing policy $\Vsig \in \permu{N}$, for any stage $n$ and state $\Vu_n$, an optimal allocation policy $\pi_n^{*,~z} (\Vu_n)$ for the \hqexpost objective has a threshold structure. In particular, there exists a threshold $t_n(c_n) \in [0, ~c_n]$ such that
    \begin{equation} \label{threshold_policy_eqpost} 
        \pi_n^{*,~z} (\Vu_n) =
        \begin{cases}
            d_n, &\quad ~d_n \leq t_n(c_n),\\
            \Pi_n(\Vu_n), &\quad ~d_n > t_n(c_n),
        \end{cases}
    \end{equation}
    where $\Pi_n(\Vu_n) < d_n$ is a feasible stationary point satisfying $\frac{\partial h_n}{\partial \pi_n}(\pi_n;~d_n) \mid_{\pi_n = \Pi_n(\Vu_n)} = 0$. 
    
    Moreover, when $n = N$ the threshold value $t_n(c_n) = c_n$. For any stage $n < N$, 
    \begin{numcases}{t_n(c_n) =} 
            &$0, \quad c_n < \scarce$, \tag{\textit{scarce resource}}\\
            &$\T_n(c_n) \in [0, ~c_n]$, $\quad \scarce 
        \leq c_n \leq \abundant$, \tag{\textit{intermediate resource}}\\
            &$c_n,$ $\quad c_n > \abundant,$\tag{\textit{abundant resource}}
    \end{numcases} 
    where $\scarce \triangleq d_n^{\min} + d_{n+1}^{\min} + \sum \limits_{i = n+2}^N d_i^{\max}$, $\abundant \triangleq \sum \limits_{i = n}^N d_i^{\max}$, and $\T_n(c_n)$ is a root of the marginal profit on the boundary $\varphi_n(d_n)$.
\end{restatable}

\begin{remark}
Since $d_n \in [d_n^{\min}, ~d_n^{\max}]$, the threshold $t_n(c_n)$ should fall within this range. When the true threshold lies outside this interval we set $t_n(c_n) = 0$ or $t_n(c_n) = c_n$ for notational convenience; this does not affect the form or interpretation of the threshold structure.
\end{remark}

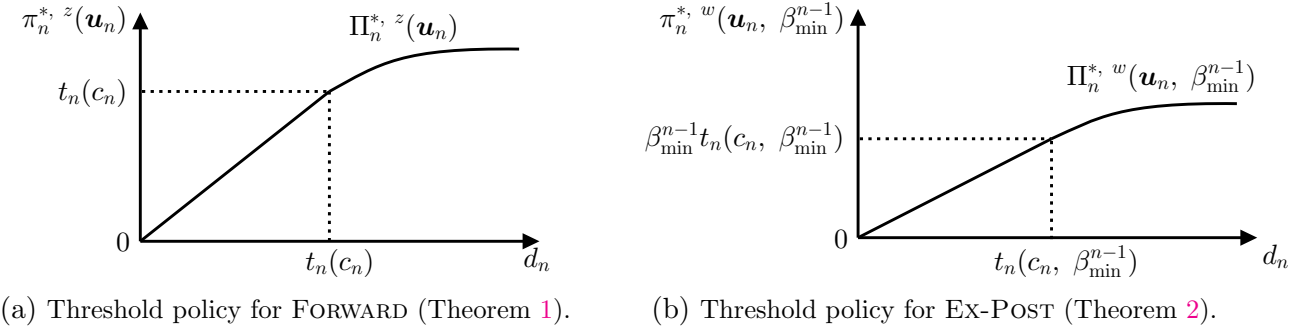
\begin{figure}[!t]
\centering
\begin{subfigure}{0.48\textwidth}
    \centering
    \scalebox{0.73}{\input{figures/opt_forward}}
    \caption{\small Threshold policy for \hqexpost (\Cref{thm:hq_threshold}).}
    \label{fig:opt_allocation_threshold_hq}
\end{subfigure}
\hfill
\begin{subfigure}{0.48\textwidth}
    \centering
    \scalebox{0.73}{\input{figures/opt_fdr}}
    \caption{\small Threshold policy for \mexpost (\Cref{thm:m_threshold}).}
    \label{fig:opt_allocation_threshold_m}
\end{subfigure}
\caption{Graphical representations of the threshold policies for the \hqexpost and \mexpost objectives with static \routing policies. In \Cref{fig:opt_allocation_threshold_hq} the slope of the line is one, whereas in \Cref{fig:opt_allocation_threshold_m} it is $\hist^{n-1}$.}
\label{fig:opt_allocation_threshold_combined}
\end{figure}

\noindent
\Cref{thm:hq_threshold} characterizes the structure of an optimal allocation across three distinct capacity regimes. See \Cref{fig:opt_allocation_threshold_hq} for a diagram.

When resources are scarce, \Cref{thm:hq_threshold} establishes that fully satisfying the demand at the current node may unnecessarily deplete capacity and reduce future fill rates. Specifically, if after allocating the demand at the current node, the capacity $c_n$ is unable to make the expected minimum fill rate to go function $Z_{n+1}^{(\Vsig, ~*)} (\Vu_{n+1})$ equal to one even under the minimum future demand $d_{n+1}^{\min}$, namely when $c_n < d_n^{\min} + d_{n+1}^{\min} + \sum \limits_{i = n+2}^N d_i^{\max}$, the threshold satisfies $t_n(c_n) = 0$ and demand is never satisfied.
When resources are abundant enough to cover all future demands, expressed as $c_n \geq \sum \limits_{i = n}^N d_i^{\max}$, it is optimal to always satisfy the current demand resulting in $t_n(c_n) = c_n$.
When resources are neither scarce nor abundant, the threshold $t_n(c_n)$ is non-trivial. In this regime, the threshold corresponds to the root of $\varphi_n(d_n)$, which identifies the critical demand level at fully satisfying the demand transitions from being beneficial to detrimental. 
\begin{proof}[Proof sketch]
First note that when $n = N$ an optimal allocation trivially satisfies a threshold structure via \Cref{eqn:hq_dp} since
\[
\pi_n^{*,~z}(\Vu_N) = c_N \wedge d_N.
\]
This follows a threshold structure with $t_N(c_N) = c_N$.
Hence, we focus on the case when $n < N$.  Recall that an optimal allocation policy $\pi_n^{*, ~z}$ satisfies (see \Cref{eqn:hq_dp}):
\[
\pi_n^{*,~z}(\Vu_n) = \argmax_{0 \leq \pi_n \leq c_n \wedge d_n} \E_{d_{n+1} \sim \D_{\sigma_{n+1}}} \left[ \frac{\pi_n}{d_n} \wedge Z_{n+1}^{(\Vsig,~*)}(\Vu_{n+1}) \right] = \argmax_{0 \leq \pi_n \leq c_n \wedge d_n} h_n(\pi_n;~d_n).
\]
From \Cref{lemma:concave_in_pi}, $Z_{n+1}^{(\Vsig,~*)}(\Vu_{n+1})$ is concave in $\pi_n$ and hence so is $h_n(\pi_n; ~d_n)$. Thus an optimal allocation $\pi_n^{*,~z}$ is either a stationary point of $h_n (\pi_n; ~d_n)$ or at the boundary $c_n \land d_n$.
We then establish that the $0 \leq \pi_n \leq c_n$ constraint is redundant.  Hence, an optimal allocation is either achieved at the boundary $d_n$ or a stationary point of $h_n(\pi_n; ~d_n)$ within the interval $(0, d_n)$.

To establish the threshold structure, we use the following intermediary lemma, which relates the sign of the derivative $\frac{\partial h_n}{\partial \pi_n} (\pi_n; ~d_n)$ at the boundary when $\pi_n = d_n$ to whether or not $h_n (\pi_n;~d_n)$ has a feasible stationary point. Indeed,
\begin{restatable}[Boundary Condition for Optimal Allocation]{lemma}{MarginalProfit} \label{lemma:marginal_profit_convexity}
    Given a static routing policy $\Vsig \in \permu{N}$, for any stage $n$ and state $\Vu_n$, let $h_n(\pi_n; ~d_n)$ be defined as in \Cref{eqn:def_h} and $\varphi_n(d_n)$ be defined as in \Cref{eqn:def_marginal_profit}.
    When the boundary function is non-negative ($\varphi_n(d_n) \ge 0$), it is optimal to fully allocate the demand, so $\pi_n^{*,~z}(\Vu_n) = d_n$. When the boundary function is negative ($\varphi_n(d_n) < 0$), it is optimal to allocate strictly less than the demand, so $\pi_n^{*,~z}(\Vu_n) < d_n$.
\end{restatable}
The remainder of the proof proceeds by considering three capacity regimes. In the scarce capacity regime, we show that for all demands $d_n \in [d_n^{\min}, d_n^{\max}]$, we have $Z_{n+1}^{(\sigma,~*)}(\Vu_{n+1}) < 1$. Thus, allocating the demand is not optimal by \Cref{lemma:hq_equate_property_cont} due to violating the {\em equating property}. 
In the abundant capacity regime, we show that for all $d_n \in [d_n^{\min}, d_n^{\max}]$, $Z_{n+1}^{(\sigma,~*)}(\Vu_{n+1}) = 1$ and hence satisfying demand is trivially optimal. Finally, when the resource is in between, we apply the {\em equating property} to express $h_n(\pi_n; ~d_n)$ by two integrals and show that the {boundary function} $\varphi_n(d_n)$ is non-increasing ($\frac{\mathrm{d} \varphi_n}{\mathrm{d} d_n} \leq 0$). Hence, by the intermediate value theorem there exists a threshold $t_n(c_n)$ such that the {boundary function} at that point satisfies $\varphi_n(t_n(c_n)) = 0$.  
When $d_n \leq t_n(c_n)$, the boundary function is non-negative ($\varphi_n(d_n) \geq 0$), and the policy allocates the full demand; when $d_n > t_n(c_n)$, the boundary function becomes negative ($\varphi_n(d_n) < 0$), and the policy allocates strictly less.
\end{proof}

%% file: figures/opt_forward.tex
\tikzset{every picture/.style={line width=0.75pt}} 

\begin{tikzpicture}[x=0.75pt,y=0.75pt,yscale=-1,xscale=1,every node/.style={font=\large}]

\draw [line width=1.5]    (78,94.71) -- (78,248) ;
\draw [shift={(78,90.71)}, rotate = 90] [fill={rgb, 255:red, 0; green, 0; blue, 0 }  ][line width=0.08]  [draw opacity=0] (11.61,-5.58) -- (0,0) -- (11.61,5.58) -- cycle    ;
\draw [line width=1.5]    (78,248) -- (348.43,248) ;
\draw [shift={(352.43,248)}, rotate = 180] [fill={rgb, 255:red, 0; green, 0; blue, 0 }  ][line width=0.08]  [draw opacity=0] (11.61,-5.58) -- (0,0) -- (11.61,5.58) -- cycle    ;
\draw [color={rgb, 255:red, 0; green, 0; blue, 0 }  ,draw opacity=1 ][line width=1.5]    (208,145) .. controls (242.43,125.86) and (257.43,114.86) .. (338.43,115.86) ;
\draw [line width=1.5]    (78,248) -- (208,145) ;
\draw [line width=1.5]  [dash pattern={on 1.69pt off 2.76pt}]  (208,248) -- (208,145) ;
\draw [line width=1.5]  [dash pattern={on 1.69pt off 2.76pt}]  (79.14,145) --  (208,145) ;

\draw (60,241) node [anchor=north west][inner sep=0.75pt]   [align=left] {$\displaystyle 0$};
\draw (340,251) node [anchor=north west][inner sep=0.75pt]   [align=left] {$\displaystyle d_{n}$};
\draw (-5,85) node [anchor=north west][inner sep=0.75pt]   [align=left] {$\displaystyle \pi_n^{*,~z}(\Vu_n)$};
\draw (190,252) node [anchor=north west][inner sep=0.75pt]   [align=left] {$\displaystyle t_n(c_n)$};
\draw (20,135) node [anchor=north west][inner sep=0.75pt]   [align=left] {$\displaystyle t_n(c_n)$};
\draw (220,90) node [anchor=north west][inner sep=0.75pt]   [align=left] {$\displaystyle \Pi_n^{*,~z}(\Vu_n)$};

\end{tikzpicture}

%% file: figures/opt_fdr.tex
\tikzset{every picture/.style={line width=0.75pt}} 

\begin{tikzpicture}[x=0.75pt,y=0.75pt,yscale=-1,xscale=1,every node/.style={font=\large}]

\draw [line width=1.5]    (78,94.71) -- (78,248) ;
\draw [shift={(78,90.71)}, rotate = 90] [fill={rgb, 255:red, 0; green, 0; blue, 0 }  ][line width=0.08]  [draw opacity=0] (11.61,-5.58) -- (0,0) -- (11.61,5.58) -- cycle    ;
\draw [line width=1.5]    (78,248) -- (348.43,248) ;
\draw [shift={(352.43,248)}, rotate = 180] [fill={rgb, 255:red, 0; green, 0; blue, 0 }  ][line width=0.08]  [draw opacity=0] (11.61,-5.58) -- (0,0) -- (11.61,5.58) -- cycle    ;
\draw [color={rgb, 255:red, 0; green, 0; blue, 0 }  ,draw opacity=1 ][line width=1.5]    (211,180) .. controls (242.43,165.86) and (257.43,154.86) .. (338.43,155.86) ;
\draw [line width=1.5]    (78,248) -- (211,180) ;
\draw [line width=1.5]  [dash pattern={on 1.69pt off 2.76pt}]  (211,249) -- (211,180) ;
\draw [line width=1.5]  [dash pattern={on 1.69pt off 2.76pt}]  (79.14,180) -- (211,180) ;

\draw (60,241) node [anchor=north west][inner sep=0.75pt]   [align=left] {$\displaystyle 0$};
\draw (354.43,251) node [anchor=north west][inner sep=0.75pt]   [align=left] {$\displaystyle d_{n}$};
\draw (-60,85) node [anchor=north west][inner sep=0.75pt]   [align=left] {$\displaystyle \pi_n^{*,~w}(\Vu_n,~\hist^{n-1})$};
\draw (170,252) node [anchor=north west][inner sep=0.75pt]   [align=left] {$\displaystyle t_n(c_n,~\hist^{n-1})$};
\draw (-70,167) node [anchor=north west][inner sep=0.75pt]   [align=left] {$\displaystyle \hist^{n-1} t_n(c_n,~\hist^{n-1})$};
\draw (220,125) node [anchor=north west][inner sep=0.75pt]   [align=left] {$\displaystyle \Pi_n^{*,~w}(\Vu_n,~\hist^{n-1})$};

\end{tikzpicture}

%% file: parts/Manshadi_threshold.tex
\subsection{\mexpost Objective} \label{subsec: manshadi_threshold}
We next establish the threshold structure of the optimal allocation policy under the \mexpost objective in a similar manner.

Let $W_n^{(\Vsig, ~*)}\left(\Vu_n, ~\beta_{\min}^{n-1} \right)$ denote the optimal expected minimum fill rate under a static \routing policy $\Vsig$ in our DP model (\Cref{eqn:m_dp}). We redefine
\begin{equation} \label{eqn:def_h_for_w}
    h_n(\pi_n; ~d_n, ~\hist^{n-1}) = \E_{d_{n+1} \sim \D_{\sigma_{n+1}}} \left[W_{n+1}^{(\Vsig, ~*)}\left(\Vu_{n+1}, ~\beta_{\min}^{n-1} \land \frac{\pi_n}{d_n} \right) \right],
\end{equation}
and the marginal profit on the boundary 
\begin{equation} \label{eqn:def_marginal_profit_for_w}
    \varphi_n(d_n,~\hist^{n-1}) \triangleq \lim_{\pi_n \to {(\hist^{n-1}d_n)}^-}\frac{\partial h_n}{\partial \pi_n} (\pi_n; ~d_n, ~\hist^{n-1}).
\end{equation}
Our main result is the following.
\begin{restatable}[Threshold Policy under \mexpost]{theorem}{MThreshold} \label{thm:m_threshold}
    Suppose that the distributions $\D_i$ for all $i \in [N]$ are continuous with connected support. Given a static routing policy $\Vsig \in \permu{N}$, for any stage $n$, minimum fill rate over previous nodes $\hist^{n-1}$ and state $\Vu_n$, an optimal allocation policy $\pi_n^{*,~w} (\Vu_n, ~\hist^{n-1})$ for the \mexpost objective has a threshold structure. In particular, there exists a threshold $t_n(c_n, ~\hist^{n-1}) \in [0, ~\dfrac{c_n}{\hist^{n-1}}]$ such that
    \begin{equation} \label{threshold_policy} 
        \pi_n^{*,~w} (\Vu_n, ~\hist^{n-1}) =
        \begin{cases}
            \hist^{n-1}d_n, &\quad ~d_n \leq t_n(c_n, ~\hist^{n-1}),\\
            \Pi_n(\Vu_n, ~\hist^{n-1}), &\quad ~d_n >t_n(c_n,~\hist^{n-1}),
        \end{cases}
    \end{equation}
    where $\Pi_n(\Vu_n, ~\hist^{n-1}) < \hist^{n-1}d_n$ is a feasible stationary point satisfying $$\frac{\partial h_n}{\partial \pi_n}(\pi_n;~d_n,~\hist^{n-1}) \mid_{\pi_n = \Pi_n(\Vu_n, ~\hist^{n-1})} = 0.$$ 
    
    Moreover, when $n = N$ the threshold value $t_n(c_n,~\hist^{n-1}) = \dfrac{c_n}{\hist^{n-1}}$. 
    For any stage $n < N$, 
    \begin{numcases}{t_n(c_n,~\hist^{n-1}) =} 
            &$0, \quad c_n \leq \scarce^{\prime}$, \tag{\textit{scarce resource}}\\
            &$\T_n(c_n,~\hist^{n-1})$, $\quad \scarce^{\prime} 
        < c_n < \abundant^{\prime}$, \tag{\textit{intermediate resource}}\\
            &$\dfrac{c_n}{\hist^{n-1}},$ $\quad c_n \geq \abundant^{\prime},$\tag{\textit{abundant resource}}
    \end{numcases} 
    where $\scarce^{\prime} \triangleq \hist^{n-1} \sum \limits_{i = n}^N d_i^{\min}$, $\abundant^{\prime} \triangleq \hist^{n-1} \sum \limits_{i = n}^N d_i^{\max}$, and 
    $\T_n(c_n,~\hist^{n-1}) \in [0, ~\dfrac{c_n}{\hist^{n-1}}]$ is a root of the marginal profit on the boundary $\varphi_n(d_n,~\hist^{n-1})$.
\end{restatable}

\begin{remark}
Similar to \Cref{thm:hq_threshold}, the threshold $t_n(c_n, ~\hist^{n-1})$ should fall within $[d_n^{\min}, ~d_n^{\max}]$. When the true threshold lies outside this interval we set $t_n(c_n, ~\hist^{n-1}) = 0$ or $t_n(c_n, ~\hist^{n-1}) = \dfrac{c_n}{\hist^{n-1}}$ for notational convenience.
\end{remark}

The proof follows largely similar arguments to \Cref{thm:hq_threshold}.
Similar to \Cref{thm:hq_threshold}, \Cref{thm:m_threshold} characterizes the structure of an optimal allocation across three distinct capacity regimes.
See \Cref{fig:opt_allocation_threshold_m} for a diagram.

Note that the \mexpost objective explicitly incorporates the prior minimum fill rate, scaling the slope of the threshold policy by the historical fill rate $\hist^{n-1}$. This scaling constrains allocations at stage $n$ by the minimum fill rate achieved across previously visited nodes. Because $\hist^{n-1}$ is non-increasing over stages, the optimal policy allocates higher fill rates early to mitigate future constraints. Consequently, as we will show in the simulations (\Cref{sec:experiments}), this behavior influences the policy’s performance on ex-ante and ex-post fairness metrics.

Compared to \Cref{thm:hq_threshold}, the capacity regimes in \Cref{thm:m_threshold} are slightly different due to the front-loading strategy.  In the abundant resource regime, the condition remains the same with the inclusion of $\hist^{n-1}$, whereby the remaining capacity is enough to guarantee the same fill rate of $\hist^{n-1}$ to any possible future demand.  In the scarce resource regime, we now require that the remaining capacity cannot sustain the fill-rate $\hist^{n-1}$ under the {\em minimum} future demand realizations,
versus just having to look at the minimum demand for the subsequent two locations as in \Cref{thm:hq_threshold}.  For the intermediate regime, the threshold is again determined by the root of $\varphi_n(\cdot)$, identifying where allocating above the historical minimum fill rate ceases to be beneficial.

%% file: parts/routing.tex
\section{Optimal Routing Policy under the Projected Proportional Allocation (PPA) Policy} \label{sec:sequencing}

In this section, we study the optimal routing decisions under a fixed allocation rule, the {\em Projected Proportional Allocation} (PPA) policy introduced in \citet{manshadi2021fair}.  At stage $n$, given  $\Vu_n = \left(c_n, ~\sigma_n, ~d_n, ~\itaS_n \right)$, the PPA allocation is:
\begin{equation} \label{eqn:PPA}
\pi^{PPA}_n \triangleq \left(\frac{d_{n}}{d_{n} + \sum \limits_{i \in \itaS_n} \mu_{i}}c_n \right) \land d_{n}, \quad \text{where } \mu_{i} = \mathbb{E}_{d_k \sim \mathcal{D}_{i}}\left[d_k\right].
\end{equation}
Throughout this section, we primarily focus on the \mexpost objective, where we show that under certain conditions, the optimal routing policy visits nodes in decreasing order of their coefficient of variation (CV). By contrast, the \hqexpost objective does not appear to admit an analogous structural characterization, which we illustrate via a counterexample in \Cref{sec:determine_end}. Nevertheless, our numerical experiments suggest that decreasing-CV remains a strong heuristic under \hqexpost as well. Proofs for this section are deferred to \Cref{appendix:decreasing_CV}.

\subsection{Deterministic Nodes} \label{sec:determine_end}
We begin by studying routing in a setting where one of the nodes has deterministic demand. Intuitively, the deterministic node should be visited last. Since its demand is known in advance, visiting it early provides no informational benefit, whereas prioritizing stochastic nodes allows the policy to incorporate realized uncertainty into subsequent decisions. We show that this intuition holds under both the optimal policy and the PPA allocation rule.

\begin{restatable}[Certainty‑Last Principle]{theorem}{DeterministicOpt} \label{thm:det_opt}
For any initial capacity $c$ and $N \ge 2$ nodes, suppose that one node $i_D$ has deterministic demand $\mu_D > 0$. Then, under the \mexpost objective $W_0^*(c)$ and joint optimization of allocation and \routing, an optimal \routing policy $\Vsig^*$ places the deterministic node last ($\sigma^*_N = i_D$).
\end{restatable}

We now show that the same structural property holds under the PPA allocation rule under three nodes with symmetric distributions.  

\begin{restatable}[Certainty‑Last Principle under PPA]{theorem}{DeterministicPPA} 
\label{thm:det_PPA}
Consider a problem with $N = 3$ nodes where nodes $i_1$ and $i_2$ follow the symmetric discrete distribution in \Cref{fig:discrete_sym_distr} with common mean $\mu > 0$. The third node $i_D$ has deterministic demand $\mu_D > 0$. Then, for any initial capacity $c$, under the PPA policy and the \mexpost objective $W_0^{(*,~\Vpi^{PPA})}(c)$, an optimal \routing policy places the deterministic node last ($\sigma^*_N = i_D$).
\end{restatable}

We next turn to the \hqexpost objective and examine whether a similar principle holds. In contrast to the \mexpost objective, the \hqexpost objective does not exhibit a clear structural property for the routing policy. We present a simple counterexample showing that the optimal routing may place the deterministic node in the middle rather than at the end.

\begin{example}[Different Simultaneous Policies]
    Let $N = 3$, suppose we have probability distributions $\D_1 = \D_2 \sim \begin{cases} 4, &\text{with probability $\frac{1}{2}$},\\ 2, &\text{with probability $\frac{1}{2}$}. \end{cases}$
    $\D_3$ is deterministic with demand equal to $1$.
    When the initial capacity $c=2$ or $9$, the optimal static \routing policy under \hqexpost puts  the deterministic node at the end, i.e., $\Vsig^{*,~z} = (1, ~2, ~3)$.
    However, when the initial capacity $c=8$, the optimal static \routing policy under \hqexpost puts the deterministic node in the middle, i.e., $\Vsig^{*,~z} = (1, ~3, ~2)$.
\end{example}

In this example, we observe that under the \hqexpost objective, there is no uniform pattern about the optimal routing policy as both the scarce and abundant capacity puts the deterministic nodes at the end, however, a higher value is achieved by placing the deterministic node in the middle for the intermediate capacity level. Motivated by this observation, we focus on \mexpost for the remainder of the section, as it exhibits stronger structural properties.

\subsection{Decreasing Coefficient of Variation (CV)}\label{sec:decreasing_CV}

Building on the deterministic node insight in \Cref{sec:determine_end}, we conjecture a more general principle that nodes should be ordered by decreasing coefficient of variation (CV).
A higher CV corresponds to greater demand volatility, so visiting such nodes earlier allows the policy to resolve more uncertainty before allocating remaining capacity. While \citet{lien2014sequential} established this for the optimal allocation policy, we show that this intuition extends to the PPA policy as well.

\paragraph{Same Mean, Different Variances}  
We first consider symmetric discrete distributions with common mean $\mu$ (see \Cref{fig:discrete_sym_distr}). As the probability mass $p_i$ at the mean increases, the distribution becomes more concentrated and has lower variance. Intuitively, nodes with larger $p_i$ (lower variance) should be visited later. The theorem below confirms this intuition for a two-node problem under the PPA policy. 

\begin{figure}
    \centering
    \scalebox{0.73}{\input{figures/discrete_sym_distr.txt}}
    \caption{Probability mass function of the symmetric discrete distributions considered.}
    \label{fig:discrete_sym_distr}
\end{figure}
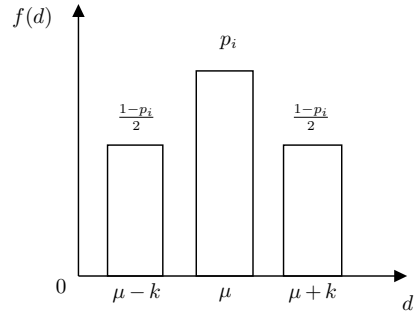

\begin{restatable}[Decreasing CV Routing with Equal Means]{theorem}{DecreasingVarTwoNode} \label{thm:decreasing_var_two_node}
Consider a two-node problem under PPA, where both nodes follow the symmetric discrete distribution in \Cref{fig:discrete_sym_distr} with common mean $\mu > 0$. Then, for any capacity $c$, the \routing order with larger variance first (i.e., $p_{\sigma_1} < p_{\sigma_2}$) is optimal.
\end{restatable}

The proof (see \Cref{appendix:decreasing_CV}) reveals the following intuition. When capacity is extremely scarce, the objective is dominated by whether the low-demand realization $\mu-k$ is satisfied. Because the first-stage demand is observed exactly, placing the more variable node first improves performance. When capacity is abundant, the low-demand realization is always satisfied, and the risk comes from the high-demand realization $\mu+k$. Even then, reducing uncertainty earlier benefits. 

\paragraph{Same Variance, Different Means}  
We next consider distributions with identical shape and variance that differ only by a shift in mean.  In this setting, the main question is whether to visit high-mean nodes earlier or later. Intuitively, it is preferable to face higher average demand when capacity is more plentiful, and to face lower demand as capacity becomes scarce.

\begin{restatable}[Decreasing CV Routing with Equal Variances]{theorem}{IncreasingMeanTwoNode} \label{thm:increasing_mean_routing_two_node}
Consider a two-node problem under the PPA policy, where the demand distributions $\D_1$ and $\D_2$ satisfy $\D_2 = \D_1 + \Delta$ for some $\Delta > 0$. Then, for any capacity $c$, an optimal \routing order visits the smaller-mean node first (i.e., $\mu_{\sigma_1} < \mu_{\sigma_2}$).
\end{restatable}

The proof constructs a one-to-one mapping between sample paths and shows that PPA’s structure discounts fill rates more heavily for a low-mean node when demand exceeds the mean. Thus, placing the higher-mean node second yields a uniformly better fill-rate profile.

Although decreasing CV is not universally optimal, experiments in \Cref{sec:experiments} show that it is nearly optimal across a broad range of practical settings.  Moreover, we note that similar results appear in \citet{lien2014sequential}, where they show that decreasing CV is optimal when $N = 2$ given the same distributions as in \Cref{thm:decreasing_var_two_node} and \Cref{thm:increasing_mean_routing_two_node} under the optimal allocation policy. 

%% file: figures/discrete_sym_distr.txt
\tikzset{every picture/.style={line width=0.75pt}} 

\begin{tikzpicture}[x=0.75pt,y=0.75pt,yscale=-1,xscale=1]

\draw    (140,210) -- (363,210) ;
\draw [shift={(366,210)}, rotate = 180] [fill={rgb, 255:red, 0; green, 0; blue, 0 }  ][line width=0.08]  [draw opacity=0] (8.93,-4.29) -- (0,0) -- (8.93,4.29) -- cycle    ;
\draw    (140,210) -- (140,26) ;
\draw [shift={(140,23)}, rotate = 90] [fill={rgb, 255:red, 0; green, 0; blue, 0 }  ][line width=0.08]  [draw opacity=0] (8.93,-4.29) -- (0,0) -- (8.93,4.29) -- cycle    ;
\draw   (160,120) -- (198,120) -- (198,210) -- (160,210) -- cycle ;
\draw   (281,120) -- (321,120) -- (321,210) -- (281,210) -- cycle ;
\draw   (221,69) -- (260,69) -- (260,210) -- (221,210) -- cycle ;

\draw (162,213.4) node [anchor=north west][inner sep=0.75pt]    {$\mu-k$};
\draw (283,213.4) node [anchor=north west][inner sep=0.75pt]    {$\mu+k$};
\draw (235,216.4) node [anchor=north west][inner sep=0.75pt]    {$\mu$};
\draw (123,210.4) node [anchor=north west][inner sep=0.75pt]    {$0$};
\draw (361,221.4) node [anchor=north west][inner sep=0.75pt]    {$d$};
\draw (93,23.4) node [anchor=north west][inner sep=0.75pt]    {$f( d)$};
\draw (236,43.4) node [anchor=north west][inner sep=0.75pt]    {$p_i$};
\draw (165,90.4) node [anchor=north west][inner sep=0.75pt]    {$\frac{1-p_i}{2}$};
\draw (285,90.4) node [anchor=north west][inner sep=0.75pt]    {$\frac{1-p_i}{2}$};

\end{tikzpicture}

%% file: parts/experiments.tex
\section{Experiments}
\label{sec:experiments}
So far we have focused on deriving optimal allocation and routing policies for the \mexpost and \hqexpost objectives, which are surrogate max-min formulations designed to balance fairness and efficiency (see \Cref{sec:model_preliminaries}). 
We proceed by addressing three research questions. First, we examine whether improvements in the surrogate objective translate into improvements in the fairness metrics, and we quantify the contribution of routing optimization to these gains. Second, we compare the \mexpost and \hqexpost objectives and provide guidance on objective selection across different operational contexts. Finally, we propose a heuristic, PPA-deCV, and numerically evaluate its performance relative to the jointly optimal allocation and routing policy. 
\footnote{Code can be found on GitHub: \url{https://github.com/Haiqing-Gao/Sequential_Fair_Allocation_Paper_Codes}.}

\begin{figure}[!t]
    \centering
    \begin{subfigure}{0.23\textwidth}
        \centering
        \includegraphics[width=\textwidth]{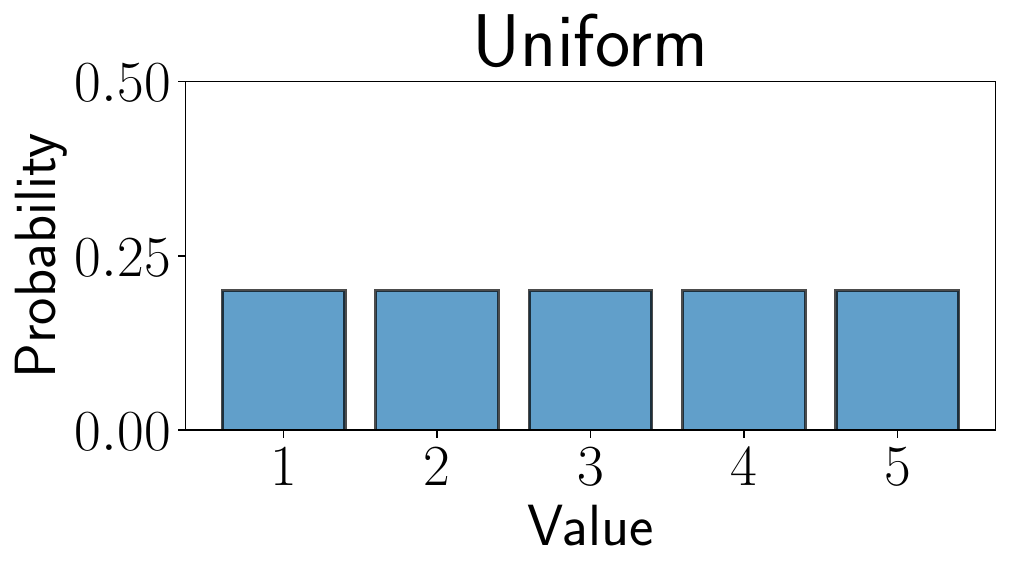}
    \end{subfigure}
    \hfill
    \begin{subfigure}{0.23\textwidth}
        \centering
        \includegraphics[width=\textwidth]{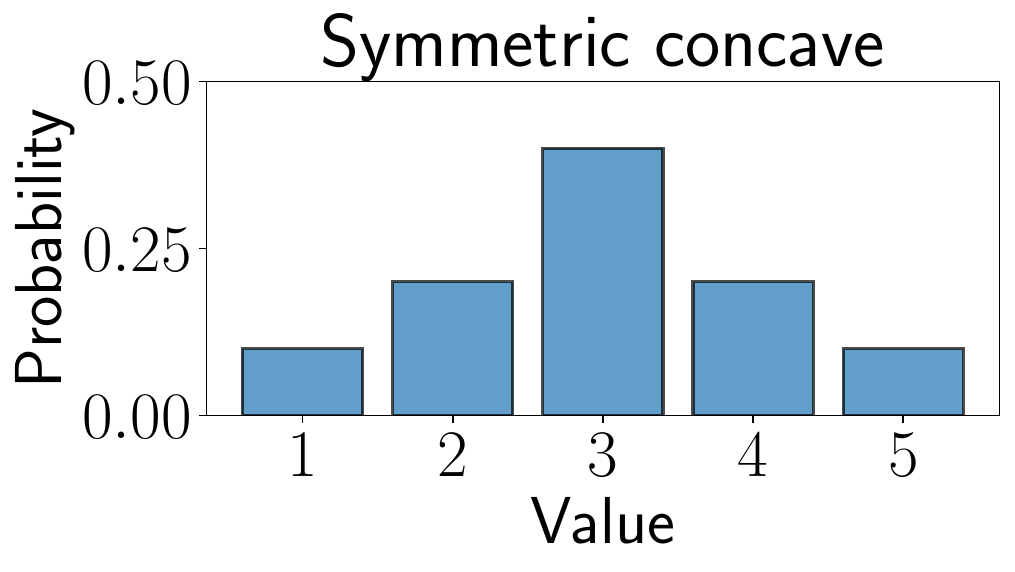}
    \end{subfigure}
    \hfill
    \begin{subfigure}{0.23\textwidth}
        \centering
        \includegraphics[width=\textwidth]{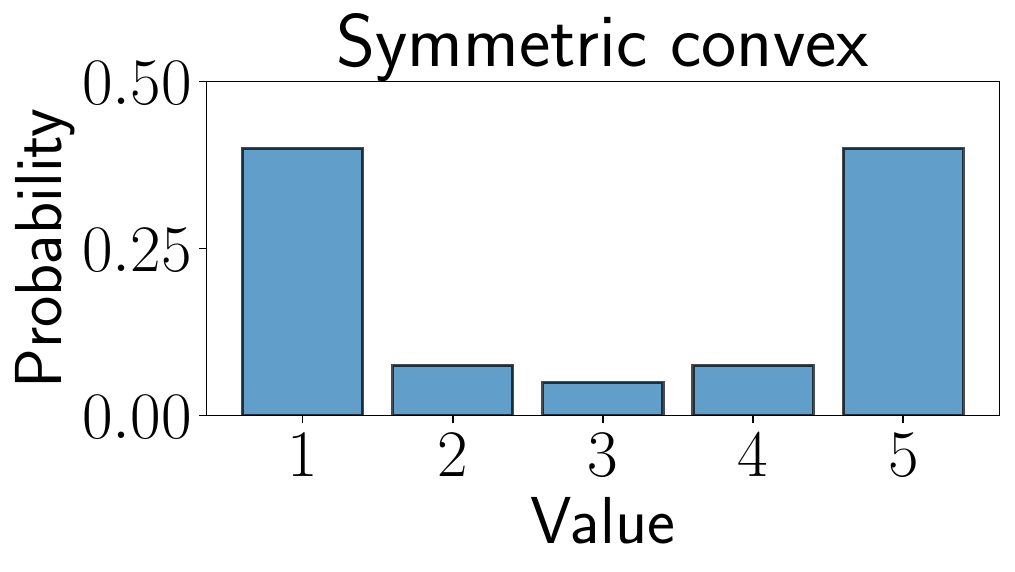}
    \end{subfigure}
    \hfill
    \begin{subfigure}{0.23\textwidth}
        \centering
        \includegraphics[width=\textwidth]{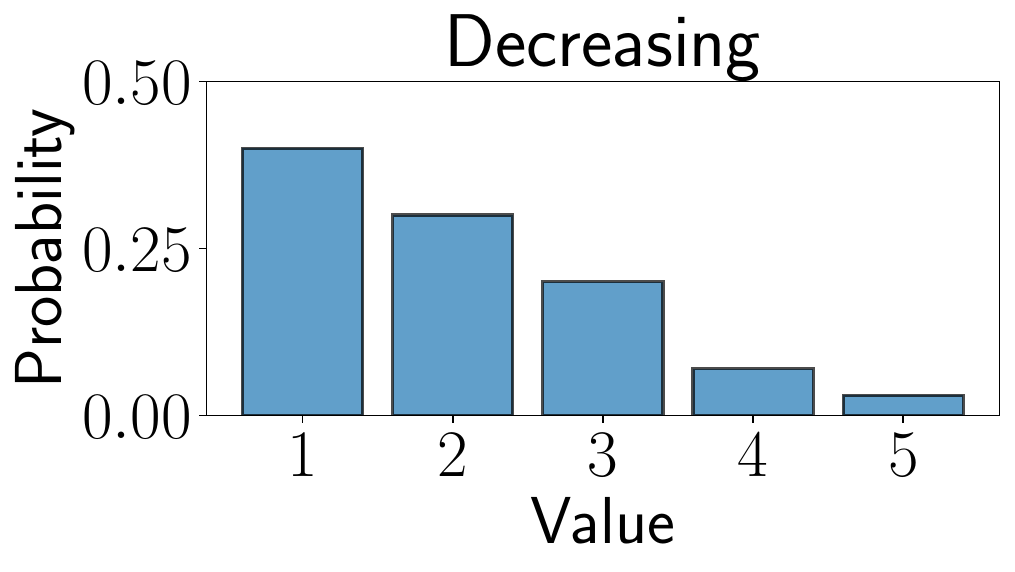}
    \end{subfigure}

    \vspace{0.5em}

    \begin{subfigure}{0.23\textwidth}
        \centering
        \includegraphics[width=\textwidth]{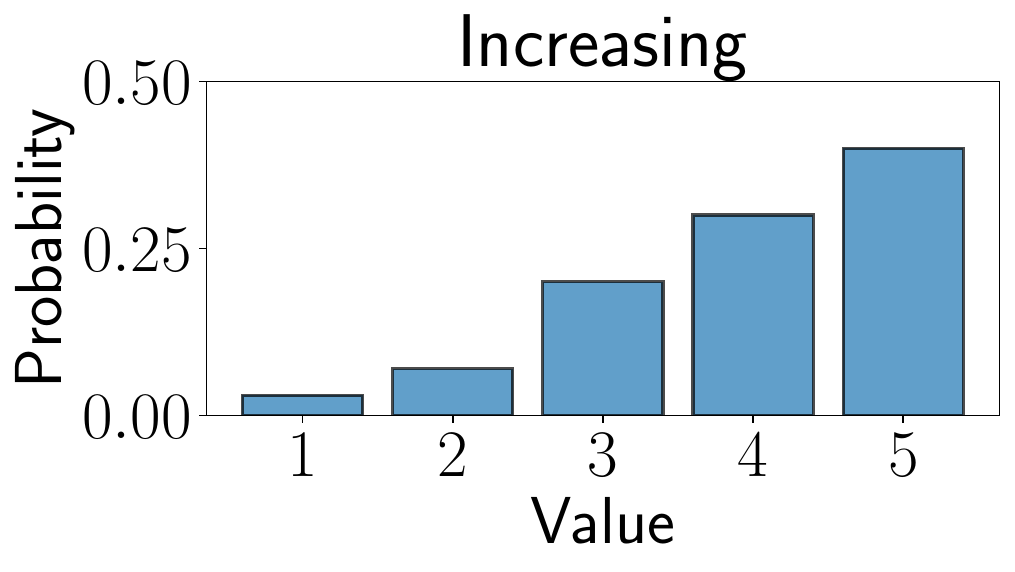}
    \end{subfigure}
    \hfill
    \begin{subfigure}{0.23\textwidth}
        \centering
        \includegraphics[width=\textwidth]{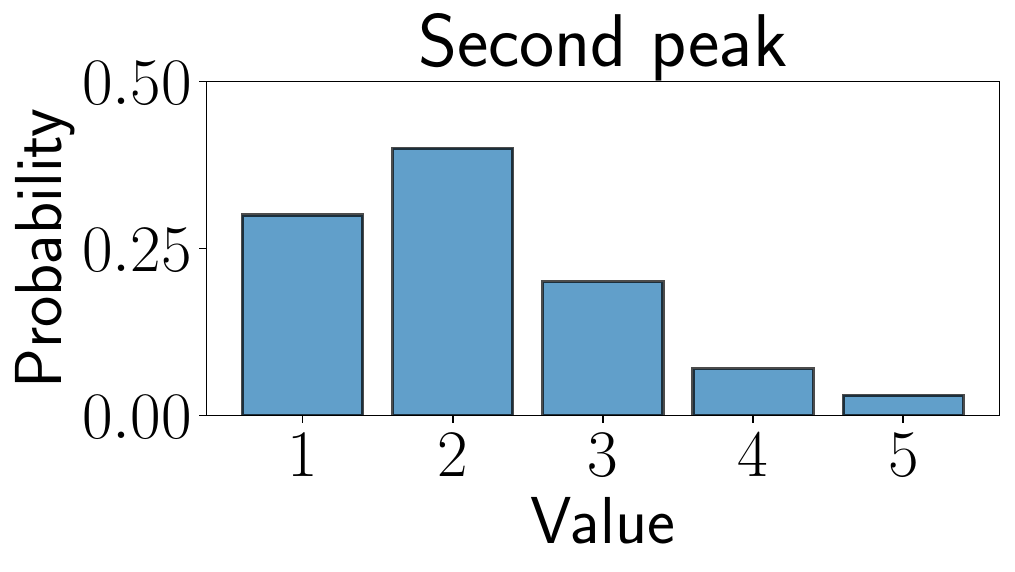}
    \end{subfigure}
    \hfill
    \begin{subfigure}{0.23\textwidth}
        \centering
        \includegraphics[width=\textwidth]{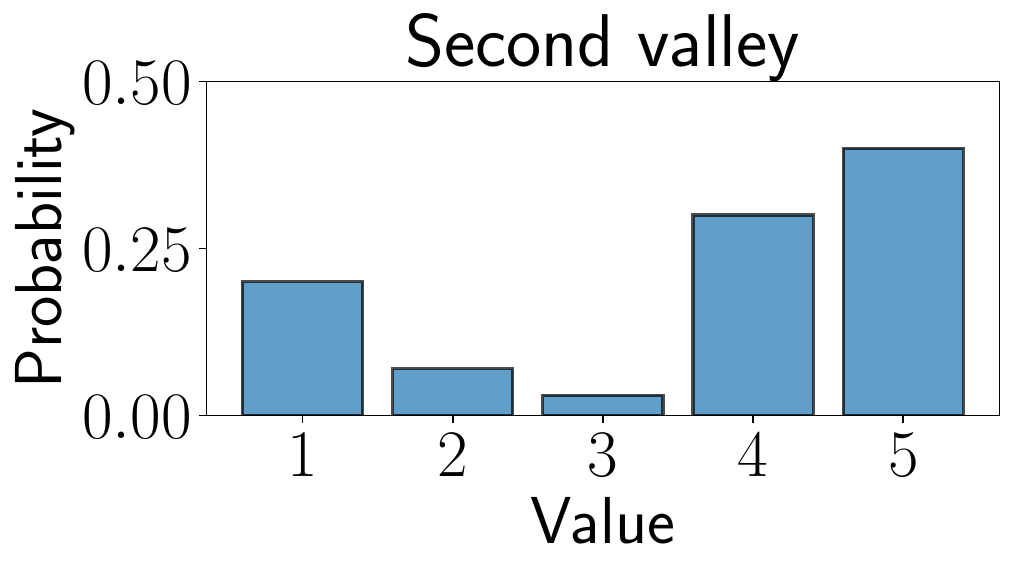}
    \end{subfigure}
    \hfill
    \begin{subfigure}{0.23\textwidth}
        \centering
        \includegraphics[width=\textwidth]{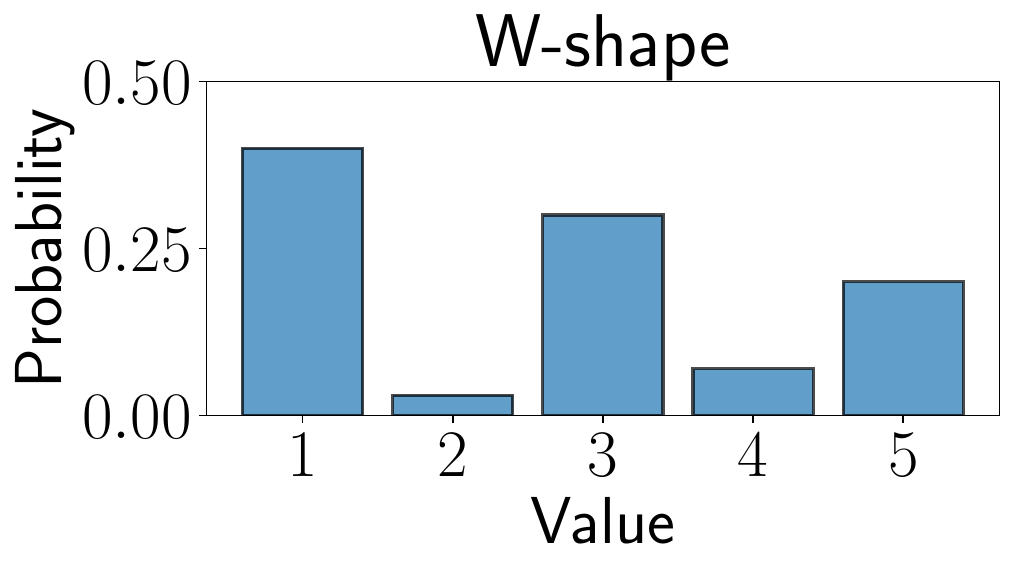}
    \end{subfigure}

    \caption{Probability mass functions of the eight discrete demand distributions.}
    \label{fig:distr_pmf}
\end{figure}

\subsection{Experimental Setup}
\label{subsec:sims_set_up}
We consider instances with $3$ to $6$ nodes, as computing the optimal dynamic \routing policy becomes computationally intensive. Demand at each node follows one of {\bf eight} discrete demand distributions, whose probability mass functions are shown in \Cref{fig:distr_pmf}. The full specification of these distributions is provided in \Cref{appendix:distr_dict}.

We study both {\em homogeneous} and {\em heterogeneous} systems. 
In {\em homogeneous} systems, we consider networks with $3$ to $6$ nodes, where all nodes share the same demand distribution selected from the eight candidates. In heterogeneous systems, we consider $3$ nodes where demand distributions vary across nodes according to the following configurations:
\begin{itemize}
\item \textbf{Same mean, different variances:} Each of the eight distributions is shifted by a constant so that all have a mean of $3.5$. We then evaluate all ${8 \choose 3}$ combinations of the distributions.
\item \textbf{Same variance, different means:} A base distribution $\mathcal{D}$ is selected from the eight candidates, and we construct a three-node system with demand distributions ${\mathcal{D}, \, \mathcal{D} + c, \, \mathcal{D} + 2c}$, where $c \in \{0.5, 1\}$.
\item \textbf{Same CV, different means and variances:} Each distribution is shifted by a constant so that the CV equals $0.35$, and all ${8 \choose 3}$ combinations are evaluated.
\item \textbf{Random:} Demand distributions are sampled directly from the original eight distributions, and all ${8 \choose 3}$ combinations are considered.
\end{itemize}
Finally, we normalize capacity by dividing by the sum of mean demands in each configuration, which we term the capacity level. We consider capacity levels ranging from $0.1$ up to the ratio of the maximum possible total demand to the total mean demand. Depending on the configuration, this results in $17$-$18$ capacity levels. Overall, this experimental design produces $7,776$ instances.

\subsection{Alignment between Max-min Fairness Objectives and Fairness Metrics} 
\label{subsec:sims_obj_vs_metric}
The surrogate objectives in \Cref{sec:model_preliminaries} are intended to balance fairness and efficiency. Here we ask: \emph{Does optimizing these objectives improve the fairness metrics $\DfairEP$ and $\DfairEA$?}

\begin{figure}[!t]
    \centering
    \begin{subfigure}{1.0\textwidth} 
        \centering
        \includegraphics[width=0.95\textwidth]{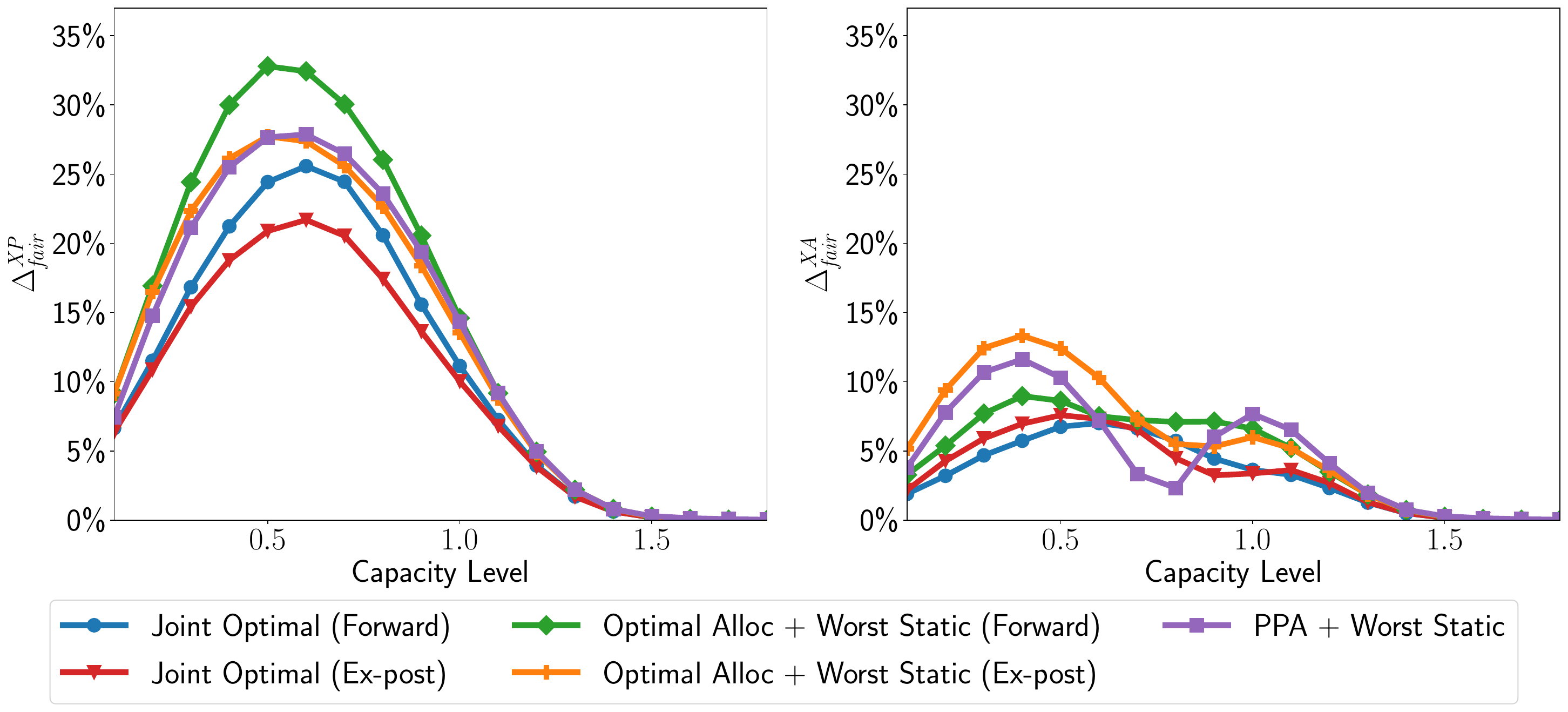}
        \caption{$\DfairEP$ and $\DfairEA$ (y‑axis) vs the relative capacity level (x‑axis).}
        \label{subfig:fair_improvement}
    \end{subfigure}

    \vspace{0.5em}

    \begin{subfigure}{1.0\textwidth} 
        \centering
        \includegraphics[width=0.95\textwidth]{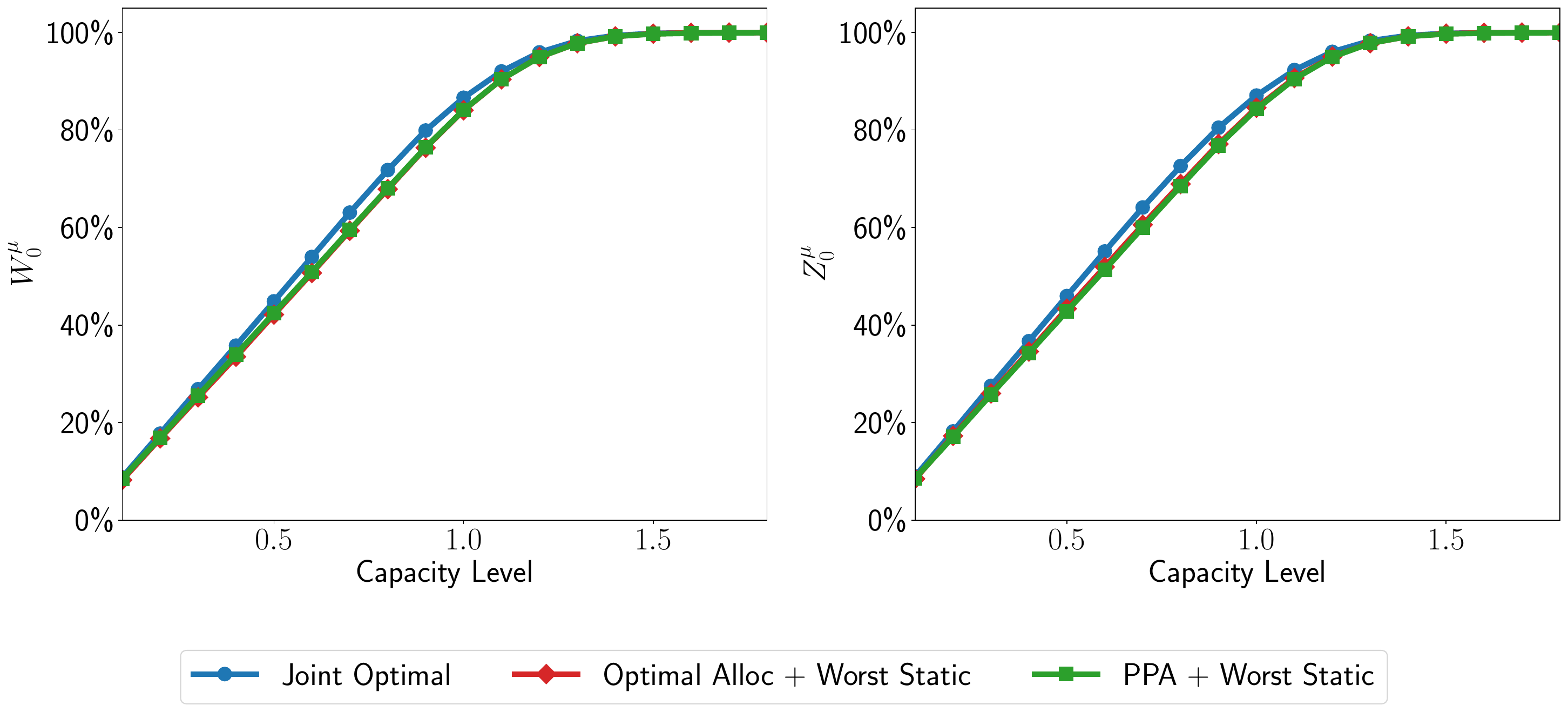}
        \caption{$W_0^{\Vmu}$ and $Z_0^{\Vmu}$ (y‑axis) vs the relative capacity level (x‑axis).}
        \label{subfig:obj_improvement}
    \end{subfigure}

    \caption{$\DfairEP$, $\DfairEA$, $W_0^{\Vmu}$ and $Z_0^{\Vmu}$ (y‑axis) plotted against the relative capacity level (x‑axis). We compare the jointly optimal allocation-routing policy, the optimal allocation combined with the worst static routing policy, and PPA under the worst static routing policy, evaluated under the \mexpost and \hqexpost objectives.}
    \label{fig:obj_align}
\end{figure}

\Cref{fig:obj_align} plots fairness metrics($\DfairEP$ and $\DfairEA$) and objective values $(Z_0^{\Vmu}$ and $W_0^{\Vmu}$) against capacity for the joint optimal policies under both objectives, averaged over all heterogeneous systems for a total of $6{,}624$ instances. We first observe in \Cref{subfig:obj_improvement} that both objective values increase monotonically as capacity increases from $0.1$ to $0.6$. However, over the same range, \Cref{subfig:fair_improvement} shows that fairness deteriorates: $\DfairEP$ increases to $25\%$ and $\DfairEA$ to $7\%$, with the worst performance occurring at $c=0.6$. Thus, improvements in the surrogate objectives do not necessarily translate into improved fairness metrics.

Despite this misalignment, optimizing these objectives remains beneficial relative to heuristic approaches. To highlight the impact of jointly optimizing allocation and routing, beyond the allocation-only focus common in prior literature, we compare the jointly optimal policy with the PPA allocation evaluated under the worst static routing policy, a conservative baseline to reflect settings in which routing decisions are not optimized. Under this comparison, improvements in objective values are modest (at most $2\%$, see \Cref{subfig:obj_improvement}), while gains in fairness metrics are substantially larger: both $\DfairEP$ and $\DfairEA$ improve by up to $10\%$ (see \Cref{subfig:fair_improvement}). These gains are most pronounced at intermediate capacity levels, which coincide with regimes where max-min objectives yield the greatest disparity.
\begin{observation}
\label{obs:fairness_waste}
\textit{
As capacity increases from scarce levels to approximately half of the total mean demand, allocation disparities become more pronounced despite increasing objective values. Nevertheless, joint optimization of allocation and routing yields substantial improvements in fairness metrics relative to heuristic policies, with the largest gains occurring precisely in these challenging regimes.
}
\end{observation}

This finding suggests that practitioners should be cautious in interpreting max-min objectives as proxies for fairness. While optimizing these objectives improves outcomes relative to heuristics, it may require additional mechanisms to mitigate early-stage imbalances.

\smallskip
\paragraph{Value of Routing} To isolate the role of routing, we compare three policies: (i) the jointly optimal allocation and routing policy, (ii) the optimal allocation policy evaluated under the worst static route, and (iii) the PPA allocation policy evaluated under the worst static route.  \Cref{subfig:fair_improvement} shows their performance on $\DfairEP$ and $\DfairEA$ across capacity levels.

For ex-post fairness $\DfairEP$ (\Cref{subfig:fair_improvement}, left), the optimal allocation under the worst route performs comparably to PPA under \mexpost, and performs worse under \hqexpost, with a gap of up to $8\%$ at $c=0.6$. Thus, under a poorly chosen route, replacing PPA with the optimal allocation yields little improvement and can even degrade performance. In contrast, the jointly optimal policy substantially outperforms both worst-route benchmarks, with gains most pronounced in intermediate capacity regimes (up to $10\%$ at $c=0.5$). This indicates that improvements in $\DfairEP$ are primarily driven by routing rather than allocation along a fixed route.
A similar comparison holds for ex-ante fairness $\DfairEA$ (\Cref{subfig:fair_improvement}, right). Under the worst route, the optimal allocation closely tracks PPA and is even outperformed in intermediate regimes, where PPA achieves gains of up to $3\%$ (at $c=0.7$). By contrast, the jointly optimal policy delivers significant improvements over both benchmarks, except in the range $c \in [0.7, 0.8]$.

Taken together, these comparisons show that the fairness gains in \Cref{subfig:fair_improvement} are driven primarily by routing optimization. Optimizing allocations under a poorly selected route has limited impact, while jointly optimizing the route substantially reduces both ex-post and ex-ante fairness disparities.

\begin{observation}
\label{obs:value_routing}
\textit{Routing optimization is the primary driver of fairness improvements in the capacity regimes where max-min objectives perform poorly as fairness metrics. Under poorly chosen routes, allocation optimization alone provides limited gains and can even worsen fairness relative to heuristic allocation policies.}
\end{observation}

\Cref{obs:value_routing} highlights that routing decisions are central to improving fairness metrics. In particular, the comparison between the jointly optimal policies and the worst-route benchmarks shows that the largest gains come from optimizing the route, whereas optimizing allocations along a poorly chosen route provides limited improvements and can even exacerbate unfairness relative to PPA.

\subsection{Objective Selection}
\label{subsec:sims_short_vs_long}

In \Cref{sec:model_preliminaries}, we argued that the key distinction between the \mexpost and \hqexpost objectives lies in their alignment with the ex-post and ex-ante fairness metrics. We now make this distinction explicit by comparing the performance of $\DfairEA$ and $\DfairEP$ under the two objectives. This leads to the following question: \emph{How should practitioners select a fairness objective based on the operational context of their application?}

In \Cref{fig:ea_vs_ep_frontier}, we plot the trade-off between $\DfairEA$ or $\DfairEP$ (y-axis) and $\Deff$ (x-axis) for the jointly optimal allocation and routing policies under the \mexpost and \hqexpost objectives. The reported curves are averaged over all $7{,}776$ instances. 
In high efficiency regimes (around $95\%$), \mexpost achieves better ex-post fairness, improving $\DfairEP$ by approximately $3\%$, while \hqexpost yields better ex-ante fairness, improving $\DfairEA$ by approximately $5\%$.

\begin{figure}[!t]
    \centering
    \includegraphics[width=0.8\textwidth]{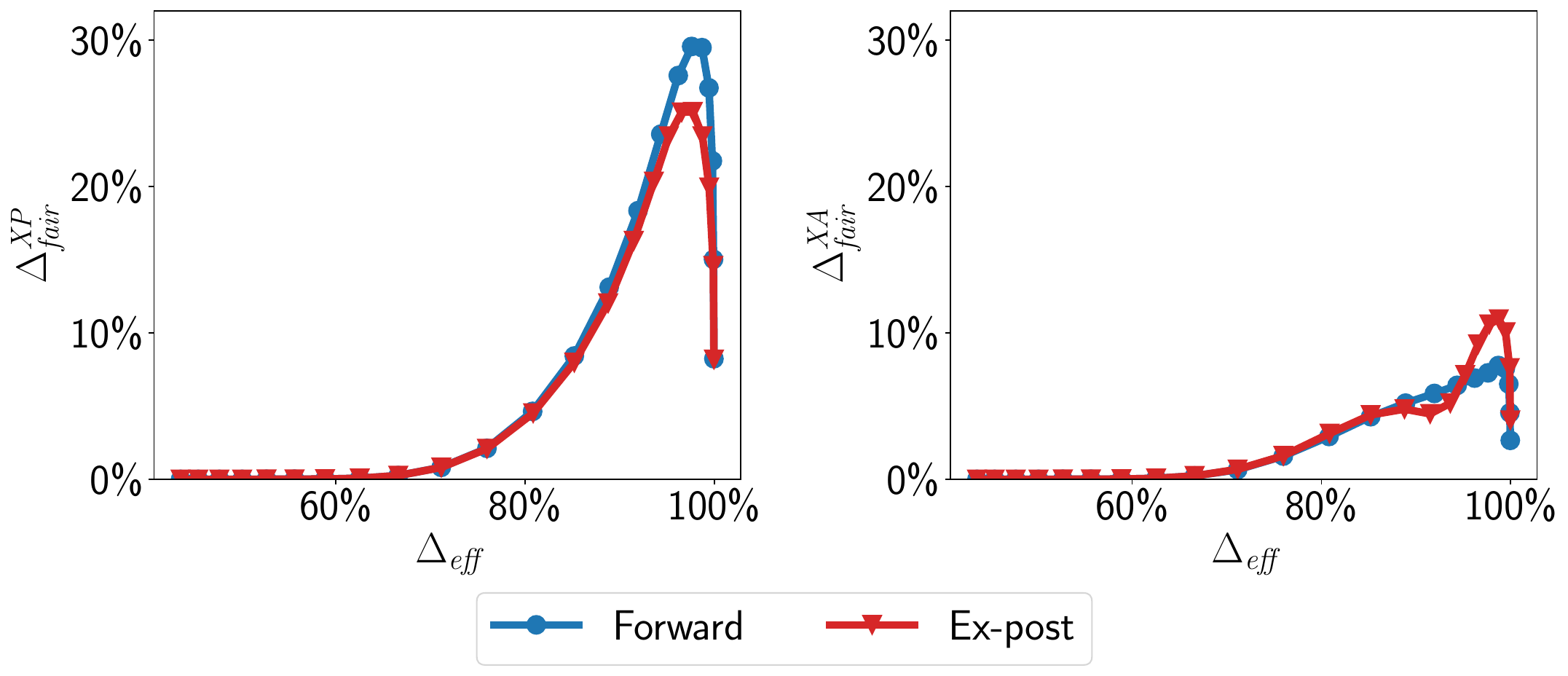}
    \caption{$\DfairEP$ and $\DfairEA$ (y-axis) versus $\Deff$ (x-axis) for the jointly optimal allocation and routing policy under the \mexpost and \hqexpost objectives across all systems.}
    \label{fig:ea_vs_ep_frontier}
\end{figure}

\begin{observation}
\label{obs:short_long}
\textit{The jointly optimal policy under the \mexpost objective substantially improves ex-post fairness $\DfairEP$ while maintaining a similar level of efficiency as \hqexpost, whereas the opposite pattern holds for ex-ante fairness $\DfairEA$.}
\end{observation}

\Cref{obs:short_long} highlights a clear operational trade-off between short-term and long-term fairness. Rather than identifying a universally superior objective, these results suggest that practitioners should select the objective that aligns with their fairness priorities.
In short-term or one-shot settings, such as emergency resource allocation or disaster response, where realized outcomes are paramount and there is limited opportunity for future correction, the \mexpost objective is more appropriate. In contrast, in repeated or long-term settings, such as food bank operations, where fairness is assessed cumulatively and temporal imbalances can be corrected over time, the \hqexpost objective provides better alignment with long-term fairness.

\begin{figure}[!t]
    \centering
    \includegraphics[width=0.8\textwidth]{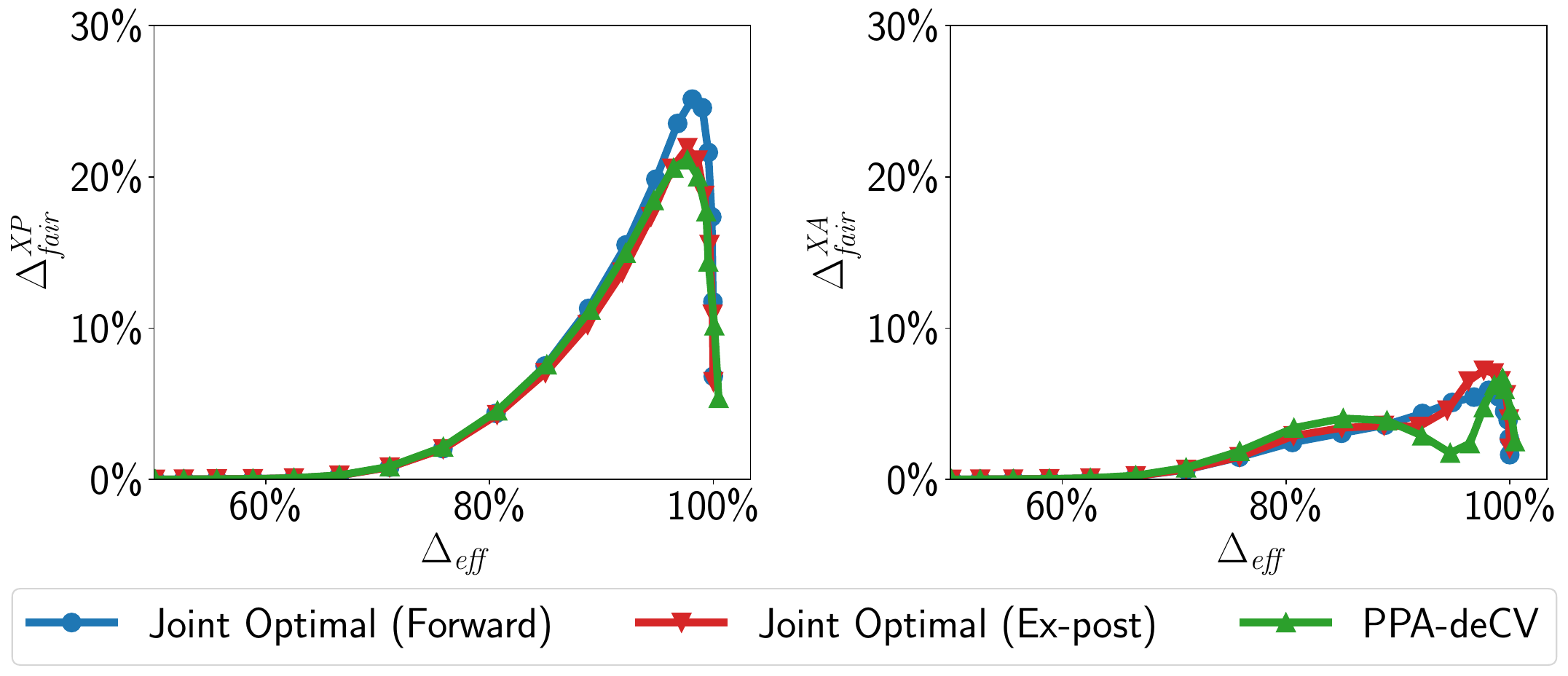}
    \caption{$\DfairEP$ and $\DfairEA$ (y-axis) versus $\Deff$ (x-axis), comparing the jointly optimal allocation and routing policy with PPA-deCV under the \mexpost and \hqexpost objectives in heterogeneous systems.}
    \label{fig:heuristics}
\end{figure}

\subsection{PPA-deCV Heuristic} \label{subsec:sims_heuristic}
Computing the jointly optimal allocation and routing policy is computationally intensive and becomes impractical in larger systems. This motivates the design of simple and implementable heuristics. In this section, we propose \textbf{PPA-deCV} and address the question: \emph{How well can our simple heuristic, PPA-deCV, approximate the performance of the jointly optimal solution?}

We formally introduce our allocation and routing heuristic, \textbf{PPA-deCV}. Prior to visiting any nodes, PPA-deCV computes each node’s coefficient of variation from its distribution and orders nodes in decreasing order of this quantity. Given this static route, resources are then allocated dynamically upon arrival according to the PPA policy in \Cref{eqn:PPA}. This heuristic builds directly on the structural results we show above and results in the literature.  PPA under a static route provides strong fairness guarantees as shown in \citet{manshadi2021fair}, aligns with the threshold structure identified in \Cref{sec:allocation}, and complements the decreasing CV routing structure characterized in \Cref{sec:sequencing}. Together, these features yield a simple and interpretable approximation of the optimal solution.

\Cref{fig:heuristics} compares the fairness-efficiency frontiers of PPA-deCV and the jointly optimal policy under the \mexpost\ and \hqexpost\ objectives. We plot $\DfairEP$ and $\DfairEA$ against efficiency $\Deff$, averaging across all heterogeneous systems except the same-CV configuration (for which decreasing CV is not defined), yielding a total of 4,608 instances. In \Cref{fig:heuristics} (left), we observe that PPA-deCV achieves nearly identical ex-post fairness to the jointly optimal policy under the \mexpost\ objective. It also improves $\DfairEP$ by approximately $5\%$ relative to the jointly optimal policy under the \hqexpost\ objective when efficiency is around $95\%$. This improvement occurs in the intermediate-capacity regime, where unfairness is most pronounced. In the same regime, \Cref{fig:heuristics} (right) shows that PPA-deCV also outperforms both jointly optimal policies in terms of ex-ante fairness, reducing $\DfairEA$ by up to $3\%.$

\begin{observation}
\textit{PPA-deCV achieves fairness-efficiency trade-offs comparable to the jointly optimal policy and can outperform it in regimes where surrogate objectives are misaligned with fairness metrics.}
\end{observation}

Overall, PPA-deCV offers a computationally efficient and practically implementable alternative that captures most of the benefits of joint optimization.

%% file: parts/conclusions.tex
\section{Conclusion} \label{sec:conclusion}
We study a dynamic fair sequential allocation and routing problem motivated by non-profit operations and analyze two max-min fairness objectives, \mexpost and \hqexpost, under demand uncertainty. For a fixed visitation order, we establish an \emph{equating property} of the optimal allocation policy, which yields a clear threshold structure. We further characterize routing policies and show that, under mild conditions, visiting locations in decreasing order of CV is optimal when the PPA policy is implemented. Combining these insights, we propose the PPA-deCV heuristic, which achieves fairness-efficiency trade-offs comparable to the jointly optimal allocation-routing policy while substantially reducing computational complexity.

Our numerical experiments highlight an important distinction between \emph{fairness objectives} and \emph{fairness metrics}. We find that optimizing max-min objectives yields only marginal objective value improvements, whereas optimizing routing decisions drives substantial gains in both ex-post and ex-ante fairness metrics. Moreover, we show that \mexpost achieves substantial gains in ex-post fairness with minimal loss in efficiency relative to \hqexpost, while \hqexpost outperforms in terms of ex-ante fairness, highlighting the importance of selecting the objective in line with operational priorities.
Several directions remain for future work. These include extending the routing results to more general settings, deriving performance bounds for the decreasing CV routing policy, and developing efficient methods for approximating the optimal allocation thresholds.

%% file: parts/appendix/dp_properties.tex
\section{Proof of Basic Properties in \Cref{subsec:DP_properties}} \label{appendix:dp_properties}
\Monotone*
\begin{proof}
$(a)$ We prove the result for both objectives by backward induction.

\noindent \textbf{Base Case:} When $n=N$, we have $W_N^{(\Vsig,~*)}\left( \Vu_N, ~\beta_{\min}^{N-1} \right)$ and $Z^{(\Vsig,~*)}_N \left( \Vu_N \right)$ are trivially non-decreasing in the capacity $c_N$ by \Cref{eqn:m_dp} and \Cref{eqn:hq_dp}.

    \noindent \textbf{Step Case:} $n+1 \rightarrow n$. For every $n \geq 1$, from \Cref{eqn:m_dp,eqn:hq_dp}, we have
    \begin{align*}
        W_n^{(\Vsig,~*)}\left(\Vu_n, ~\beta_{\min}^{n-1} \right) &= \max \limits_{0 \leq \pi_n \leq d_n \land c_n} \E_{d_{n+1} \sim \D_{\sigma_{n+1}}} \left[W_{n+1}^{(\Vsig,~*)}\left(c_n - \pi_n, ~\sigma_{n+1}, ~d_{n+1}, ~\itaS_n \backslash \{\sigma_{n+1}\}, ~\beta_{\min}^{n-1} \land \frac{\pi_n}{d_n} \right) \right],\\
        Z^{(\Vsig,~*)}_n \left( \Vu_n \right) &= \max \limits_{0 \leq \pi_n \leq d_n \land c_n} \mathbb{E}_{d_{n+1} \sim \mathcal{D}_{\sigma_{n+1}}} \left[ \frac{\pi_n}{d_n} \land Z^{(\Vsig,~*)}_{n+1} \left(c_n - \pi_n, ~\sigma_{n+1}, ~d_{n+1}, ~\itaS_n \backslash \{\sigma_{n+1}\} \right) \right].
    \end{align*}
    From the induction hypothesis, $W^{(\Vsig,~*)}_{n+1} (\cdot)$, $Z^{(\Vsig,~*)}_{n+1} \left(\cdot \right)$ is non-decreasing in the remaining capacity $c_{n+1} = c_n - \pi_n$, and hence is non-decreasing in the current capacity $c_n$. Taking the expectation and the maximum over the allocation $\pi_n$, $W^{(\Vsig,~*)}_{n} (\cdot)$, $Z^{(\Vsig,~*)}_{n} \left(\cdot \right)$ remain non-decreasing in $c_n$.

$(b)$ This directly follows from \citet[Proposition 7]{ma2022fairness}.

$(c)$ When $n=N$, we have $W_N^{(\Vsig,~*)}\left( \Vu_N, ~\beta_{\min}^{N-1} \right)$ and $Z^{(\Vsig,~*)}_N \left( \Vu_N \right)$ are trivially non-increasing in the demand $d_N$ by \Cref{eqn:m_dp} and \Cref{eqn:hq_dp}.

For any stage $1 \leq n < N$, let $d_{n,1} > d_{n,2}$, we have 
\begin{align*}
    &\quad W_n^{(\Vsig,~*)}\left(c_n,~\sigma_n,~d_{n,1},~\itaS_n, ~\beta_{\min}^{n-1} \right)\\
    &= \max \limits_{0 \leq \pi_n \leq d_{n,1} \land c_n} \E_{d_{n+1} \sim \D_{\sigma_{n+1}}} \left[W_{n+1}^{(\Vsig,~*)}\left(c_n - \pi_n, ~\sigma_{n+1}, ~d_{n+1}, ~\itaS_n \backslash \{\sigma_{n+1}\}, ~\beta_{\min}^{n-1} \land \frac{\pi_n}{d_{n,1}} \right) \right],\\
    & \overset{(i)}{=} \E_{d_{n+1} \sim \D_{\sigma_{n+1}}} \left[W_{n+1}^{(\Vsig,~*)}\left(c_n - \pi_n^{*,1}, ~\sigma_{n+1}, ~d_{n+1}, ~\itaS_n \backslash \{\sigma_{n+1}\}, ~\beta_{\min}^{n-1} \land \frac{\pi_n^{*,1}}{d_{n,1}} \right) \right],\\
    & = \E_{d_{n+1} \sim \D_{\sigma_{n+1}}} \left[W_{n+1}^{(\Vsig,~*)}\left(c_n - \pi_n^{*,1}, ~\sigma_{n+1}, ~d_{n+1}, ~\itaS_n \backslash \{\sigma_{n+1}\}, ~\beta_{\min}^{n-1} \land \frac{\pi_n^{*,1} \cdot \frac{d_{n,2}}{d_{n,1}}}{d_{n,2}} \right) \right],\\
    & \overset{(ii)}{\leq} \E_{d_{n+1} \sim \D_{\sigma_{n+1}}} \left[W_{n+1}^{(\Vsig,~*)}\left(c_n - \pi_n^{2}, ~\sigma_{n+1}, ~d_{n+1}, ~\itaS_n \backslash \{\sigma_{n+1}\}, ~\beta_{\min}^{n-1} \land \frac{\pi_n^{2}}{d_{n,2}} \right) \right],\\
    &\overset{(iii)}{\leq} \max \limits_{0 \leq \pi_n \leq d_{n,2} \land c_n} \E_{d_{n+1} \sim \D_{\sigma_{n+1}}} \left[W_{n+1}^{(\Vsig,~*)}\left(c_n - \pi_n, ~\sigma_{n+1}, ~d_{n+1}, ~\itaS_n \backslash \{\sigma_{n+1}\}, ~\beta_{\min}^{n-1} \land \frac{\pi_n}{d_{n,2}} \right) \right],\\
    &=W_n^{(\Vsig,~*)}\left(c_n,~\sigma_n,~d_{n,2},~\itaS_n, ~\beta_{\min}^{n-1} \right).
\end{align*}
Equality $(i)$ is from denoting the optimal allocation by $\pi_n^{*,1}$. Inequality $(ii)$ is by setting $\pi_n^2 = \pi_n^{*,1} \cdot \dfrac{d_{n,2}}{d_{n,1}}$, and since $\pi_n^2 < \pi_n^{*,1}$, we apply the monotonicity of $W_{n+1}^{(\Vsig,~*)}(\cdot)$ in $c_{n+1}$ (\Cref{lemma:monotone}). Inequality $(iii)$ is because $\pi_n^2 \leq d_{n,2}$ and thus feasible.

Similarly, for \hqexpost, we have
\begin{align*}
    &\quad Z^{(\Vsig,~*)}_n \left(c_n,~\sigma_n,~d_{n,1},~\itaS_n \right) \\
    &= \max \limits_{0 \leq \pi_n \leq d_{n,1} \land c_n} \mathbb{E}_{d_{n+1} \sim \mathcal{D}_{\sigma_{n+1}}} \left[ \frac{\pi_n}{d_{n,1}} \land Z^{(\Vsig,~*)}_{n+1} \left(c_n - \pi_n, ~\sigma_{n+1}, ~d_{n+1}, ~\itaS_n \backslash \{\sigma_{n+1}\} \right) \right],\\
    &= \mathbb{E}_{d_{n+1} \sim \mathcal{D}_{\sigma_{n+1}}} \left[ \frac{\pi_n^{*,1}}{d_{n,1}} \land Z^{(\Vsig,~*)}_{n+1} \left(c_n - \pi_n^{*,1}, ~\sigma_{n+1}, ~d_{n+1}, ~\itaS_n \backslash \{\sigma_{n+1}\} \right) \right],\\
    &= \mathbb{E}_{d_{n+1} \sim \mathcal{D}_{\sigma_{n+1}}} \left[ \frac{\pi_n^{*,1} \cdot \frac{d_{n,2}}{d_{n,1}}}{d_{n,2}} \land Z^{(\Vsig,~*)}_{n+1} \left(c_n - \pi_n^{*,1}, ~\sigma_{n+1}, ~d_{n+1}, ~\itaS_n \backslash \{\sigma_{n+1}\} \right) \right],\\
    &= \mathbb{E}_{d_{n+1} \sim \mathcal{D}_{\sigma_{n+1}}} \left[ \frac{\pi_n^{2}}{d_{n,2}} \land Z^{(\Vsig,~*)}_{n+1} \left(c_n - \pi_n^{2}, ~\sigma_{n+1}, ~d_{n+1}, ~\itaS_n \backslash \{\sigma_{n+1}\} \right) \right],\\
    &\leq \max \limits_{0 \leq \pi_n \leq d_{n,2} \land c_n} \mathbb{E}_{d_{n+1} \sim \mathcal{D}_{\sigma_{n+1}}} \left[ \frac{\pi_n}{d_{n,2}} \land Z^{(\Vsig,~*)}_{n+1} \left(c_n - \pi_n, ~\sigma_{n+1}, ~d_{n+1}, ~\itaS_n \backslash \{\sigma_{n+1}\} \right) \right],
\end{align*}
using the same definition of $\pi_n^{*,1}$ and $\pi_n^2$ and follow the same argument above.

$(d)$ We prove this by backward induction. 

\noindent \textbf{Base Case:} When $n=N$, we have $W_N^{(\Vsig,~*)}\left( \Vu_N, ~\beta_{\min}^{N-1} \right) = \beta_{\min}^{N-1} \land \frac{c_N}{d_N} \leq \beta_{\min}^{N-1}$.

\noindent \textbf{Step Case:} $n+1 \rightarrow n$. For every $n \geq 1$,
\begin{align*}
    W_n^{(\Vsig,~*)}\left(\Vu_n, ~\beta_{\min}^{n-1} \right) &\overset{(i)}{=} \max \limits_{0 \leq \pi_n \leq c_n \land d_n} \E_{d_{n+1} \sim \D_{\sigma_{n+1}}} \left[W_{n+1}^{(\Vsig,~*)}\left(c_n - \pi_n, ~\sigma_{n+1}, ~d_{n+1}, ~\itaS_n \backslash \{\sigma_{n+1}\}, ~\beta_{\min}^{n-1} \land \frac{\pi_n}{d_n} \right) \right],\\
    &\overset{(ii)}{\leq} \max \limits_{0 \leq \pi_n \leq c_n \land d_n} \E_{d_{n+1} \sim \D_{\sigma_{n+1}}} \left[\beta_{\min}^{n-1} \land \frac{\pi_n}{d_n} \right] \leq \beta_{\min}^{n-1},
\end{align*}
where equality $(i)$ is from \Cref{eqn:m_bellman} and inequality $(ii)$ is from the induction hypothesis.
\end{proof}

\Concave*
\begin{proof}
$(a)$ This follows directly from \citet[Proposition 7]{ma2022fairness}.

$(b)$ The result for the \mexpost objective follows from \citet[Proposition 7]{ma2022fairness}. For the \hqexpost objective, we show the result by backward induction. 

\noindent \textbf{Base Case}: When $n = N$, we have $Z^{(\Vsig, ~*)}_N(\Vu_N)$ is trivially concave in the capacity $c_N$.

\noindent \textbf{Step Case}: $n + 1 \rightarrow n$. From \Cref{eqn:hq_dp},
\begin{equation*}
    Z^{{(\Vsig,~*)}}_n(\Vu_n) = \max_{0 \leq \pi_n \leq d_n \wedge c_n} \E_{d_{n+1} \sim \D_{\sigma_{n+1}}} \left[\frac{\pi_n}{d_n} \land Z^{{(\Vsig,~*)}}_{n+1} \left( \Vu_{n+1} \right) \right],
\end{equation*}
where $\Vu_{n+1} = (c_n - \pi_n, ~\sigma_{n+1}, ~d_{n+1}, ~\itaS_{n} \setminus \{\sigma_{n+1}\})$. By the induction hypothesis that $Z^{(\Vsig,~*)}_{n+1} (\cdot)$ is concave in $c_{n+1}$ and the preservation of concavity under the linear combination $c_{n+1} = c_n - \pi_n$, $Z^{(\Vsig,~*)}_{n+1} \left( \Vu_{n+1} \right)$ is concave in $c_n$. Hence, as the minimum of two concave functions, $\dfrac{\pi_n}{d_n} \land Z^{(\Vsig,~*)}_{n+1} \left( \Vu_{n+1} \right)$ is also concave in $c_n$.  We then have that $Z^{{(\Vsig,~*)}}_n(\Vu_n)$ is concave in $c_n$ by taking the expectation and applying \citet[Section~3.2.5]{boyd2004convex} since $0 \leq \pi_n \leq d_n \wedge c_n$ is a convex set.
\end{proof}

\HqEquatingProperty*
\begin{proof}
    Suppose that for some stage $0 < n < N$ and a given current demand $d_n$, for any future demand $d_{n+1}$, an optimal allocation $\pi_n^{*, ~z}(\Vu_n)$ does not satisfy the equating property. There are three possibilities where \Cref{eqn:hq_equating_current_future} could fail to hold: 
    \begin{enumerate}[label=\textbf{Case \Roman*:}, align=left]
        \item For all demand $d_{n+1} \in [\lb_{n+1}, ~\ub_{n+1}]$, we have $\frac{\pi_n^{*, ~z}(\Vu_n)}{d_n} < Z^{(\Vsig, ~*)}_{n+1} (\Vu_{n+1}).$
        \item For all demand $d_{n+1} \in [\lb_{n+1}, ~\ub_{n+1}]$, we have $\frac{\pi_n^{*, ~z}(\Vu_n)}{d_n} > Z^{(\Vsig, ~*)}_{n+1} (\Vu_{n+1}).$ 
        \item Both $\frac{\pi_n^{*, ~z}(\Vu_n)}{d_n} < Z^{(\Vsig, ~*)}_{n+1} (\Vu_{n+1})$ and $\frac{\pi_n^{*, ~z}(\Vu_n)}{d_n} > Z^{(\Vsig, ~*)}_{n+1} (\Vu_{n+1})$ hold for some demand $d_{n+1}$.
    \end{enumerate}
    In the first two cases, we construct a new policy denoted by $\tilde{\Vpi}$ that has a higher objective value, leading to a contradiction. In the third case, we directly show that an optimal allocation policy satisfies \Cref{eqn:hq_equating_current_future}.
            
   \noindent \textbf{Case I:} We first consider the case when for all $d_{n+1}$, the current fill rate is smaller than the expected future minimum fill rate to go function, that is, $\frac{\pi_n^{*, ~z}(\Vu_n)}{d_n} < Z^{(\Vsig, ~*)}_{n+1} (\Vu_{n+1})$. We construct a new policy $\tilde{\pi}_n$ by setting the current fill rate $\frac{\tilde{\pi}_n}{d_n}$ equal to the expected future minimum fill rate to go when $d_{n+1} = \ub_{n+1}$, i.e.,
    \begin{equation} \label{new_policy_1}
        \frac{\tilde{\pi}_n}{d_n} = Z^{(\Vsig, ~*)}_{n+1} (c_n - \tilde{\pi}_n, ~\sigma_{n+1}, ~\ub_{n+1}, ~\itaS_{n+1}).
    \end{equation}
    Note that $\tilde{\pi}_n \leq c_n$ trivially since the right-hand side is non-positive when $\tilde{\pi}_n > c_n$. To see that such a $\tilde{\pi}_n \leq d_n$ exists, let
    $$l_n(\pi_n; ~d_{n+1}) \triangleq Z^{(\Vsig, ~*)}_{n+1} (c_n - \pi_n, ~\sigma_{n+1}, ~d_{n+1}, ~\itaS_{n+1})d_n - \pi_n.$$
    The sign of $l_n(\pi_n; ~d_{n+1})$ indicates whether the current fill rate or the expected future minimum fill rate to go function is smaller at $d_{n+1}$. 
    Hence, \Cref{new_policy_1} is equivalent to equating the current fill rate and the expected future minimum fill rate to go at $d_{n+1}^{\max}$, i.e., $l_n(\tilde{\pi}_n; ~d_{n+1}^{\max}) = 0$.
    Clearly, when we allocate nothing, we have $l_n(0; ~d_{n+1}^{\max}) > 0$, and when we allocate the demand, we have $l_n(d_n; ~d_{n+1}^{\max}) \leq 0$ (due to the fact that $0 < Z^{(\Vsig, ~*)}_{n+1}(\Vu_n) \leq 1$ when $c_{n+1} > 0$). Hence, by continuity ($Z^{(\Vsig, ~*)}_{n+1}(\Vu_{n+1})$ being the maximum or minimum of continuous functions) and the intermediate value theorem (IVT), there exists an allocation $0 < \tilde{\pi}_n \leq d_n$ such that $l_n(\tilde{\pi}_n; ~d_{n+1}^{\max}) = 0$.
    If we define $\tilde{\Vu}_{n+1} = (c_n - \tilde{\pi}_n, ~\sigma_{n+1}, ~d_{n+1}, ~\itaS_{n+1})$, then
    \begin{align*}
        Z^{(\Vsig, ~\tilde{\pi}_n)}_{n} (\Vu_{n}) & \overset{}{=} \mathbb{E}_{d_{n+1} \sim \mathcal{D}_{\sigma_{n+1}} } \left[ \frac{\tilde{\pi}_n}{d_{n}}\land Z^{(\Vsig, ~*)}_{n+1} \left(\tilde{\Vu}_{n+1} \right)\right] \overset{(i)}{=} \frac{\tilde{\pi}_n}{d_{n}},\\
        &\overset{(ii)}{>} \frac{\pi^{*, ~z}_n}{d_{n}} \overset{(iii) }{=} \mathbb{E}_{d_{n+1} \sim \mathcal{D}_{\sigma_{n+1}} } \left[ \frac{\pi^{*, ~z}_n}{d_{n}}\land Z^{(\Vsig, ~*)}_{n+1} \left(\Vu_{n+1} \right)\right] \overset{ }{=} Z^{(\Vsig, ~*)}_{n} (\Vu_{n}).
    \end{align*}
    Equality $(i)$ is due to \Cref{new_policy_1} and the monotonicity of $Z^{(\Vsig, ~*)}_{n+1} \left(\tilde{\Vu}_{n+1} \right)$ in demand $d_{n+1}$ in \Cref{lemma:monotone}. Inequality $(ii)$ is from the assumption that $\frac{\pi_n^{*, ~z}(\Vu_n)}{d_n} < Z^{(\Vsig, ~*)}_{n+1} \left(c_n - \pi_n, ~\sigma_{n+1}, ~d_{n+1}^{\max}, ~\itaS_{n+1} \right)$, i.e., $l_n(\pi_n^{*,~z}; ~d_{n+1}^{\max}) > 0$. Furthermore, for any given future demand $d_{n+1}$, since $Z^{(\Vsig, ~*)}_{n+1} (c_n - \pi_n, ~\sigma_{n+1}, ~d_{n+1}, ~\itaS_{n+1})$ is non-decreasing in the capacity $c_{n+1} = c_n - \pi_n$ from \Cref{lemma:monotone}, $l_n(\pi_n; ~d_{n+1})$ is decreasing in $\pi_n$. Combined with the fact that $l(\tilde{\pi}_n; ~d_{n+1}^{\max}) = 0$, we have that $\tilde{\pi}_n > \pi_n^{*,~z}$. Equality $(iii)$ directly follows from the assumption.

  \noindent \textbf{Case II:} We next consider the case when for all $d_{n+1}$, the current fill rate is greater than the expected future fill rate to go function, that is, $\frac{\pi_n^{*, ~z}(\Vu_n)}{d_n} > Z^{(\Vsig, ~*)}_{n+1} (\Vu_{n+1})$. We construct a new allocation $\tilde{\pi}_n$ by setting the current fill rate $\frac{\tilde{\pi}_n}{d_n}$ equal to the expected future minimum fill rate to go with $d_{n+1} = \lb_{n+1}$:
    \begin{equation} \label{new_policy_2}
        \frac{\tilde{\pi}_n}{d_n} = Z^{(\Vsig, ~*)}_{n+1} (c_n - \tilde{\pi}_n, ~\sigma_{n+1}, ~\lb_{n+1}, ~\itaS_{n+1}).
    \end{equation}
    Following the same logic as in \textbf{Case I}, such a $\tilde{\pi}_n$ exists by applying the IVT to $l_n(\pi_n; ~d_{n+1}^{\min})$. Thus, 
     \begin{align*}
        Z^{(\Vsig, ~\tilde{\pi}_n)}_{n} (\Vu_{n}) &\overset{ }{=} \mathbb{E}_{d_{n+1} \sim \mathcal{D}_{\sigma_{n+1}} } \left[ \frac{\tilde{\pi}_n}{d_{n}}\land Z^{(\Vsig, ~*)}_{n+1} \left(\tilde{\Vu}_{n+1} \right)\right] \overset{(i)}{=} \mathbb{E}_{d_{n+1} \sim \mathcal{D}_{\sigma_{n+1}} } \left[Z^{(\Vsig, ~*)}_{n+1} \left(\tilde{\Vu}_{n+1} \right)\right],\\
        &\overset{(ii)}{>} \mathbb{E}_{d_{n+1} \sim \mathcal{D}_{\sigma_{n+1}} } \left[Z^{(\Vsig, ~*)}_{n+1} \left(\Vu_{n+1} \right)\right] \overset{(iii)}{=} \mathbb{E}_{d_{n+1} \sim \mathcal{D}_{\sigma_{n+1}} } \left[ \frac{\pi^{*, ~z}}{d_{n}}\land Z^{(\Vsig, ~*)}_{n+1} \left(\Vu_{n+1} \right)\right] \overset{ }{=} Z^{(\Vsig, ~*)}_{n} (\Vu_{n}),
    \end{align*}
    where equality $(i)$ is due to \Cref{new_policy_2} and the monotonicity of $Z^{(\Vsig, ~*)}_{n+1}(\tilde{\Vu}_{n+1})$ in the demand $d_{n+1}$ in \Cref{lemma:monotone}. Inequality $(ii)$ is from the assumption that $\frac{\pi_n^{*, ~z}(\Vu_n)}{d_n} > Z^{(\Vsig, ~*)}_{n+1} \left(c_n - \pi_n, ~\sigma_{n+1}, ~d_{n+1}^{\min}, ~\itaS_{n+1} \right)$, i.e., $l_n(\pi_n^{*,~z}; ~d_{n+1}^{\min}) < 0$. Since $l_n(\pi_n; ~d_{n+1}^{\min})$ is decreasing in $\pi_n$ and $l_n(\tilde{\pi}_n; ~d_{n+1}^{\min}) = 0$, we have $\tilde{\pi}_n < \pi^{*, ~z} \leq d_n$. By applying the monotonicity of $Z^{(\Vsig, ~*)}_{n+1}(\Vu_{n+1})$ in the capacity $c_{n+1} = c_n - \pi_n$ in \Cref{lemma:monotone}, we have $Z^{(\Vsig, ~*)}_{n+1}(\tilde{\Vu}_{n+1}) > Z^{(\Vsig, ~*)}_{n+1}(\Vu_{n+1})$ (if equal, $\tilde{\pi}_n$ is optimal). Equality $(iii)$ is from the assumption that for all $d_{n+1}$, the minimum is achieved by $Z^{(\Vsig, ~*)}_{n+1} \left(\Vu_{n+1} \right)$.
    
    Combining the arguments above, we find a policy $\tilde{\pi}_n(\Vu_n)$ with a strictly higher minimum fill rate, contradicting the optimality assumption of the original allocation $\pi_n^{*, ~z}(\Vu_n)$.

    \noindent \textbf{Case III:} Lastly, we consider the case where both directions hold. Note that under the allocation $\pi_n^{*,~z}$, $l_n(\pi_n^{*,~z}; ~d_{n+1})$ depends only on the future demand $d_{n+1}$. Since $Z_{n+1}^{(\Vsig, ~*)}(\Vu_{n+1})$ is continuous in $d_{n+1}$, $l_n(\pi_n^{*,~z}; ~d_{n+1})$ is continuous in $d_{n+1}$. By assumption, the function $l_n(\pi_n^{*,~z}; ~d_{n+1})$ attains both positive and negative values over certain demands in the support. Owing to the connected support assumption, we can apply the IVT to show there must exist a realization $d_{n+1}$ such that $l_n(\pi_n^{*,~z}; ~d_{n+1}) = 0$, that is:
    \[
    Z_{n+1}^{(\Vsig, ~*)}(\Vu_{n+1}) = \frac{\pi_n^{*,~z}}{d_n}.
    \]
    This establishes that the equating property is indeed satisfied, and highlights precisely where the connected support assumption is used.
\end{proof}

\MaEquatingProperty*
\begin{proof}
     Suppose that there exists a stage $n \in [N]$ such that for given state $\Vu_{n}$ and every sample path $\overset{\rightarrow}{d}_{n+1:N} = (d_{n+1}, \ldots, d_N)$, an optimal allocation $\pi_n^{*,~w}\left(\Vu_n, ~\hist^{n-1} \right)$ does not satisfy the equating property. For every sample path $\overset{\rightarrow}{d}_{n+1:N}$, define
     \begin{equation} \label{eqn:B_def}
         B_n (\Vu_{n}, ~\beta_{\min}^{n-1}, ~\pi_n ; ~\overset{\rightarrow}{d}_{n+1:N}) \triangleq \min \limits_{i = n+1, \dots, N}\left\{\frac{\pi_i^{*,~w}(\Vu_{i}, ~\beta_{\min}^{i-1})}{d_i} \right\}
     \end{equation}
     to be the future minimum fill rate over the sample path $\overset{\rightarrow}{d}_{n+1:N}$ under an optimal allocation $\Vpi^{*,~w}$ from stage $n+1$ till $N$. There are three possibilities where \Cref{eqn:m_equating_current_future} could fail to hold:
     \begin{enumerate}[label=\textbf{Case \Roman*:}, align=left]
        \item For all sample paths $\overset{\rightarrow}{d}_{n+1:N}$, we have 
        \begin{equation} \label{ineq_m1}
         \frac{\pi_{n}^{*,~w}(\Vu_{n}, ~\beta_{\min}^{n-1})}{d_{n}} < B_n (\Vu_{n}, ~\beta_{\min}^{n-1}, ~\pi_{n}^{*,~w}(\Vu_{n}, ~\beta_{\min}^{n-1}); ~\overset{\rightarrow}{d}_{n+1:N}).
        \end{equation}
        \item For all sample paths $\overset{\rightarrow}{d}_{n+1:N}$, we have 
        \begin{equation} \label{ineq_m2}
         \frac{\pi_{n}^{*,~w}(\Vu_{n}, ~\beta_{\min}^{n-1})}{d_{n}} > B_n (\Vu_{n}, ~\beta_{\min}^{n-1}, ~\pi_{n}^{*,~w}(\Vu_{n}, ~\beta_{\min}^{n-1}); ~\overset{\rightarrow}{d}_{n+1:N}).
        \end{equation}
        \item Both \Cref{ineq_m1} and \Cref{ineq_m2} hold for some sample path $\overset{\rightarrow}{d}_{n+1:N}$.
    \end{enumerate}
    In the first two cases, we construct a new policy denoted by $\tilde{\Vpi}$ that has a higher objective value, leading to a contradiction. In the third case, we directly show that an optimal allocation policy satisfies \Cref{eqn:m_equating_current_future}.

     We start off by showing that for any sample path $\vec{d}_{n+1:N}$ there exists a feasible allocation $\tilde{\pi}_n(\vec{d}_{n+1:N})$ which equates the current and future fill rates. Towards this, let
     \[
     l_n(\pi_n; ~\overset{\rightarrow}{d}_{n+1:N}) \triangleq B_n (\Vu_{n}, ~\beta_{\min}^{n-1}, ~\pi_n; ~\overset{\rightarrow}{d}_{n+1:N}) d_n - \pi_{n}.
     \]
     The sign of $l_n(\cdot)$ indicates whether the current fill rate or the future minimum fill rate under a given sample path $\overset{\rightarrow}{d}_{n+1:N}$ is smaller. The following lemma establishes that $l_n(\pi_n; ~\overset{\rightarrow}{d}_{n+1:N})$ is continuous in the allocation $\pi_n$ by Berge's maximum theorem.

     \begin{restatable}{lemma}{ContinuousDifference}
     \label{lemma:continuous_difference_function}
     The function $l_n(\pi_n; ~\overset{\rightarrow}{d}_{n+1:N})$ is continuous in the allocation $\pi_n$.
     \end{restatable} 
    Given the continuity of $l_n(\pi_n; ~\overset{\rightarrow}{d}_{n+1:N})$, we show that for any sequence of future demands $\vec{d}_{n+1:N}$, there exists a feasible allocation $\tilde{\pi}_n(\vec{d}_{n+1:N})$ such that $l_n(\tilde{\pi}_n(\vec{d}_{n+1:N}); ~\vec{d}_{n+1:N}) = 0$. Indeed, this follows because $l_n(0; ~\overset{\rightarrow}{d}_{n+1:N}) > 0$ and $l(d_n; ~\overset{\rightarrow}{d}_{n+1:N}) \leq 0$ (note that $0 < B_n(\cdot) \leq 1$ when $c_n > 0$). Since $l_n(\pi_n; ~\vec{d}_{n+1:N})$ is continuous in $\pi_n$, by applying the IVT, there exists $0 < \tilde{\pi}_n(\vec{d}_{n+1:N}) < d_n$ such that
    \begin{equation} \label{pi_sample_path}
        l_n(\tilde{\pi}_n(\vec{d}_{n+1:N}); ~\vec{d}_{n+1:N}) = 0.
    \end{equation}
    Next we show that 
    {$B_n (\Vu_{n}, ~\beta_{\min}^{n-1}, ~\pi_n; ~\overset{\rightarrow}{d}_{n+1:N})$ is non-increasing in $\pi_n$ and thus $l_n(\pi_n; ~\vec{d}_{n+1:N})$ is decreasing in $\pi_n$.} 
    \begin{restatable}{lemma}{DecreasingL} \label{lemma:decreasing_l_and_B}
        The future minimum fill rate function $B_n (\Vu_{n}, ~\beta_{\min}^{n-1}, ~\pi_n; ~\overset{\rightarrow}{d}_{n+1:N})$ is non-increasing in the allocation decision $\pi_n$. Consequently, the function $l_n(\pi_n; ~\overset{\rightarrow}{d}_{n+1:N})$ is decreasing in $\pi_n$. Moreover, whenever $B_n (\Vu_{n}, ~\beta_{\min}^{n-1}, ~\pi_n; ~\overset{\rightarrow}{d}_{n+1:N}) < 1$, the function is strictly decreasing in $\pi_n$.
    \end{restatable}
    Lastly, note that we can rewrite the \mexpost objective as an expectation of the minimum fill rate over the previously visited nodes $\hist^{n-1}$, the current fill rate $\dfrac{\pi_n}{d_n}$, and the future minimum fill rate $B_n(\cdot)$:
     \begin{equation} \label{mexpost_with_B}
         W_{n}^{(\Vsig, ~*)}\left(\Vu_{n}, ~\beta_{\min}^{n-1} \right) = \max \limits_{0 \leq \pi_{n} \leq c_{n} \land d_{n}} \E_{\vec{d}_{n+1:N}} \left[\beta_{\min}^{n-1} \land \frac{\pi_{n}}{d_{n}} \land B_n (\Vu_{n}, ~\beta_{\min}^{n-1}, ~\pi_n; ~\overset{\rightarrow}{d}_{n+1:N}) \right].
     \end{equation}
     The rest of the proof proceeds with cases.
    
    \noindent \textbf{Case I:} We first consider the case in which, for all sample paths, an optimal allocation $\pi_n^{*,~w}\left(\Vu_n,~\hist^{n-1} \right)$ (for the remainder of the proof, we simplify the notation to $\pi_n^{*,~w}$) yields the current fill rate that is lower than the future minimum fill rate, i.e., \Cref{ineq_m1} is true. Let $\tilde{\pi}_{n} = \min \limits_{\overset{\rightarrow}{d}_{n+1:N}} \tilde{\pi}_{n}(\vec{d}_{n+1:N})$ be the minimum of the zeros of $l_n$ in \Cref{pi_sample_path} (note that the minimum is achieved since the demand distributions have compact supports).  Then we have
     \begin{align*}
         W_n^{(\Vsig, ~*)}(\Vu_n, ~\beta_{\min}^{n-1}) &= \E_{d_{n+1} \sim \D_{\sigma_{n+1}}} \left[W_{n+1}^{(\Vsig, ~*)}\left(\Vu_{n+1}, ~\beta_{\min}^{n-1} \land \frac{\pi_n^{*,~w}}{d_n} \right) \right] \overset{(i)}{=} \beta_{\min}^{n-1} \land \frac{\pi_n^{*,~w}}{d_n} \\
         &\overset{(ii)}{<} \beta_{\min}^{n-1} \land \frac{\tilde{\pi}_n}{d_n} \overset{(iii)}{=} W_n^{(\Vsig, ~\tilde{\pi}_n)}(\Vu_n, ~\beta_{\min}^{n-1}).
     \end{align*}
    Equality $(i)$ is due to \Cref{mexpost_with_B} and the assumption that \Cref{ineq_m1} is true. Inequality $(ii)$ is from the fact that $l_n(\pi_n; ~\overset{\rightarrow}{d}_{n+1:N})$ is decreasing in $\pi_n$. Let $\vec{d}^{\prime}_{n+1:N} = \argmin \limits_{\overset{\rightarrow}{d}_{n+1:N}} \tilde{\pi}_{n}(\vec{d}_{n+1:N})$ be the sample path where $\tilde{\pi}_n$ equates the current fill rate and the future minimum fill rate. From the assumption, we know that $l_n(\pi_n^{*,~w}; ~\vec{d}^{\prime}_{n+1:N}) > 0$, combining with definition $l_n(\tilde{\pi}_n; ~\vec{d}^{\prime}_{n+1:N}) = 0$, we have $\tilde{\pi}_{n} > \pi_n^{*,~w}$. Equality $(iii)$ follows by applying the construction of $\tilde{\pi}_n = \min \limits_{\overset{\rightarrow}{d}_{n+1:N}} \tilde{\pi}_{n}(\vec{d}_{n+1:N})$ together with the monotonicity of $B_n(\cdot)$ (non-increasing in $\pi_n$) to \Cref{mexpost_with_B}. We know that for every sample path $$B_n (\Vu_{n}, ~\beta_{\min}^{n-1}, ~\tilde{\pi}_n; ~\overset{\rightarrow}{d}_{n+1:N}) \geq B_n (\Vu_{n}, ~\beta_{\min}^{n-1}, ~\tilde{\pi}_n(\vec{d}_{n+1:N}); ~\overset{\rightarrow}{d}_{n+1:N}) \overset{(i)}{=} \frac{\tilde{\pi}_n(\vec{d}_{n+1:N})}{d_n} \geq \frac{\tilde{\pi}_n}{d_n},$$ where $(i)$ is because $\tilde{\pi}_n(\vec{d}_{n+1:N})$ is defined as the zeros of $l_n$. Therefore, we can omit the future minimum fill rate in $W_n^{(\Vsig, ~\tilde{\pi}_n)}(\Vu_n, ~\beta_{\min}^{n-1})$.
     
     \noindent \textbf{Case II:} We next consider the case when for all sample paths, the current fill rate is greater than the future minimum fill rate, i.e., \Cref{ineq_m2} is true. Let $\tilde{\pi}_{n} = \max \limits_{\overset{\rightarrow}{d}_{n+1:N}} \tilde{\pi}_{n}(\vec{d}_{n+1:N})$ be the maximum of the zeros in \Cref{pi_sample_path} (again the maximum is achieved due to the compactness of the range of $\overset{\rightarrow}{d}_{n+1:N}$), we have
     \begin{align*}
         W_n^{(\Vsig, ~*)}(\Vu_n, ~\beta_{\min}^{N-1}) &= \E_{d_{n+1} \sim \D_{\sigma_{n+1}}} \left[W_{n+1}^{(\Vsig, ~*)}\left(\Vu_{n+1}, ~\beta_{\min}^{n-1} \land \frac{\pi_n^{*,~w}}{d_n} \right) \right] \\
         &\overset{(i)}{=} \E_{\vec{d}_{n+1:N}} \left[\beta_{\min}^{n-1} \land B_n (\Vu_{n}, ~\beta_{\min}^{n-1}, ~\pi_n^{*,~w}; ~\overset{\rightarrow}{d}_{n+1:N}) \right]\\
         &\overset{(ii)}{<} \E_{\vec{d}_{n+1:N}} \left[\beta_{\min}^{n-1} \land B_n (\Vu_{n}, ~\beta_{\min}^{n-1}, ~\tilde{\pi}_n; ~\overset{\rightarrow}{d}_{n+1:N}) \right]
         \overset{(iii)}{=} W_n^{(\Vsig, ~\tilde{\pi}_n)}(\Vu_n, ~\beta_{\min}^{N-1}),
     \end{align*}
     where equality $(i)$ is due to the assumption (\Cref{ineq_m2}) and \Cref{mexpost_with_B}. Equality $(ii)$ is from the construction of $\tilde{\pi}_{n}$ and the fact that $B_n(\cdot)$ is non-increasing in $\pi_n$. Let $\vec{d}^{\prime}_{n+1:N} = \argmax \limits_{\overset{\rightarrow}{d}_{n+1:N}} \tilde{\pi}_{n}(\vec{d}_{n+1:N})$ be the sample path where $\tilde{\pi}_n$ equates the current fill rate and the future minimum fill rate. From the assumption, we know that $l_n(\pi_n^{*,~w}; ~\vec{d}^{\prime}_{n+1:N}) < 0$, combining with definition $l_n(\tilde{\pi}_n; ~\vec{d}^{\prime}_{n+1:N}) = 0$ and the decreasing property of $l_n(\pi_n; ~\overset{\rightarrow}{d}_{n+1:N})$ in $\pi_n$, we have $\tilde{\pi}_{n} < \pi_n^{*,~w}$. Furthermore, we have $ B_n (\Vu_{n}, ~\beta_{\min}^{n-1}, ~\pi_n^{*,~w}; ~\overset{\rightarrow}{d}_{n+1:N}) < \frac{\pi_n^{*,~w}}{d_n} \leq 1$, and from \Cref{lemma:decreasing_l_and_B}, when $\tilde{\pi}_{n} < \pi_n^{*,~w}$, we have $B_n(\Vu_{n}, ~\beta_{\min}^{n-1}, ~\pi_n^{*,~w}; ~\overset{\rightarrow}{d}_{n+1:N}) < B_n(\Vu_{n}, ~\beta_{\min}^{n-1}, ~\tilde{\pi}_n; ~\overset{\rightarrow}{d}_{n+1:N})$. Equality $(iii)$ is because $\tilde{\pi}_{n} = \max \limits_{\overset{\rightarrow}{d}_{n+1:N}} \tilde{\pi}_{n}(\vec{d}_{n+1:N})$ and the fact that $B_n(\cdot)$ is non-increasing in $\pi_n$, we know that for every sample path $\overset{\rightarrow}{d}_{n+1:N}$, we have
     \[
     \frac{\tilde{\pi}_n}{d_n} \geq \frac{\tilde{\pi}_n(\vec{d}_{n+1:N})}{d_n} \overset{(i)}{=} B_n (\Vu_{n}, ~\beta_{\min}^{n-1}, ~\tilde{\pi}_n(\vec{d}_{n+1:N}); ~\overset{\rightarrow}{d}_{n+1:N}) \geq B_n (\Vu_{n}, ~\beta_{\min}^{n-1}, ~\tilde{\pi}_n; ~\overset{\rightarrow}{d}_{n+1:N}),
     \]
     where $(i)$ is because $\tilde{\pi}_n(\vec{d}_{n+1:N})$ is defined as the zeros of $l_n$. With \Cref{mexpost_with_B}, we can omit the current fill rate of the objective function $W_n^{(\Vsig, ~\tilde{\pi}_n)}(\Vu_n, ~\beta_{\min}^{n-1})$.

     Combining the arguments above, we find a policy $\tilde{\Vpi}$ with a higher objective value, contradicting the optimality assumption of the original allocation $\Vpi^{*, ~w}$.
     
    \noindent \textbf{Case III:} Lastly, we consider the case when both \Cref{ineq_m1} and \Cref{ineq_m2} are true, meaning that $l_n(\pi_n^{*,~w}; ~\vec{d}_{n+1:N})$ attains both positive and negative values over certain demands in the support. We will show that we can find an optimal allocation satisfying \Cref{eqn:m_equating_current_future}. Let $n$ be the largest stage such that \Cref{eqn:m_equating_current_future} is not true. Note that under the allocation $\pi_n^{*,~w}$, $l_n(\pi_n^{*,~w}; ~\vec{d}_{n+1:N})$ depends only on the future demand $\vec{d}_{n+1:N}$. Since $B_n(\cdot)$ is continuous in $\vec{d}_{n+1:N}$, $l_n(\pi_n^{*,~w}; ~\vec{d}_{n+1:N})$ is continuous in $\vec{d}_{n+1:N}$, and the demand distribution has connected support, by the IVT there must exist a realization $\vec{d}_{n+1:N}$ such that $l_n(\pi_n^{*,~w}; ~\vec{d}_{n+1:N}) = 0$, that is:
    \[
    \frac{\pi_{n}^{*,~w}(\Vu_{n}, ~\beta_{\min}^{n-1})}{d_{n}} = B_n (\Vu_{n}, ~\beta_{\min}^{n-1}, ~\pi_{n}^{*,~w}(\Vu_{n}, ~\beta_{\min}^{n-1}); ~\overset{\rightarrow}{d}_{n+1:N}).
    \]

    Next we fix the sample path $\vec{d}_{n+1:N}$ in which $l_n(\pi_n^{*,~w}; ~\vec{d}_{n+1:N}) = 0$. Let $n^{\prime}$ be the stage to achieve the minimum fill rate over $n+1$ to $N$, i.e., $$B_n (\Vu_{n}, ~\beta_{\min}^{n-1}, ~\pi_{n}^{*,~w}(\Vu_{n}, ~\beta_{\min}^{n-1}); ~\overset{\rightarrow}{d}_{n+1:N}) = \frac{\pi_{n^{\prime}}^{*,~w}(\Vu_{n^{\prime}}, ~\beta_{\min}^{n^{\prime}-1})}{d_{n^{\prime}}}.$$
    Since stage $n$ is the largest stage to violate \Cref{eqn:m_equating_current_future}, we know that for any stage $i > n$ and any state $\Vu_i$, there exists a sample path $\overset{\rightarrow}{d}_{i:N}$ such that fill rates are equal. Because of that, we can find a sample path $\overset{\rightarrow}{d}^{\prime}_{n^{\prime}+1:N}$ such that \Cref{eqn:m_equating_current_future} holds. For stages between $n$ and $n^{\prime}$, if there exists $n < n'' < n^{\prime}$ such that $\frac{\pi_{n''}^{*,~w}(\Vu_{n''}, ~\beta_{\min}^{n''-1})}{d_{n''}} > \frac{\pi_{n^{\prime}}^{*,~w}(\Vu_{n^{\prime}}, ~\beta_{\min}^{n^{\prime}-1})}{d_{n^{\prime}}}$ (as the current fill rate is the minimum), 
    we adjust the allocation at $n''$ to $\tilde{\pi}_{n''} = \frac{\pi_{n}^{*,~w}(\Vu_{n}, ~\beta_{\min}^{n-1})}{d_{n}} \cdot d_{n''}$. 
    Note that the objective value remains the same since 
    \begin{align*}
        W_n^{(\Vsig, ~*)}(\Vu_n, ~\beta_{\min}^{n-1}) &= \E_{d_{n+1} \sim \D_{\sigma_{n+1}}} \left[W_{n+1}^{(\Vsig, ~*)}\left(\Vu_{n+1}, ~\beta_{\min}^{n-1} \land \frac{\pi_n^{*,~w}}{d_n} \right) \right]\\
        &\overset{(i)}{\leq} \E_{d_{n+1} \sim \D_{\sigma_{n+1}}} \left[\beta_{\min}^{n-1} \land \frac{\pi_n^{*,~w}}{d_n} \right] = \beta_{\min}^{n-1} \land \frac{\pi_n^{*,~w}}{d_n} \overset{(ii)}{=} \beta_{\min}^{n-1} \land \frac{\tilde{\pi}_{n''}}{d_{n''}},
    \end{align*}
    where inequality $(i)$ is from \Cref{lemma:monotone} and equality $(ii)$ is by construction. Hence, there is an optimal allocation satisfying \Cref{eqn:m_equating_current_future} on the sample path $\left(\overset{\rightarrow}{d}_{n+1:n^{\prime}}, ~\overset{\rightarrow}{d}^{\prime}_{n^{\prime}+1:N} \right)$.
\end{proof}

\Jensen*
\begin{proof}
    Let $\Vmu = \left(\Vsig, ~\Vpi \right)$.  We show via backwards induction that for any stage $n$, state $\Vu_n$, and the minimum fill rate over visited customers $\hist^{n-1}$ that:
    \begin{equation} \label{eqn:wz_ineq_n}
        W_n^{\Vmu}\left(\Vu_n, ~\beta_{\min}^{n-1} \right) \leq Z_n^{\Vmu}\left(\Vu_n\right).
    \end{equation}
    
    \noindent \textbf{Base Case:} $n = N$. This is clearly true from \Cref{eqn:m_bellman,eqn:hq_bellman}.
            
    \noindent \textbf{Step Case:} $n+1 \rightarrow n$. For any stage $n$, we have
    \begin{align*}
        W_n^{\Vmu}\left(\Vu_n, ~\beta_{\min}^{n-1} \right) &= \E_{d_{n+1} \sim \D_{\sigma_{n+1}}} \left[W_{n+1}^{\Vmu}\left(\Vu_{n+1}, ~\beta_{\min}^{n-1} \land \frac{\pi_n}{d_n} \right) \right],\\
        &\overset{(i)}{\leq} \E_{d_{n+1} \sim \D_{\sigma_{n+1}}} \left[\frac{\pi_n}{d_n} \land Z_{n+1}^{\Vmu}\left(\Vu_{n+1}\right)\right] = Z_n^{\Vmu}\left(\Vu_n\right),
    \end{align*}
    where $(i)$ is due to $W_{n+1}^{\Vmu}\left(\Vu_{n+1}, \hist^{n} \right) \leq \hist^n = \hist^{n-1} \land \frac{\pi_n}{d_n} \leq \frac{\pi_n}{d_n}$ by \Cref{lemma:monotone} and the induction hypothesis. 
    Since for all $n \geq 1$, \Cref{eqn:wz_ineq_n} is true, we have 
    $$W_0^{\Vmu} \left(c \right) = \E_{d_1 \sim \D_{\sigma_1}} \left[W_1^{\Vmu}\left(\Vu_1, ~1 \right) \right] \leq \E_{d_1 \sim \D_{\sigma_1}} \left[Z_1^{\Vmu}\left(\Vu_1\right) \right]= Z_0^{\Vmu} \left(c \right).$$
\end{proof}

\subsection{Auxiliary Lemmas} \label{appendix:equating_lemmas}

\ContinuousDifference*
\begin{proof}
     To prove that $l_n(\pi_n; ~\overset{\rightarrow}{d}_{n+1:N})$ is continuous in the allocation $\pi_n$, it suffices to show that $B_n(\cdot)$ is continuous in $\pi_n$. Since $B_n(\cdot)$ is the minimum of the fill rates from stage $n+1$ till $N$, we only need to show that for any $i > n$ , $\pi_{i}^{*,~w}(\Vu_{i}, ~\beta_{\min}^{i-1})$ is continuous in $\pi_n$.
     For simplicity, we assume that the optimal allocation is unique. When there are multiple optimal solutions, we extend the result to that $\pi_{i}^{*,~w}(\Vu_{i}, ~\beta_{\min}^{i-1})$ is upper hemicontinuous in $\pi_n$, which will not affect the proof of \Cref{lemma:m_equate_property}.
     
     We start by showing for any $i > n$ that $\pi_{i}^{*,~w}(\Vu_{i}, ~\beta_{\min}^{i-1})$ is continuous in the remaining capacity $c_i$ and the minimum fill rate over the previous nodes $\hist^{i-1}$. Since both $c_i$ and $\hist^{i-1}$ are continuous functions of $\pi_n$, an induction argument combined with this continuity implies that $\pi_{i}^{*,~w}(\Vu_{i}, ~\beta_{\min}^{i-1})$ is continuous in $\pi_n$.
    
     First, we prove that for any stage $i$, the feasible set $\left\{\pi_i: 0 \leq \pi_i \leq c_i \land d_i \right\}$ is a compact and continuous correspondence. Given $c_i$, the set is bounded and closed and thus compact. To show the continuity, let $f_i(c_i) = 0$ and $g_i(c_i) = c_i \land d_i$. Then the feasible set can be written as $\left\{\pi_i: f_i(c_i) \leq \pi_i \leq g_i(c_i) \right\}$, where both functions are continuous in $c_i$. Applying the lemma of continuous correspondence \citep[11.18(e)]{border1985fixed}, we know that the feasible set is continuous in $c_i$.
     Now note that
     \[
     \pi_{i}^{*,~w} (\Vu_{i}, ~\beta_{\min}^{i-1}) = \argmax \limits_{0 \leq \pi_{i} \leq c_{i} \land d_{i}} \E_{d_{i+1} \sim \D_{\sigma_{i+1}}}
     \left[W_{i+1}^{(\Vsig, ~*)}\left(\Vu_{i+1}, ~\beta_{\min}^{i-1} \land \frac{\pi_{i}}{d_{i}} \right) \right],
     \]
     where $W_{i+1}^{(\Vsig, ~*)}\left(\Vu_{i+1}, ~\beta_{\min}^{i-1} \land \frac{\pi_{i}}{d_{i}} \right)$ is continuous in both $c_{i}$ and $\hist^{i-1}$. Applying Berge's Maximum Theorem \citep[Theorem 12.1]{border1985fixed}, we know that the maximizer $\pi_{i}^{*,~w} (\Vu_{i}, ~\beta_{\min}^{i-1})$ is continuous in $c_{i}$ and $\hist^{i-1}$ if it is unique; otherwise, the theorem instead implies that the set of maximizers is upper hemicontinuous, and we take the supremum. 

     Next, we show that for any $i > n$, $\pi_{i}^{*,~w}(\Vu_{i}, ~\beta_{\min}^{i-1})$ is continuous in $\pi_n$ by induction.

     \noindent \textbf{Base Case}: $i = n+1$. Note that $\pi_{n+1}^{*,~w}(\Vu_{n+1}, ~\beta_{\min}^{n})$ is a function of $\pi_n$ through the remaining capacity $c_{n+1} = c_{n} - \pi_n$ and the updated minimum fill rate $\hist^{n} = \hist^{n-1} \land \frac{\pi_n}{d_n}$, thus, from the chain rule we know $\pi_{n+1}^{*,~w}(\Vu_{n+1}, ~\beta_{\min}^{n})$ is continuous in $\pi_n$. 

     \noindent \textbf{Step Case}: $i \rightarrow i+1$. Note that the remaining capacity $c_{i+1} = c_n - \pi_n - \sum \limits_{j = n+1}^{i} \pi_{j}^{*,~w}(\Vu_{j}, ~\beta_{\min}^{j-1})$ and the updated minimum fill rate $\hist^{i} = \min \limits_{j = n+1, \dots, i} \left\{\hist^{n-1}, \frac{\pi_n}{d_n}, \frac{\pi_{j}^{*,~w}(\Vu_{j}, ~\beta_{\min}^{j-1})}{d_j} \right\}$. From the induction hypothesis, for all $j \leq i$, $\pi_{j}^{*,~w}(\Vu_{j}, ~\beta_{\min}^{j-1})$ are continuous in $\pi_n$ and the fact that $\pi_{i+1}^{*,~w} (\Vu_{i+1}, ~\beta_{\min}^{i})$ is continuous in $c_{i+1}$ and $\hist^{i}$, we conclude that $\pi_{i+1}^{*,~w} (\Vu_{i+1}, ~\beta_{\min}^{i})$ is continuous in $\pi_n$.
 \end{proof}

\DecreasingL*
\begin{proof}
    Note that to show the future minimum fill rate \(B_n(\Vu_{n}, ~\beta_{\min}^{n-1}, ~\pi_n; ~\overset{\rightarrow}{d}_{n+1:N})\) is non-increasing in \(\pi_{n}\), it's equivalent to show that for any stage $i > n+1$, the optimal allocation \(\pi_{i}^{*, ~w}(\Vu_i, ~\hist^{i-1})\) is non-decreasing in the capacity $c_{n+1} = c_n - \pi_n$ and non-increasing in the prior minimum fill rate $\hist^{n} = \hist^{n-1} \land \dfrac{\pi_n}{d_n}$. 
    
    Towards this, we first show for any stage $i$, we have \(\pi_{i}^{*, ~w}(\Vu_i, ~\hist^{i-1})\) is non-decreasing in \(c_{i}\). Second, we show that for any stage $i < N$, the optimal remaining capacity \(c_{i+1}^{*, ~w}(\Vu_i,~\hist^{i-1}) \triangleq c_{i}-\pi_{i}^{*, ~w}(\Vu_i,~\hist^{i-1})\) is also non-decreasing in \(c_{i}\). Then by induction, when \(c_{n+1}\) decreases, for any stage $i>n$, we have \(\pi_{i}^{*, ~w}(\Vu_{i},~\hist^{i-1})\) is non-increasing. Finally, we show that for any stage $i$, we have \(\pi_{i}^{*, ~w}(\Vu_{i},~\hist^{i-1})\) is non-increasing in $\hist^{i-1}$ and then by induction, when $\hist^n$ increases, for any stage $i>n$, we have \(\pi_{i}^{*, ~w}(\Vu_{i},~\hist^{i-1})\) is non-increasing, and therefore, the minimum fill rate \(B_n(\Vu_{n}, ~\beta_{\min}^{n-1}, ~\pi_n; ~\overset{\rightarrow}{d}_{n+1:N})\) is also non-increasing. Consequently, \(B_n(\Vu_{n}, ~\beta_{\min}^{n-1}, ~\pi_n; ~\overset{\rightarrow}{d}_{n+1:N})\) is non-increasing in $\pi_n$ due to $c_{n+1} = c_n - \pi_n$ and $\hist^{n} = \hist^{n-1} \land \dfrac{\pi_n}{d_n}$. We note that similar results (monotonicity of $\pi_i^{*,~w}$ with respect to $c_i$ and $\hist^{i-1}$) appear in \cite{ma2022fairness}.     
    
    \noindent \textbf{Part I:} For any stage $i$, we prove that \(\pi_{i}^{*,~w}(\Vu_i,~\hist^{i-1})\) is non-decreasing in \(c_{i}\).
    By definition, 
    \begin{align*}
        \pi_{i}^{*,~w}(\Vu_i,~\hist^{i-1})&= \argmax \limits_{\scriptstyle 0 \leq \pi_i \leq c_i \land d_i} \mathbb{E}_{d_{i+1} \sim D_{\sigma_{i+1}}}\left[W_{i+1}^{(\Vsig, ~*)}\left(c_{i}-\pi_{i}, ~d_{i+1}, ~\hist^{i-1} \wedge \frac{\pi_{i}}{d_{i}}\right)\right],\\
        &\triangleq \argmax \limits_{\scriptstyle 0 \leq \pi_i \leq c_i \land d_i} h_{i}\left(\pi_{i}; ~c_{i}, ~d_{i}, ~\beta_{\min }^{i-1}\right).
    \end{align*}
    
    We show that \(h_{i}\left(\pi_{i} ; ~c_i, ~d_{i}, ~\beta_{\min }^{i-1}\right)\) is supermodular in \(\left(c_{i}, ~\pi_{i}\right)\). Note that it's equivalent to show \(\frac{\partial^{2} h_{i}}{\partial c_{i} \partial \pi_{i}}(\cdot) \geq 0\) by \citet[Corollary 2.6.1]{topkis1998supermodularity}. We show this as follows:
    \begin{align*}
    \frac{\partial^{2} h_{i}}{\partial c_{i} \partial \pi_{i}}(\cdot) & = \mathbb{E}_{d_{i+1} \sim \D_{\sigma_{i+1}}} \left[ \frac{\partial}{\partial c_{i}}\left(-\frac{\partial W_{i+1}^{(\Vsig, ~*)}}{\partial c_{i+1}} \left(c_{i}-\pi_{i}, ~d_{i+1}, ~\hist^{i-1} \wedge \frac{\pi_{i}}{d_{i}}\right) \right. \right.   \\
    &\left. \left. \quad + \frac{1}{d_i} \frac{\partial W_{i+1}^{(\Vsig, ~*)}}{\partial \hist^i}\left(c_{i}-\pi_{i}, ~d_{i+1}, ~\hist^{i-1} \wedge \frac{\pi_{i}}{d_{i}} \right) \mathbf{1}\{\pi_{i} \leq \beta_{\min }^{i-1} d_{i}\} \right) \right], \\
    & =\mathbb{E}_{d_{i+1} \sim \D_{\sigma_{i+1}}} \left[ \underbrace{-\frac{\partial^2 W_{i+1}^{(\Vsig, ~*)}}{\partial c_{i+1}^2}\bigl(c_{i}-\pi_{i}, ~d_{i+1}, ~\hist^{i-1} \wedge \frac{\pi_{i}}{d_{i}}\bigr)}_{(1)} \right. \\
    &\left. \quad + \underbrace{\frac{1}{d_i} \frac{\partial^2 W_{i+1}^{(\Vsig, ~*)}}{\partial c_{i+1} \partial \hist^i}\bigl(c_{i}-\pi_{i}, ~d_{i+1}, ~\hist^{i-1} \wedge \frac{\pi_{i}}{d_{i}}\bigr) \mathbf{1}\{\pi_{i} \leq \beta_{\min }^{i-1} d_{i}\}}_{(2)} \right] \geq 0.
    \end{align*}
    Term $(1)$ is non-negative by concavity of $W_{i+1}^{(\Vsig, ~*)}\left(c_{i}-\pi_{i}, ~d_{i+1}, ~\hist^{i-1} \wedge \frac{\pi_{i}}{d_{i}}\right)$ in \(c_{i+1}\) (\Cref{lemma:concave_in_pi}), and term $(2)$ is non-negative due to the supermodularity of $W_{i+1}^{(\Vsig, ~*)}\left(c_{i}-\pi_{i}, ~d_{i+1}, ~\hist^{i-1} \wedge \frac{\pi_{i}}{d_{i}}\right)$ in $(c_{i+1}, ~\hist^i)$ from \citet[Proposition 7]{ma2022fairness}. 
    Since the set $\{0 \leq \pi_i \leq c_i \land d_i\}$ is a sublattice on $(c_i,~\pi_i)$ and \(h_{i}\left(\pi_{i} ; ~c_i,~d_{i}, ~\beta_{\min }^{i-1}\right)\) is supermodular in \(\left(c_{i}, ~\pi_{i}\right)\), we can apply \citet[Theorem 2.8.2]{topkis1998supermodularity} to show that $\pi_i^{*,~w}(\cdot)$ is non-decreasing in $c_i$. If there exists multiple optimal solutions, we apply \citet[Theorem 2.8.3]{topkis1998supermodularity} to show that the supremum of the optimal solutions is non-decreasing.

    \noindent \textbf{Part II:} Similar to \textbf{Part I}, we will show that for any stage $i < N$, the optimal remaining capacity \(c_{i+1}^{*,~w}(\Vu_i,~\hist^{i-1}) \triangleq c_{i}-\pi_{i}^{*,~w}(\Vu_i,~\hist^{i-1})\) is non-decreasing in \(c_{i}\) using supermodularity. First, we transform the optimization problem such that
    \begin{equation} \label{eqn:remaining_capacity_opt}
    \begin{aligned}
    c_{i+1}^{*, ~w}(\Vu_i,~\hist^{i-1}) & = \argmax \limits_{c_i - d_i \vee 0 \leq c_{i+1} \leq c_i} \mathbb{E}_{d_{i+1} \sim \D_{\sigma_{i+1}}}\left[W_{i+1}^{(\Vsig, *)}\left(c_{i+1}, ~d_{i+1}, ~\frac{c_{i}-c_{i+1}}{d_{i}} \wedge \hist^{i-1} \right)\right], \\
    & \triangleq \argmax \limits_{c_i - d_i \vee 0 \leq c_{i+1} \leq c_i} f_{i}\left(c_{i+1}; ~c_{i}, ~d_{i},~\hist^{i-1}\right).
    \end{aligned}
    \end{equation}
    
    We show that \(\frac{\partial^{2} f_{i}}{\partial c_{i} \partial c_{i+1}}\left(c_{i+1}; ~c_{i}, ~d_{i},~\hist^{i-1}\right) \geq 0\) and thus, \(f_{i}\left(\cdot \right)\) is supermodular in \(\left(c_{i+1}, c_{i}\right)\) by \citet[Corollary 2.6.1]{topkis1998supermodularity}. Moreover,
    \begin{align*}
    \frac{\partial^{2} f_{i}}{\partial c_{i} \partial c_{i+1}}\left(\cdot\right) & = \mathbb{E}_{d_{i+1} \sim \D_{\sigma_{i+1}}} \Bigl[ \frac{\partial}{\partial c_i} \Bigl( \frac{\partial W_{i+1}^{(\Vsig, *)}}{\partial c_{i+1}} \bigl(c_{i+1}, ~d_{i+1}, ~\frac{c_{i}-c_{i+1}}{d_{i}} \wedge \hist^{i-1} \bigr) \\
    &\quad - \frac{1}{d_i}\frac{\partial W_{i+1}^{(\Vsig, *)}}{\partial \hist^{i}} \bigl(c_{i+1}, ~d_{i+1}, ~\frac{c_{i}-c_{i+1}}{d_{i}} \wedge \hist^{i-1} \bigr) \mathbf{1}\{c_{i+1} \geq c_i - \hist^{i-1}d_i\} \Bigr) \Bigr], \\
    & = \mathbb{E}_{d_{i+1} \sim \D_{\sigma_{i+1}}} \Bigl[ \underbrace{\frac{1}{d_i}\frac{\partial^2 W_{i+1}^{(\Vsig, *)}}{\partial \hist^i \partial c_{i+1}} \bigl(c_{i+1}, ~d_{i+1}, ~\frac{c_{i}-c_{i+1}}{d_{i}} \wedge \hist^{i-1} \bigr)}_{(1)} \\
    &\quad  \underbrace{- \frac{1}{d_i^2}\frac{\partial^2 W_{i+1}^{(\Vsig, *)}}{\partial \left(\hist^{i}\right)^2} \bigl(c_{i+1}, ~d_{i+1}, ~\frac{c_{i}-c_{i+1}}{d_{i}} \wedge \hist^{i-1} \bigr) \mathbf{1}\{c_{i+1} \geq c_i - \hist^{i-1}d_i \}}_{(2)} \Bigr] \geq 0.
    \end{align*}
    Term $(1)$ is non-negative since $W_{i+1}^{(\Vsig, ~*)} (\cdot)$ is supermodular in $(c_{i+1}, ~\hist^i)$ from \citet[Proposition 7]{ma2022fairness}, and term $(2)$ is also non-negative due to $W_{i+1}^{(\Vsig, ~*)}(\cdot)$'s concavity in $\hist^{i}$ from \Cref{lemma:concave_in_pi}. 
    Again, the feasible set $\{c_i - d_i \vee 0 \leq c_{i+1} \leq c_i\}$ is a sublattice of $(c_{i+1, ~c_i})$ and $f_{i}\left(c_{i+1}; ~c_{i}, ~d_{i},~\hist^{i-1}\right)$ is supermodular in $(c_{i+1},~c_i)$, applying either \citet[Theorem 2.8.2]{topkis1998supermodularity} (when the optimal solution is unique) or \citet[Theorem 2.8.3]{topkis1998supermodularity} (when there exist multiple optimal solutions) will prove that $c_{i+1}^{*,~w}(\cdot)$ is non-decreasing.

    \noindent \textbf{Part III:} For any stage $i$, we prove that \(\pi_{i}^{*,~w}(\Vu_i,~\hist^{i-1})\) is non-increasing in \(\hist^{i-1}\). We show that \(h_{i}\left(\pi_{i} ; ~c_i, ~d_{i}, ~\beta_{\min }^{i-1}\right)\) is submodular in \(\left(\hist^{i-1}, ~\pi_{i}\right)\). Note that it's equivalent to show \(\frac{\partial^{2} h_{i}}{\partial \hist^{i-1} \partial \pi_{i}}(\cdot) \leq 0\) by \citet[Corollary 2.6.1]{topkis1998supermodularity}. We show this as follows:
    \begin{align*}
    \frac{\partial^{2} h_{i}}{\partial \hist^{i-1} \partial \pi_{i}}(\cdot) & = \mathbb{E}_{d_{i+1} \sim \D_{\sigma_{i+1}}} \left[ \frac{\partial}{\partial \hist^{i-1}}\left(-\frac{\partial W_{i+1}^{(\Vsig, ~*)}}{\partial c_{i+1}} \left(c_{i}-\pi_{i}, ~d_{i+1}, ~\hist^{i-1} \wedge \frac{\pi_{i}}{d_{i}}\right) \right. \right.   \\
    &\left. \left. \quad + \frac{1}{d_i} \frac{\partial W_{i+1}^{(\Vsig, ~*)}}{\partial \hist^i}\left(c_{i}-\pi_{i}, ~d_{i+1}, ~\hist^{i-1} \wedge \frac{\pi_{i}}{d_{i}} \right) \mathbf{1}\{\pi_{i} \leq \beta_{\min }^{i-1} d_{i}\} \right) \right], \\
    & =\mathbb{E}_{d_{i+1} \sim \D_{\sigma_{i+1}}} \left[ -\frac{\partial^2 W_{i+1}^{(\Vsig, ~*)}}{\partial \hist^i \partial c_{i+1}}\bigl(c_{i}-\pi_{i}, ~d_{i+1}, ~\hist^{i-1} \wedge \frac{\pi_{i}}{d_{i}}\bigr) \mathbf{1}\{\pi_{i} > \beta_{\min }^{i-1} d_{i}\} \right. \\
    &\left. \quad + \frac{1}{d_i} \frac{\partial^2 W_{i+1}^{(\Vsig, ~*)}}{\partial (\hist^i)^2}\bigl(c_{i}-\pi_{i}, ~d_{i+1}, ~\hist^{i-1} \wedge \frac{\pi_{i}}{d_{i}}\bigr) \mathbf{1}\{\pi_{i} \leq \beta_{\min }^{i-1} d_{i}\} \mathbf{1}\{\pi_{i} > \beta_{\min }^{i-1} d_{i}\} \right],\\
    & = - \mathbb{E}_{d_{i+1} \sim \D_{\sigma_{i+1}}} \left[\frac{\partial^2 W_{i+1}^{(\Vsig, ~*)}}{\partial \hist^i \partial c_{i+1}}\bigl(c_{i}-\pi_{i}, ~d_{i+1}, ~\hist^{i-1} \wedge \frac{\pi_{i}}{d_{i}}\bigr) \mathbf{1}\{\pi_{i} \geq \beta_{\min }^{i-1} d_{i}\}\right] \overset{(i)}{\leq} 0.
    \end{align*}
    Inequality $(i)$ is from the supermodularity of $W_{i+1}^{(\Vsig, ~*)}\left(c_{i}-\pi_{i}, ~d_{i+1}, ~\hist^{i-1} \wedge \frac{\pi_{i}}{d_{i}}\right)$ in $(c_{i+1}, ~\hist^i)$ in \citet[Proposition 7]{ma2022fairness}. Similar to \textbf{Part I} we can apply \citet[Theorem 2.8.2]{topkis1998supermodularity} to show that $\pi_i^{*,~w}(\cdot)$ is non-increasing in $\hist^{i-1}$.
\end{proof}

%% file: parts/appendix/Threshold_properties.tex
\section{Proof of Threshold Structure in \Cref{sec:allocation}} \label{appendix:allocation}

\HqThreshold*
Before we prove \Cref{thm:hq_threshold}, we state preliminary lemmas and defer their proofs to \Cref{appendix: threshold_lemmas}. First, we show that we can simplify the constraint $0 \leq \pi_n \leq c_n \land d_n$ to $\pi_n \leq d_n$.

\begin{restatable}{lemma}{NewFormulation} \label{lemma:unconstrained_pi}
   Given any static routing policy $\Vsig$, for any stage $0 < n < N$ and state $\Vu_n$, any unconstrained maximizer of $h_n(\pi_n; ~d_n)$ is positive and strictly less than the capacity $c_n$. Furthermore, when the optimal solution is not unique, there always exists an unconstrained maximizer that allocates no greater than the demand $d_n$.
\end{restatable}

By \Cref{lemma:unconstrained_pi}, for any stage $0 < n < N$, we immediately have an equivalent formulation to \Cref{eqn:hq_dp}:
\begin{equation}
    \begin{split}
        Z^{(\Vsig, ~*)}_n (\Vu_n) &= \max \limits_{\pi_n \leq d_n} \E_{d_{n+1}\sim \D_{\sigma_{n+1}}} \left[\frac{\pi_n}{d_n} \land Z^{(\Vsig, ~*)}_{n+1} (\Vu_{n+1}) \right],\\
        &= \max \limits_{\pi_n \leq d_n} h_n(\pi_n; ~d_n).
        \label{eqn:new_hq_dp} 
    \end{split}
\end{equation}

We also need \Cref{lemma:marginal_profit_convexity} to show that, due to concavity, the sign of the marginal profit on the boundary determines whether one should allocate the demand or not. 

\MarginalProfit*

Lastly, to prove \Cref{thm:hq_threshold}, we also need \Cref{lemma:expected_fill_rate_one} that gives necessary and sufficient conditions on the capacity $c_n$ such that we achieve a fill-rate of one, i.e., it characterizes capacity regimes where $Z_{n}^{(\Vsig, ~*)} (c_n, ~d_{n}) = 1$.

\begin{restatable}{lemma}{FillRateOneCapacity} 
    \label{lemma:expected_fill_rate_one}
    Given a static routing policy $\Vsig$, for any stage $1 \leq n \leq N$, the expected minimum fill rate $Z_n^{(\Vsig, ~*)}(c_n, ~d_n) = 1$ if and only if 
    \begin{equation}
        c_n \geq d_n + \sum \limits_{i = n+1}^N d_i^{\max}.
    \end{equation}
\end{restatable}

Now we are ready to prove \Cref{thm:hq_threshold}.

\begin{proof}
    First note that when $n = N$ the optimal allocation policy trivially satisfies a threshold structure via \Cref{eqn:hq_dp} since
    \[
    \pi_n^{*,~z}(\Vu_N) = c_N \wedge d_N.
    \]
    This follows a threshold structure with $t_N = c_N$. Hence, we focus on the case when $n < N$.
    The remainder of the proof proceeds in cases corresponding to the three capacity regimes. In the scarce regime, we show that allocating the demand does not satisfy the {\em equating property} and hence is not optimal by \Cref{lemma:hq_equate_property_cont}. In the abundant regime, we apply \Cref{lemma:expected_fill_rate_one} to show that allocating the demand is trivially optimal. Lastly, in the intermediate regime, we write the expectation representation of $h_n(\pi_n; ~d_n)$ as an integral and establish the sign of $\varphi_n(d_n)$.  Then via \Cref{lemma:marginal_profit_convexity}, we obtain the threshold structure.
    
    Since $\D_{n+1}$ has continuous and connected support, we let $F_{n+1}(\cdot)$ denote its cdf and $f_{n+1}(\cdot)$ its pdf.  Thus we can rewrite $h_n(\pi_n;~d_n)$ as
    \begin{equation} \label{eqn:intergal_obj}
        h_n(\pi_n; ~d_n) = \int_{d^{\min}_{n+1}}^{d^{\max}_{n+1}} \left( \frac{\pi_n}{d_n} \land Z_{n+1}^{(\Vsig, ~*)}(c_n - \pi_n,~ d_{n+1}) \right) f_{n+1}(d_{n+1}) \, \mathrm{d} d_{n+1},
    \end{equation}
    where the state $\Vu_n$ is simplified to $(c_n, ~d_n)$ since $\Vsig$ is given.
    
    For simplicity, we assume that for any state $n$, our objective $Z_{n}^{(\Vsig, ~*)} (c_n,~d_n)$ is second-order differentiable in the capacity $c_n$ and observed demand $d_n$. Indeed, $Z_{n}^{(\Vsig, ~*)} (c_n,~d_n)$ is smooth except for a finite number of points due to the minimum operator, replacing derivatives with supergradients preserves the proof.
    We next split into three cases based on the remaining capacity $c_n$, which affects at what values of future demands the equating property hold, and thus the form of the integral.
    
    \textbf{Case I (Scarce Capacity):} We first consider the case when $c_n < d_n^{\min} + d_{n+1}^{\min} + \sum \limits_{i = n+2}^N d_i^{\max}$.  Note that by \Cref{lemma:hq_equate_property_cont}, an optimal allocation satisfies the {\em equating property}, that is, for any observed demand $d_n \in [d_n^{\min},~d_n^{\max}]$, there exists a future demand $d_{n+1} \in  [d_{n+1}^{\min},~d_{n+1}^{\max}]$ such that
    \begin{equation} \label{eqn:proof_hq_equating}
        \frac{\pi_n^{*,~z}}{d_n} = Z_{n+1}^{(\Vsig, ~*)}(c_n - d_n,~ d_{n+1}).
    \end{equation}
    Suppose $\pi_n^{*,~z} = d_n$. By \Cref{eqn:proof_hq_equating}, there must exist a future demand such that 
    \begin{equation} \label{eqn:proof_allocating_equating}
        \frac{\pi_n^{*,~z}}{d_n} = 1 = Z_{n+1}^{(\Vsig, ~*)}(c_n - d_n,~ d_{n+1}).
    \end{equation}
    By the decreasing property of $Z_{n+1}^{(\Vsig, ~*)}$ in demand (\Cref{lemma:monotone}), \Cref{eqn:proof_allocating_equating} must be satisfied at $d_{n+1} = d_{n+1}^{\min}$.
    By \Cref{lemma:expected_fill_rate_one}, this requires that the remaining capacity satisfies $c_{n+1} \geq d_{n+1}^{\min} + \sum \limits_{i = n+2}^N d_i^{\max}$. However, from the capacity condition in the scarce regime, 
    $$c_{n+1} = c_n - \pi_n^{*,~z}(\Vu_n) = c_n - d_n < d_{n+1}^{\min} + \sum \limits_{i = n+2}^N d_i^{\max}.$$
    That results in a contradiction. 
    Since this is true {\em for any demand} $d_n$, the threshold $t_n(c_n) = 0$.

     \textbf{Case II (Abundant Capacity):} We next consider the case when $c_n > \sum \limits_{i = n}^N d_i^{\max}$. From \Cref{lemma:expected_fill_rate_one}, we immediately conclude that $\pi_n^{*,~z} = d_n$ and the threshold $t_n(c_n) = c_n$.

     \textbf{Case III (Intermediate Capacity):} We lastly consider the case when $d_n^{\min} + d_{n+1}^{\min} + \sum \limits_{i = n+2}^N d_i^{\max} \leq c_n \leq \sum \limits_{i = n}^N d_i^{\max}$. We start by presenting \Cref{lemma:non-increasing_marginal} that shows that in the intermediate capacity regime, $\varphi_n(d_n)$ is non-increasing in $d_n$. We defer its proof to \Cref{appendix: threshold_lemmas}.

     \begin{restatable}{lemma}{DecreasingMarginalProfit}
     \label{lemma:non-increasing_marginal}
         For any demand $d_n \in [d_n^{\min}, ~d_n^{\max}]$, $\varphi_n(d_n)$ is non-increasing in $d_n$. Furthermore, if $c_n > d_n^{\max} + d_{n+1}^{\min} + \sum \limits_{i=n+2}^N d_i^{\max}$, then $\varphi_n(d_n)$ is decreasing in $d_n$.
     \end{restatable}

    Leveraging \Cref{lemma:non-increasing_marginal}, we show the threshold structure. Indeed, if $\varphi_n(d_n^{\min}) < 0$, then for any demand $d_n \in [d_n^{\min}, ~d_n^{\max}]$, we have $\varphi_n(d_n) < 0$. By \Cref{lemma:marginal_profit_convexity}, we know that $\pi_n^{*,~z}(d_n) < d_n$ for all $d_n$ and so $t_n(c_n) = 0$. If $\varphi_n(d_n^{\max}) > 0$ then for any demand $d_n \in [d_n^{\min}, ~d_n^{\max}]$, we have $\varphi_n(d_n) > 0$. By \Cref{lemma:marginal_profit_convexity}, we know that $\pi_n^{*,z}(d_n) = d_n$ for all $d_n$ and so $t_n(c_n) = c_n$. Otherwise, we apply the IVT to show that there exists a demand $d_n = t_n(c_n)$ such that $\varphi_n(t_n(c_n)) = 0$. When $d_n < t_n(c_n)$, the marginal profit is non-negative ($\varphi_n(d_n) \ge 0$), and by \Cref{lemma:marginal_profit_convexity}, the optimal policy fully allocates the demand, so $\pi_n^{*,~z}(\Vu_n) = d_n$. When $d_n > t_n(c_n)$, the marginal profit is negative ($\varphi_n(d_n) < 0$), and by \Cref{lemma:marginal_profit_convexity}, the optimal policy allocates strictly less than the demand, so $\pi_n^{*,~z}(\Vu_n) < d_n$.
\end{proof}

\MThreshold*
Before we prove \Cref{thm:m_threshold}, we start by showing that the sign of the boundary function determines whether $\pi_n = \hist^{n-1}d_n$ is optimal or not.  Its proof is deferred to \Cref{appendix: threshold_lemmas}.

\begin{restatable}{lemma}{MMarginalProfit} \label{lemma:marginal_profit_convexity_for_w}
    Given a static routing policy $\Vsig \in \permu{N}$, for any stage $n$, minimum fill rate over previous nodes $\hist^{n-1}$ and state $\Vu_n$, let $h_n(\pi_n;~d_n, ~\hist^{n-1})$ be defined as in \Cref{eqn:def_h_for_w} and $\varphi_n(d_n,~\hist^{n-1})$ be defined as in \Cref{eqn:def_marginal_profit_for_w}.
    When the boundary function is non-negative ($\varphi_n(d_n,~\hist^{n-1}) \geq 0$), the optimal policy maintains the same minimum fill rate, so $\pi_n^{*,~w}(\Vu_n, \hist^{n-1}) = \hist^{n-1}d_n$. When the boundary function is negative ($\varphi_n(d_n,~\hist^{n-1}) < 0$), the optimal policy strictly decreases the minimum fill rate, so $\pi_n^{*,~w}(\Vu_n,~\hist^{n-1}) < \hist^{n-1}d_n$.
\end{restatable}

Lastly, we present a lemma that gives necessary and sufficient conditions on the capacity $c_n$ such that we maintain the same minimum fill rate, i.e., it characterizes capacity regimes where $W_{n}^{(\Vsig, ~*)} (c_n, ~d_{n}, ~\hist^{n-1}) = \hist^{n-1}$.

\begin{restatable}{lemma}{MSameFillRate} 
    \label{lemma:same_fill_rate}
    Given a static routing policy $\Vsig$, for any stage $1 \leq n \leq N$, the expected minimum fill rate $W_{n}^{(\Vsig, ~*)} (c_n, ~d_{n}, ~\hist^{n-1}) = \hist^{n-1}$ if and only if
    \begin{equation}
        c_n \geq \hist^{n-1} \left(d_n + \sum \limits_{i = n+1}^N d_i^{\max} \right).
    \end{equation}
\end{restatable}

Now we are ready to prove \Cref{thm:m_threshold}.

\begin{proof}
    First note that when $n = N$, maintaining the minimum fill rate is an optimal allocation policy that trivially satisfies a threshold structure via \Cref{eqn:m_dp}:
    \[
    \pi_n^{*,~w}(\Vu_N,~\hist^{N-1}) = c_N \wedge \hist^{N-1}d_N.
    \]
    This follows a threshold structure with $t_N = \dfrac{c_N}{\hist^{N-1}}$.
    Hence, we focus on the case when $n < N$. The remainder of the proof proceeds in cases corresponding to the three capacity regimes. In the scarce regime, we show that allocating the prior minimum fill rate is not optimal due to the violation of the {\em equating property}. In the abundant regime, we show that by \Cref{lemma:same_fill_rate}, allocating the prior minimum fill rate is trivially optimal. Lastly, we show that the boundary function $\varphi_n(d_n,~\hist^{n-1})$ is non-increasing, and thus, by \Cref{lemma:marginal_profit_convexity_for_w} we obtain the threshold structure.

    \textbf{Case I (Scarce Capacity):} We first consider the case when $c_n < \hist^{n-1} \sum \limits_{i = n}^N d_i^{\min}$. 
    Note that an optimal allocation policy satisfies the {\em equating property} by \Cref{lemma:m_equate_property}, so given any demand $d_n$, there exists a sequence of future demands $(d_{n+1}, \ldots, d_N)$ such that
    \begin{equation}
    \label{eq:proof_equate_threshold}
    \frac{\pi_n^{*,~w}}{d_n} = \ldots = \frac{\pi_N^{*,~w}}{d_N}.
    \end{equation}
    Suppose that $\pi_n^{*,~w}(\Vu_n,~ \hist^{n-1}) = \hist^{n-1} d_n$. By \Cref{eq:proof_equate_threshold} there must exist a sample path such that 
    \[
    \frac{\pi_n^{*,~w}(\Vu_n, ~\hist^{n-1})}{d_n} = \hist^{n-1} = \frac{\pi_{n+1}^{*,~w}}{d_{n+1}} = \ldots = \frac{\pi_N^{*,~w}}{d_N}.
    \]
    This requires that the remaining capacity satisfies $c_{n+1} \geq \hist^{n-1} \sum \limits_{i=n+1}^N d_i \geq \hist^{n-1} \sum \limits_{i=n+1}^N d_i^{\min}$. However, owing to the capacity condition in the scarce regime, the remaining capacity is
    \begin{equation*}
    c_{n+1} = c_n - \pi_n^{*,~w}(\Vu_n, ~\hist^{n-1}) = c_n - \hist^{n-1} d_n < \hist^{n-1} \sum_{i=n+1}^N d_i^{\min}.
    \end{equation*}
    That induces a contradiction and hence $\pi_n^{*,~w}(\Vu_n, ~\hist^{n-1}) < \hist^{n-1} d_n$. Since that is true {\em for any value} of $d_n$, and the threshold $t_n(c_n,~\hist^{n-1})$ is equal to zero.

     \textbf{Case II (Abundant Capacity):} We next consider the case when $c_n > \hist^{n-1} \sum \limits_{i = n}^N d_i^{\max}$. From \Cref{lemma:same_fill_rate}, we immediately conclude that $\pi_n^{*,~w} = \hist^{n-1}d_n$ and the threshold $t_n(c_n,~\hist^{n-1}) = \dfrac{c_n}{\hist^{n-1}}$.
     
     \textbf{Case III (Intermediate Capacity):} We lastly consider the case when $\hist^{n-1} \sum \limits_{i = n}^N d_i^{\min} \leq c_n \leq \hist^{n-1} \sum \limits_{i = n}^N d_i^{\max}$. We start off by showing the following Lemma, that $\varphi_n(d_n,~\hist^{n-1})$ is non-increasing in this capacity regime. We defer its proof to \Cref{appendix: threshold_lemmas}.

     \begin{restatable}{lemma}{MDecreasingMarginalProfit}
     \label{lemma:non-increasing_marginal_for_w}
         For any demand $d_n \in [d_n^{\min}, ~d_n^{\max}]$, $\varphi_n(d_n,~\hist^{n-1})$ is non-increasing in $d_n$.
     \end{restatable}

    Leveraging \Cref{lemma:non-increasing_marginal_for_w}, we show the threshold structure. Indeed, if $\varphi_n(d_n^{\min},~\hist^{n-1}) < 0$, for any demand $d_n \in [d_n^{\min}, ~d_n^{\max}]$, $\varphi_n(d_n,~\hist^{n-1}) < 0$. By \Cref{lemma:marginal_profit_convexity_for_w}, we know that $\pi_n^{*,~w}(\Vu_n,,~\hist^{n-1}) < \hist^{n-1}d_n$ for all $d_n$ and so $t_n(c_n,~\hist^{n-1}) = 0$. If $\varphi_n(d_n^{\max},~\hist^{n-1}) > 0$ then for any demand $d_n \in [d_n^{\min}, ~d_n^{\max}]$, $\varphi_n(d_n,~\hist^{n-1}) > 0$. By \Cref{lemma:marginal_profit_convexity_for_w}, we know that $\pi_n^{*,w}(\Vu_n,~\hist^{n-1}) = \hist^{n-1}d_n$ for all $d_n$ and so $t_n(c_n,~\hist^{n-1}) = \dfrac{c_n}{\hist^{n-1}}$. 

    Otherwise we have $\varphi_n(d_n^{\min},~\hist^{n-1}) \geq 0$ and $\varphi_n(d_n^{\max},~\hist^{n-1}) \leq 0$. Since $\varphi_n(d_n, ~\hist^{n-1})$ is continuous in $d_n$, we can apply the IVT to show that there exists a demand $t_n(c_n,~\hist^{n-1})$ such that $\varphi_n(t_n(c_n,~\hist^{n-1}),~\hist^{n-1}) = 0$. When $d_n < t_n(c_n,~\hist^{n-1})$, the marginal profit is non-negative ($\varphi_n(d_n,~\hist^{n-1}) \geq 0$), and by \Cref{lemma:marginal_profit_convexity_for_w}, the optimal policy fully allocates the demand, so $\pi_n^{*,~w}(\Vu_n,~\hist^{n-1}) = \hist^{n-1}d_n$. When $d_n > t_n(c_n,~\hist^{n-1})$, the marginal profit is negative ($\varphi_n(d_n,~\hist^{n-1}) < 0$), and by \Cref{lemma:marginal_profit_convexity_for_w}, the optimal policy allocates strictly less than the demand, so $\pi_n^{*,~w}(\Vu_n,~\hist^{n-1}) < \hist^{n-1} d_n$.
\end{proof}

\subsection{Auxiliary Lemmas} \label{appendix: threshold_lemmas}

\NewFormulation*
\begin{proof}
    First, we argue that the optimal allocation will never run out of capacity. Recall the definition of $h_n(\pi_n; ~d_n)$ in \Cref{eqn:def_h} as
    \[
        h_n(\pi_n; ~d_n) = \E_{d_{n+1}\sim \D_{\sigma_{n+1}}} \left[\frac{\pi_n}{d_n} \land Z^{(\Vsig, ~*)}_{n+1} (\Vu_{n+1}) \right].
    \] 
    We have $h_n(c_n;~d_n) = h_n(0;~d_n) = 0$ since by allocating $\pi_n = c_n$ the remaining capacity $c_n = 0$, and by allocating $\pi_n = 0$ the current fill rate is zero.
    Moreover, when $\pi_n < 0$ or $\pi_n > c_n$, $h_n(\pi_n; ~d_n) < 0$. Combined with the fact that $h_n(\pi_n; ~d_n) > 0$ when $0 < \pi_n < c_n$, we know that an unconstrained maximizer is positive and strictly less than the capacity $c_n$.

    Next we show that if there are multiple maximizers of $h_n(\pi_n;~d_n)$, then at least one of them is no more than the demand $d_n$.  This follows from the fact that $h_n(\pi_n;~d_n)$ is non-increasing in $\pi_n$ for $\pi_n \geq d_n$.  Indeed, for any $\pi_n \geq d_n$, we have $h_n(\pi_n;~d_n) = \E_{d_{n+1} \sim D_{\sigma_{n+1}}}[Z^{(\Vsig, ~*)}_{n+1} (\Vu_{n+1})]$. However, $Z^{(\Vsig, ~*)}_{n+1} (\Vu_{n+1})$ is non-decreasing in $c_{n+1} = c_n - \pi_n$ by \Cref{lemma:monotone} and so $Z^{(\Vsig, ~*)}_{n+1} (\Vu_{n+1})$ is non-increasing in $\pi_n$. Therefore, allocating more than the demand will not increase the objective value $h_n(\pi_n;~d_n)$, and there exists an unconstrained maximizer which is no greater than the demand $d_n$.
\end{proof}

\MarginalProfit*
\begin{proof}
    Recall that at stage $n$ by \Cref{lemma:unconstrained_pi} the optimal allocation $\pi_n^{*,~z}(\Vu_n)$ satisfies:
    \[
    \pi_n^{*,~z}(\Vu_n) = \argmax_{\pi_n \leq d_n} h_n(\pi_n; ~d_n).
    \]
    Note that from \Cref{eqn:def_h} $h_n(\pi_n; ~d_n)$ is the expectation of the minimum of $\dfrac{\pi_n}{d_n}$ and $Z_{n+1}^{(\Vsig, ~*)}(\Vu_n)$. Since $Z_{n+1}^{(\Vsig, ~*)}(\Vu_n)$ is concave in $c_{n+1} = c_n - \pi_n$ by \Cref{lemma:concave_in_pi} and thus concave in $\pi_n$, we know that $h_n(\pi_n; ~d_n)$ is concave in $\pi_n$. Hence, the optimal allocation $\pi_n^{*,~z}(\Vu_n)$ is either $d_n$, or a stationary point such that $\frac{\partial h_n}{\partial \pi_n}(\pi_n;~d_n) = 0$.  We first establish that whenever allocating the demand has a positive marginal profit, i.e., $\varphi(d_n) > 0$, then for all $\pi_n \leq d_n$, $\frac{\partial h_n}{\partial \pi_n}(\pi_n;~d_n) > 0$, and so as a result, the optimal policy allocates the demand $d_n$ as there are {\em no} feasible stationary points (when $\varphi_n(d_n) = 0$, $\pi_n = d_n$ is both a stationary point and a boundary point, and thus, $\pi_n^{*,~z}(\Vu_n) = d_n$). We then similarly show that whenever allocating the demand has a negative marginal profit, i.e., $\varphi(d_n) < 0$, then there must be some $0 < \pi_n < d_n$ such that $\frac{\partial h_n}{\partial \pi_n}(\pi_n;~d_n) = 0$, and so the optimal policy allocates strictly less than the demand $d_n$.

    \smallskip
    \noindent {\bf Case I ($\varphi_n(d_n) > 0$)}: First consider the case when allocating the demand has a positive marginal profit, so $\varphi_n(d_n) > 0$. 
    By concavity of $h_n(\pi_n; ~d_n)$, we know that the marginal profit function $\frac{\partial h_n}{\partial \pi_n} (\pi_n; ~d_n)$ is non-increasing, that is, $$\frac{\partial^2 h_n}{\partial \pi_n^2} (\pi_n; ~d_n) \leq 0.$$
    Therefore, for all $\pi_n \leq d_n$, we have $\frac{\partial h_n}{\partial \pi_n} (\pi_n; ~d_n) \geq \varphi_n(d_n) > 0$. In this case, no feasible stationary point exists; the optimal solution is on the boundary, i.e., $\pi_n^{*,~z} (\Vu_n) = d_n$.

    \smallskip
    \noindent {\bf Case II ($\varphi_n(d_n) < 0$)}: Now we consider the case when allocating the demand has a negative marginal profit, so $\varphi_n(d_n) < 0$. Note that since $h_n(0;~d_n) = 0$ and there exists some $\pi_n > 0$ such that $h_n(\pi_n;~d_n) > 0$, the marginal profit must be positive in a neighborhood of zero, i.e., $$\lim \limits_{\pi_n \to 0^+} \frac{\partial h_n}{\partial \pi_n} (\pi_n; ~d_n) > 0.$$ By the continuity of the marginal profit function $\frac{\partial h_n}{\partial \pi_n}(\pi_n; ~d_n)$, this is equivalent to stating that for some $\epsilon > 0$, $\frac{\partial h_n}{\partial \pi_n} (\epsilon; ~d_n) > 0$. Together with the fact that $\varphi_n(d_n) < 0$, we can apply the IVT to conclude that there exists a feasible stationary point $\Pi_n(\Vu_n) < d_n$ such that $\frac{\partial h_n}{\partial \pi_n}(\pi_n; ~d_n)\mid_{\pi_n = \Pi_n(\Vu_n)} = 0$, i.e., $\pi_n^{*,~z}(\Vu_n) = \Pi_n(\Vu_n) < d_n$.
\end{proof}

\FillRateOneCapacity*
\begin{proof}
   The ``if'' direction is immediate: when the capacity suffices to cover the maximum possible future demand, allocating $\pi_i = d_i$ for all $n \le i \le N$ yields $Z_{n}^{(\Vsig,~\Vpi)}(c_n, ~d_n) = 1$, which is clearly optimal.
   Next, we show the ``only if'' direction by backward induction.
    
    \noindent \textbf{Base Case:} $n = N$. Recall that $$Z_{N}^{(\Vsig, ~*)} (c_N, ~d_{N}) = \frac{c_N \land d_N}{d_N}.$$ Then whenever $Z_{N}^{(\Vsig, ~*)} (c_N, ~d_{N}) = 1$ it must hold that $c_N \geq d_N$.
    
    \noindent \textbf{Step Case:} $n+1 \rightarrow n$. By \Cref{eqn:new_hq_dp}, we have
    \begin{equation*}
        Z^{(\Vsig, ~*)}_{n} (c_n, ~d_{n}) = \max \limits_{\pi_n \leq d_{n}} \E_{d_{n+1}\sim \D_{\sigma_{n+1}}} \left[\frac{\pi_{n}}{d_{n}} \land Z^{(\Vsig, ~*)}_{n+1} (c_n - \pi_{n}, ~d_{n+1}) \right].
    \end{equation*}
    To achieve $Z^{(\Vsig, ~*)}_{n} (c_n, ~d_{n}) = 1$, it must hold that $\pi_n^{*,~z} (\Vu_n) = d_n$ and that $Z^{(\Vsig, ~*)}_{n+1} (c_n - d_n, ~d_{n+1}) = 1$ for all $d_{n+1} \in [d_{n+1}^{\min}, d_{n+1}^{\max}]$. Since the objective is non-increasing in the demand $d_{n+1}$ by \Cref{lemma:monotone}, this is equivalent to requiring $Z^{(\Vsig, ~*)}_{n+1} (c_n - d_{n}, ~d_{n+1}^{\max}) = 1$.
    From the induction hypothesis, this holds if and only if 
    \[
    c_{n+1} = c_n - d_n \geq d_{n+1}^{\max} + \sum \limits_{i = n+2}^N d_i^{\max} \quad \Longleftrightarrow \ c_n \geq d_n + \sum \limits_{i = n+1}^N d_i^{\max}.
    \]
    Thus, $Z^{(\Vsig, ~*)}_n (c_n, ~d_n) = 1$ if and only if $c_n \geq d_n + \sum \limits_{i = n+1}^N d_i^{\max}.$
\end{proof}

\DecreasingMarginalProfit*
\begin{proof}
    We begin by expressing $h_n(\pi_n; ~d_n)$ in an integral form that isolates the regions where each term inside the minimum operator dominates.  Specifically, for a given $\pi_n$, we partition the domain of $d_{n+1}$ based on whether the current fill rate $\pi_n / d_n$ or the expected future minimum fill-rate to go function $Z_{n+1}^{(\Vsig, *)}(c_n - \pi_n, ~d_{n+1})$ attains the minimum.  By the equating property (\Cref{lemma:hq_equate_property_cont}), there exists a critical demand that separates these two regions. This critical value enables us to rewrite $h_n(\pi_n; ~d_n)$ as an integral without the minimum operator, yielding an explicit expression for the boundary function $\varphi_n(d_n)$.  Using this representation, we then analyze $\dfrac{\mathrm{d}\varphi_n}{\mathrm{d}d_n}$ and show that it is nonpositive, establishing the monotonicity of $\varphi_n(d_n)$.

    We begin with the case when the capacity satisfies $c_n > d_n^{\max} + d_{n+1}^{\min} + \sum \limits_{i=n+2}^N d_i^{\max}$ and define the {\em critical demand} that allows us to decompose $h_n(\pi_n;~d_n)$ into two different integrals.

    \begin{definition}[Critical Demand] \label{def:critical_demand}
        Let $$\mathcal{G}_{n+1}(c_n, ~d_n, ~\pi_n) \triangleq \left\{d_{n+1} \in [d_{n+1}^{\min}, ~d_{n+1}^{\max}]: \frac{\pi_n}{d_n} = Z^{(\Vsig, ~*)}_{n+1} (c_n - \pi_n,  ~d_{n+1}) \right\}$$ be the set of solutions to \Cref{eqn:hq_equating_current_future}. The {\bf critical demand} is the supremum of the set, i.e., $$g_{n+1}(c_n, ~d_n, ~\pi_n) \triangleq \sup\, \mathcal{G}_{n+1}(c_n, ~d_n, ~\pi_n).$$
    \end{definition}
    Note that $\mathcal{G}_{n+1}(c_n, ~d_n, ~\pi_n^{*,~z}(\Vu_n)) \neq \emptyset$ since the optimal allocation satisfies the equating property (\Cref{lemma:hq_equate_property_cont}).

    First, we show under the capacity condition $c_n > d_n^{\max} + d_{n+1}^{\min} + \sum \limits_{i = n+2}^N d_i^{\max}$ that when $\pi_n = d_n$, the corresponding critical demand satisfies $g_{n+1}(c_n, ~d_n, ~d_n) > d_{n+1}^{\min}$. For notational simplicity, we denote this quantity by $g_{n+1}(c_n, ~d_n)$.

    Indeed, for any $d_n \in [d_n^{\min}, ~d_n^{\max}]$ and some $\epsilon > 0$, we have $c_{n+1} = c_n - d_n > d_{n}^{\max} - d_n + d_{n+1}^{\min} + \sum \limits_{i = n+2}^N d_i^{\max} \geq \epsilon + d_{n+1}^{\min} + \sum \limits_{i = n+2}^N d_i^{\max}$. By \Cref{lemma:expected_fill_rate_one} this implies that $Z_{n+1}^{(\Vsig, ~*)}(c_n - d_n, ~d_{n+1}^{\min}+\epsilon) = 1.$  Thus, by the definition of the critical demand (\Cref{def:critical_demand}), it then follows that $g_{n+1}(c_n, ~d_n) > d_{n+1}^{\min}$.

    Thus, by continuity we know that in the neighborhood around $\pi_n = d_n$ we have:
    \begin{equation} \label{eqn:h_formula_general}
        \begin{split}
        h_n(\pi_n;~d_n) &= \underbrace{\int_{d_{n+1}^{\min}}^{g_{n+1}(c_n, ~d_n,~\pi_n)} \frac{\pi_n}{d_n} f_{n+1}(d_{n+1})  \, \mathrm{d} d_{n+1}}_{(1)}\\
        &+ \underbrace{\int_{g_{n+1}(c_n, ~d_n,~\pi_n)}^{d_{n+1}^{\max}} Z_{n+1}^{(\Vsig, ~*)} (c_n - \pi_n, ~d_{n+1}) f_{n+1}(d_{n+1}) \, \mathrm{d} d_{n+1}}_{(2)}.
        \end{split}
    \end{equation}

    We now compute the derivative of each term separately using the Leibniz integral rule. 
    For simplicity, we assume further that for any state $n$, the critical demand $g_{n+1}(c_n, ~d_n, ~\pi_n)$ is first order differentiable in $c_n$, $d_n$, and $\pi_n$. Indeed, by the implicit function theorem, $g_{n+1}(c_n, ~d_n, ~\pi_n)$ is smooth except for a finite number of points introduced by the minimum operator. In those cases, replacing derivatives with supergradients preserves the validity of the proof. For the first term we have
    \begin{align*}
        \frac{\partial (1)}{\partial \pi_n}
        &= 
        \frac{\partial}{\partial \pi_n}
        \left(
            \int_{d_{n+1}^{\min}}^{g_{n+1}(c_n, ~d_n, ~\pi_n)} 
            \frac{\pi_n}{d_n} f_{n+1}(d_{n+1}) \, \mathrm{d}d_{n+1}
        \right) \\
        &=
        \frac{\pi_n}{d_n} f_{n+1}(g_{n+1}(c_n, ~d_n, ~\pi_n))
        \frac{\partial g_{n+1}}{\partial \pi_n}(c_n, ~d_n, ~\pi_n)
        + \int_{d_{n+1}^{\min}}^{g_{n+1}(c_n, ~d_n, ~\pi_n)} 
            \frac{1}{d_n} f_{n+1}(d_{n+1}) \, \mathrm{d}d_{n+1}.
    \end{align*}
    Evaluating this at $\pi_n = d_n$ yields:
    \[
    \frac{\partial (1)}{\partial \pi_n} \Bigg |_{\pi_n = d_n} = f_{n+1}(g_{n+1}(c_n, ~d_n))
        \frac{\partial g_{n+1}}{\partial \pi_n}(c_n, ~d_n) + \frac{F_{n+1}(g_{n+1}(c_n,~ d_n))}{d_n}.
    \]

    For the second term we have:
    \begin{align*}
    \frac{\partial (2)}{\partial \pi_n}
    &= 
    \frac{\partial}{\partial \pi_n}
    \left(
        \int_{g_{n+1}(c_n,~ d_n, ~\pi_n)}^{d_{n+1}^{\max}} 
        Z_{n+1}^{(\Vsig, *)}(c_n - \pi_n, ~d_{n+1}) 
        f_{n+1}(d_{n+1}) \, \mathrm{d}d_{n+1}
    \right) \\
    &=
    -\,Z_{n+1}^{(\Vsig, *)}(c_n - \pi_n, ~g_{n+1}(c_n, ~d_n, ~\pi_n))
    f_{n+1}(g_{n+1}(c_n, ~d_n, ~\pi_n))
    \frac{\partial g_{n+1}}{\partial \pi_n}(c_n, ~d_n, ~\pi_n) \\
    &\quad
    - \int_{g_{n+1}(c_n, ~d_n, ~\pi_n)}^{d_{n+1}^{\max}} 
        \frac{\partial Z_{n+1}^{(\Vsig, *)}}{\partial c_{n+1}}
        (c_n - \pi_n, ~d_{n+1}) f_{n+1}(d_{n+1}) 
        \, \mathrm{d}d_{n+1}.
    \end{align*}
    Evaluating this at $\pi_n = d_n$ yields:
    \begin{align*}
    \frac{\partial (2)}{\partial \pi_n} \Bigg |_{\pi_n = d_n} &= -\,Z_{n+1}^{(\Vsig, *)}(c_n - d_n, ~g_{n+1}(c_n, ~d_n))
    f_{n+1}(g_{n+1}(c_n, ~d_n))
    \frac{\partial g_{n+1}}{\partial \pi_n}(c_n, ~d_n) \\
    & \quad - \int_{g_{n+1}(c_n, ~d_n)}^{d_{n+1}^{\max}} \frac{\partial Z_{n+1}^{(\Vsig, ~*)} }{\partial c_{n+1}} (c_n-d_n, ~d_{n+1}) f_{n+1}(d_{n+1}) \, \mathrm{d} d_{n+1}.
    \end{align*}

    By definition of the critical demand $g_{n+1}(c_n, ~d_n)$ such that $Z_{n+1}^{(\Vsig, ~*)} (c_n - d_n, ~g_{n+1}(c_n, ~d_n)) = 1$, when combining both terms gives that:     
    \[
    \varphi_n(d_n) = \frac{\partial h_n(\pi_n; ~d_n)}{\partial \pi_n} \Bigg |_{\pi_n = d_n} = \frac{F_{n+1}(g_{n+1}(c_n, ~d_n))}{d_n} - \int_{g_{n+1}(c_n, ~d_n)}^{d_{n+1}^{\max}} \frac{\partial Z_{n+1}^{(\Vsig, ~*)} }{\partial c_{n+1}} (c_n-d_n, ~d_{n+1}) f_{n+1}(d_{n+1}) \, \mathrm{d} d_{n+1}.
    \]

    Hence the derivative of $\varphi_n(d_n)$ with respect to $d_n$ is:
    \begin{align*}
         \frac{\mathrm{d} \varphi_n}{\mathrm{d} d_n} &= \underbrace{-\frac{F_{n+1}(g_{n+1}(c_n, ~d_n))}{d_n^2}}_{(1)} + \underbrace{\frac{f_{n+1}(g_{n+1}(c_n, ~d_n)) \frac{\partial g_{n+1}}{\partial d_n} (c_n, ~d_n)}{d_n}}_{(2)}\\
        &~~~+ \underbrace{\int_{g_{n+1}(c_n, ~d_n)}^{d_{n+1}^{\max}} \frac{\partial^2 Z_{n+1}^{(\Vsig, ~*)} }{\partial c_{n+1}^2} (c_n - d_n, ~d_{n+1}) f_{n+1}(d_{n+1}) \, \mathrm{d} d_{n+1}}_{(3)}\\
        &~~~+ \underbrace{\frac{\partial Z_{n+1}^{(\Vsig, ~*)} }{\partial c_{n+1}}  (c_n - d_n, g_{n+1}(c_n, ~d_n)) f_{n+1}(g_{n+1}(c_n, ~d_n)) \frac{\partial g_{n+1}}{\partial d_n}(c_n, ~d_n)}_{(4)}.
    \end{align*}
    Before we check the sign of the derivative of the marginal profit, we state \Cref{lemma:nonincreasing_solution}, which indicates that the critical demand $g_{n+1}(c_n, ~d_n)$ is non-increasing in $d_n$. Again, the proof is deferred.
    \begin{restatable}{lemma}{NondecreasingCriticalDemand}
        \label{lemma:nonincreasing_solution}
    Suppose that $c_n \geq d_n + d_{n+1}^{\min} + \sum \limits_{i = n+2}^N d_i^{\max}$. Let $g_{n+1}(c_n, ~d_n)$ be the critical demand when the allocation satisfies the demand, i.e., $g_{n+1}(c_n, ~d_n) \triangleq g_{n+1}(c_n, ~d_n, ~\pi_n)\big |_{\pi_n = d_n}.$  Then $g_{n+1}(c_n, ~d_n)$ is non-increasing in the current demand $d_n$.
    \end{restatable}
    
    Therefore, $(1)$ is negative because the CDF is positive when $g_{n+1}(c_n, ~d_n) > d_{n+1}^{\min}$, $(2)$ is non-positive due to the deinition of the pdf and non-increasing property of the critical demand $g_{n+1}(c_n, ~d_n)$ in $d_n$ (\Cref{lemma:nonincreasing_solution}), $(3)$ is non-positive by the concavity of $Z_{n+1}^{(\Vsig, ~*)} (\Vu_{n+1})$ in capacity $c_{n+1}$ (\Cref{lemma:concave_in_pi}) and the positiveness of the pdf, and $(4)$ is non-positive because of the non-decreasing property of $Z_{n+1}^{(\Vsig, ~*)} (\Vu_{n+1})$ in capacity $c_{n+1}$ (\Cref{lemma:monotone}), definition of pdf and non-increasing property of the critical demand $g_{n+1}(c_n, ~d_n)$ in demand $d_n$ (\Cref{lemma:nonincreasing_solution}). Overall, this shows that $\dfrac{\mathrm{d} \varphi_n}{\mathrm{d} d_n} < 0$ and so $\varphi_n(d_n)$ is decreasing in $d_n$ as required.

        Next we consider the case when $c_n \leq d_n^{\max} + d_{n+1}^{\min} + \sum \limits_{i=n+2}^N d_i^{\max}$. Here, we have that for some $d_n$, the critical demand $g_{n+1} (c_n, ~d_n) < d_{n+1}^{\min}$.  Then we only have term $(2)$ in \Cref{eqn:h_formula_general}. As a consequence, the derivative of $h_n(\pi_n;~d_n)$ is
        \begin{align*}
            \frac{\partial h_n}{\partial \pi_n} (\pi_n;~d_n) &= -\int_{d_{n+1}^{\min}}^{d_{n+1}^{\max}} \frac{\partial Z_{n+1}^{(\Vsig, ~*)} }{\partial c_{n+1}} (c_n - \pi_n, ~d_{n+1}) f_{n+1}(d_{n+1}) \, \mathrm{d} d_{n+1}.
        \end{align*}
    and the derivative of $\varphi_n(d_n)$ is
    \begin{align*}
         \frac{\mathrm{d} \varphi_n}{\mathrm{d} d_n} &= \int_{d_{n+1}^{\min}}^{d_{n+1}^{\max}} \frac{\partial^2 Z_{n+1}^{(\Vsig, ~*)} }{\partial c_{n+1}^2} (c_n - d_n, ~d_{n+1}) f_{n+1}(d_{n+1}) \, \mathrm{d} d_{n+1} \leq 0,
    \end{align*}
    from the same argument above. This completes the proof of \Cref{lemma:non-increasing_marginal}.
\end{proof}

We now present the proof for \Cref{lemma:nonincreasing_solution}.
\NondecreasingCriticalDemand*
\begin{proof}
    Recall that by \Cref{def:critical_demand}, the critical demand $g_{n+1}(c_n, ~d_n, ~\pi_n)$ is defined as the maximum $d_{n+1}$ such that the current fill rate equals to the future expected minimum fill rate to go function:
    $$\frac{\pi_n}{d_n} = Z^{(\Vsig, ~*)}_{n+1} (c_n - \pi_n, ~g_{n+1}(c_n, ~d_n, ~\pi_n)).$$ 
    
    Hence, $g_{n+1}(c_n,~d_n)$ is the maximum value of demand such that $Z^{(\Vsig, ~*)}_{n+1} (c_n - d_n, ~g_{n+1}(c_n, ~d_n)) = 1$ holds. We show by contradiction that for any demand $d_n^{\prime} > d_n$, we have $g_{n+1}(c_n, ~d_n^{\prime}) \leq g_{n+1}(c_n, ~d_n)$.

    Assume that $g_{n+1}(c_n, ~d_n^{\prime}) > g_{n+1}(c_n, ~d_n)$ instead. Then we have 
    $$1 \overset{(i)}{=} Z^{(\Vsig, ~*)}_{n+1} (c_n - d_n^{\prime},~g_{n+1}(c_n,~d_n^{\prime})) \overset{(ii)}{\leq} Z^{(\Vsig, ~*)}_{n+1} (c_n - d_n, ~g_{n+1}(c_n,~d_n^{\prime})),$$
    where equality $(i)$ is by \Cref{def:critical_demand}, and inequality $(ii)$ is due to $c_{n+1}^{\prime} = c_n - d_n^{\prime} < c_{n+1} = c_n - d_n$ and the non-decreasing property in capacity $c_{n+1}$ (\Cref{lemma:monotone}). This indicates that $g_{n+1}(c_n,~d_n^{\prime}) \in \mathcal{G}_{n+1}(c_n, ~d_n, ~\pi_n) \Big|_{\pi_n = d_n}$ and contradicts to the fact that $g_{n+1}(c_n, ~d_n)$ is the supremum.
\end{proof}

\MMarginalProfit*
\begin{proof}
    At stage $n$, the optimal allocation policy satisfies:
    \[
    \pi_n^{*,~w}(\Vu_n, \hist^{n-1}) = \argmax_{0 \leq \pi_n \leq c_n \wedge d_n} h_n(\pi_n;~d_n,~\hist^{n-1}).
    \]
    From \Cref{lemma:concave_in_pi}, we know that $W^{(\Vsig,~*)}_{n+1}\left(\Vu_{n+1}, ~\hist^{n-1} \land \dfrac{\pi_n}{d_n} \right)$ is concave in $\pi_n$ and thus $h_n(\pi_n;~d_n,~\hist^{n-1})$, as the expectation of $W^{(\Vsig,~*)}_{n+1}(\cdot)$, is concave in $\pi_n$. Hence, the optimal allocation must be either the boundary point $\pi_n = c_n \wedge d_n$ or a feasible stationary point satisfying:
    \[
    \frac{\partial h_n}{\partial \pi_n}(\pi_n;~d_n,~\hist^{n-1}) = 0.
    \]

   The key observation is that the marginal profit becomes non-positive once $\pi_n$ exceeds $\hist^{n-1} d_n$. This ensures that it suffices to restrict attention to $\pi_n \leq \hist^{n-1} d_n$. To see this, differentiate $h_n$ via:
    \begin{equation}
    \label{eqn:partial_h_derivative_proof}
    \begin{split}
        &\frac{\partial h_n}{\partial \pi_n}(\pi_n; ~d_n, ~\hist^{n-1})\\
        =&\E_{d_{n+1} \sim \D_{\sigma_{n+1}}} \left[- \frac{\partial W_{n+1}^{(\Vsig, ~*)}}{\partial c_{n+1}}\left(\Vu_{n+1}, ~\beta_{\min}^{n-1} \land \frac{\pi_n}{d_n} \right) + \frac{1}{d_n} \frac{\partial W_{n+1}^{(\Vsig, ~*)}}{\partial \hist^{n}}\left(\Vu_{n+1}, ~\frac{\pi_n}{d_n} \right) \mathbf{1} \left\{\pi_n \leq \hist^{n-1}d_n \right\} \right].
    \end{split}
    \end{equation}
    The first term arises since $\Vu_{n+1} = c_n - \pi_n$, and the second term only arises when $\frac{\pi_n}{d_n} \wedge \hist^{n-1} = \dfrac{\pi_n}{d_n}$. Note that the marginal profit is smooth except at the point $\pi_n = \hist^{n-1}d_n$. 

    When $\pi_n > \hist^{n-1} d_n$, from \Cref{eqn:partial_h_derivative_proof} we have 
    \[
    \frac{\partial h_n}{\partial \pi_n}(\pi_n; ~d_n, ~\hist^{n-1}) = \E_{d_{n+1} \sim \D_{\sigma_{n+1}}} \left[- \frac{\partial W_{n+1}^{(\Vsig, ~*)}}{\partial c_{n+1}}\left(\Vu_{n+1}, ~\beta_{\min}^{n-1} \land \frac{\pi_n}{d_n} \right)\right].
    \]
    By \Cref{lemma:monotone} we know that this term is non-positive. Together with concavity, which implies that the derivative is non-increasing, it suffices to restrict attention to $\pi \le \hist^{n-1} d_n$.
    
    The remainder of the proof follows from identical arguments to \Cref{lemma:marginal_profit_convexity}, where we break into cases depending on the sign of $\varphi_n(d_n,~\hist^{n-1})$.  We omit the details here.
\end{proof}

\MSameFillRate*
\begin{proof}
   The ``if'' direction is immediate: when the capacity suffices to cover the maximum possible future demand, allocating $\pi_i = \hist^{n-1}d_i$ for all $n \le i \le N$ yields $W_{n}^{(\Vsig,~\Vpi)}(c_n, ~d_n,~\hist^{n-1}) = \hist^{n-1}$, which is clearly optimal combined with \Cref{lemma:monotone}.
   Next, we show the ``only if'' direction by backward induction.
    
    \noindent \textbf{Base Case:} $n = N$. Recall that $$W_{N}^{(\Vsig, ~*)} (c_N, ~d_{N},~\hist^{n-1}) = \hist^{n-1} \land \frac{c_N}{d_N}.$$ Then whenever $W_{N}^{(\Vsig, ~*)} (c_N, ~d_{N},~\hist^{n-1}) = \hist^{n-1}$ it must hold that $c_N \geq \hist^{n-1} d_N$.
    
    \noindent \textbf{Step Case:} $n+1 \rightarrow n$. By \Cref{eqn:m_dp}, we have
    \begin{equation*}
        W_n^{(\Vsig, ~*)}\left(c_n,~d_n ~\beta_{\min}^{n-1} \right) = \E_{d_{n+1} \sim \D_{\sigma_{n+1}}} \left[W_{n+1}^{(\Vsig, ~*)}\left(c_{n+1}, ~d_{n+1},~\beta_{\min}^{n-1} \land \frac{\pi_n}{d_n} \right) \right].
    \end{equation*}
    To achieve $W_n^{(\Vsig, ~*)}\left(c_n, ~d_n,~\beta_{\min}^{n-1} \right) = \hist^{n-1}$, it must hold that $\pi_n = \hist^{n-1}d_n$ and that $W_{n+1}^{(\Vsig, ~*)}\left(c_{n+1},~d_{n+1}, ~\beta_{\min}^{n-1} \right) = \hist^{n-1}$ for all $d_{n+1} \in [d_{n+1}^{\min}, ~d_{n+1}^{\max}]$. Since the objective is non-increasing in the demand $d_{n+1}$ by \Cref{lemma:monotone}, this is equivalent to requiring $W_{n+1}^{(\Vsig, ~*)}\left(c_{n+1}, ~d_{n+1}^{\max}, ~\beta_{\min}^{n-1} \right) = \hist^{n-1}$.
    From the induction hypothesis, this holds if and only if 
    \[
    c_{n+1} = c_n - \hist^{n-1}d_n \geq \hist^{n-1} \left(d_{n+1}^{\max} + \sum \limits_{i = n+2}^N d_i^{\max} \right) \quad \Longleftrightarrow \ c_n \geq \hist^{n-1} \left(d_n + \sum \limits_{i = n+1}^N d_i^{\max} \right).
    \]
    Thus, $W_n^{(\Vsig, ~*)}\left(c_n, ~d_n,~\beta_{\min}^{n-1} \right) = \hist^{n-1}$ if and only if $c_n \geq \hist^{n-1} \left( d_n + \sum \limits_{i = n+1}^N d_i^{\max} \right).$
\end{proof}

\MDecreasingMarginalProfit*
\begin{proof}
    We begin by explicitly deriving the boundary function using its definition. Then we show the derivative of $\varphi_n(d_n,~\hist^{n-1})$ with respect to $d_n$ is non-positive. 
    
    By \Cref{eqn:partial_h_derivative_proof}, we have
    \begin{equation}
    \begin{split}
        &\varphi_n(d_n,~\hist^{n-1}) \\
        = &\lim_{\pi_n \to (\hist^{n-1}d_n)^-}\frac{\partial h_n}{\partial \pi_n}(\pi_n; ~d_n, ~\hist^{n-1})\\
        = &\lim_{\pi_n \to (\hist^{n-1}d_n)^-} \E_{d_{n+1} \sim \D_{\sigma_{n+1}}} \left[- \frac{\partial W_{n+1}^{(\Vsig, ~*)}}{\partial c_{n+1}}\left(\Vu_{n+1}, ~\frac{\pi_n}{d_n} \right) + \frac{1}{d_n} \frac{\partial W_{n+1}^{(\Vsig, ~*)}}{\partial \hist^{n}}\left(\Vu_{n+1}, ~\frac{\pi_n}{d_n} \right) \right],\\
        = &\E_{d_{n+1} \sim \D_{\sigma_{n+1}}} \left[\underbrace{- \frac{\partial W_{n+1}^{(\Vsig, ~*)}}{\partial c_{n+1}}\left(c_n - \hist^{n-1}d_n,~d_{n+1}, ~\beta_{\min}^{n-1} \right)}_{(1)} + \underbrace{ \frac{1}{d_n} \frac{\partial W_{n+1}^{(\Vsig, ~*)}}{\partial \hist^{n}}\left(c_n - \hist^{n-1}d_n,~d_{n+1}, ~\beta_{\min}^{n-1} \right)}_{(2)} \right].
    \end{split}
    \end{equation}
    We now compute the derivative of the boundary function $\varphi_n(d_n,~\hist^{n-1})$ with respect to the demand $d_n$ for each term. For the first term, we have
    \begin{align*}
        \frac{\partial (1)}{\partial d_n} &= \frac{\partial}{\partial d_n} \left( - \frac{\partial W_{n+1}^{(\Vsig, ~*)}}{\partial c_{n+1}}\left(c_n - \hist^{n-1}d_n,~d_{n+1}, ~\beta_{\min}^{n-1} \right) \right),\\
        &= \hist^{n-1}\frac{\partial^2 W_{n+1}^{(\Vsig, ~*)}}{\partial c_{n+1}^2}\left(c_n - \hist^{n-1}d_n,~d_{n+1}, ~\beta_{\min}^{n-1} \right),
    \end{align*}
    which is non-positive due to concavity (\Cref{lemma:concave_in_pi}).
    For the second term, we have
    \begin{align*}
        \frac{\partial (2)}{\partial d_n} &= \frac{\partial}{\partial d_n} \left( \frac{1}{d_n} \frac{\partial W_{n+1}^{(\Vsig, ~*)}}{\partial \hist^{n}}\left(c_n - \hist^{n-1}d_n,~d_{n+1}, ~\beta_{\min}^{n-1} \right)\right),\\
        &= -\frac{1}{d_n^2}\frac{\partial W_{n+1}^{(\Vsig, ~*)}}{\partial \hist^{n}}\left(c_n - \hist^{n-1}d_n,~d_{n+1}, ~\beta_{\min}^{n-1} \right) - \frac{\hist^{n-1}}{d_n}\frac{\partial^2 W_{n+1}^{(\Vsig, ~*)}}{\partial c_{n+1} \partial \hist^{n}}\left(c_n - \hist^{n-1}d_n,~d_{n+1}, ~\beta_{\min}^{n-1} \right)
    \end{align*}
    by the product rule. Note that $W_{n+1}^{(\Vsig,~*)} (\Vu_{n+1},~\hist^{n})$ is non-decreasing in $\hist^{n-1}$ (\Cref{lemma:monotone}) and supermodular in $(c_{n+1}, ~\hist^{n-1})$ from \citet[Proposition 7]{ma2022fairness}, so the second term is also non-positive.
\end{proof}

%% file: parts/appendix/decreasing_CV_proofs.tex
\section{Proof of Analytical Routing Results in \cref{sec:sequencing}} \label{appendix:decreasing_CV}

\DeterministicOpt*
\begin{proof}
    We show the result by backward induction. Specifically, fix any stage $0 < n \leq N - 2$ at which the next node has not yet been selected.  For any state state $\Vu_n$, minimum fill rate over previously visited nodes $\hist^{n-1}$, and any feasible routing policy $\Vsig$ (either dynamic or static) for which the deterministic node $i_D$ has not yet been visited, let $\Vsig'$ denote the policy obtained from $\Vsig$ by preserving the original visitation order and moving $i_D$ to the end.  We show that
    \(
    W_n^{(*,~\Vsig')}(\Vu_n,~\beta_{\min}^{n-1}) \;\geq\; W_n^{(*,~\Vsig)}(\Vu_n,~\beta_{\min}^{n-1}).
    \)
    Taking expectations on both sides for $n = 1$ establishes that
    \(
    W_0^{(*,~\Vsig’)}(c) \;\geq\; W_0^{(*,~\Vsig)}(c).
    \)
    We assume $N \geq 3$, since the case $N = 2$ follows directly from the base case.

    \noindent \textbf{Base Case:} $n = N - 2$.     
    At this stage, there are only two possible routing policies to consider: (i)  $\Vsig$ which visits the deterministic node at the next stage ($\sigma_{N-1} = i_D$) or (ii) $\Vsig'$ which visits the deterministic node at the end ($\sigma^{\prime}_{N} = i_D$).  We first analyze the objective under $\Vsig$, where we visit $i_D$ before the final stochastic node:
    \begin{align*}
        &\quad W_{N-2}^{(*,~\Vsig) }\left(c_{N-2}, ~\sigma_{N-2}, ~d_{N-2}, ~\{i_D, ~\sigma_N\}, ~\beta_{\min}^{N-2} \right)\\
        &\overset{(i)}{=} W_{N-1}^{(*,~\Vsig) }\left(c_{N-2} - \pi_{N-2}^{*,~w}, ~i_D, ~\mu_D, ~\{\sigma_N\}, ~\beta_{\min}^{N-2} \land \frac{\pi_{N-2}^{*,~w}}{d_{N-2}} \right),\\
        &\overset{(ii)}{=} \max \limits_{0 \leq \pi_{N-1} \leq c_{N-1} \land \mu_D} \E_{d_N \sim \D_{\sigma_N}} \left[ \beta_{\min}^{N-2} \land \frac{\pi_{N-2}^{*,~w}}{d_{N-2}} \land \frac{\pi_{N-1}}{\mu_D} \land \frac{c_{N-1} - \pi_{N-1}}{d_N}\right],\\
        &\overset{(iii)}{=} \max \limits_{0 \leq \pi_{N-1} \leq c_{N-1}} \E_{d_N \sim \D_{\sigma_N}} \left[ \beta_{\min}^{N-2} \land \frac{\pi_{N-2}^{*,~w}}{d_{N-2}} \land \frac{\pi_{N-1}}{\mu_D} \land \frac{c_{N-1} - \pi_{N-1}}{d_N}\right],\\
        &\overset{(iv)}{=} \max \limits_{0 \leq \pi_{N} \leq c_{N-1}} \E_{d_N \sim \D_{\sigma_N}} \left[ \beta_{\min}^{N-2} \land \frac{\pi_{N-2}^{*,~w}}{d_{N-2}} \land \frac{c_{N-1} - \pi_{N}}{\mu_D} \land \frac{\pi_{N}}{d_N} \right],
    \end{align*}
    where equality $(i)$ is from \Cref{eqn:m_bellman} by assuming that the optimal allocation $\pi_{N-2}^{*,~w}$ is found and the fact that $i_D$ is deterministic, equality $(ii)$ is from \Cref{eqn:m_dp} and $c_{N-1} = c_{N-2} - \pi_{N-2}^{*,~w}$, equality $(iii)$ removes the constraint $\pi_{N-1} \leq \mu_D$ since it does not change the objective value ($\hist^{N-2} \leq 1$), and equality $(iv)$ is from changing the decision variable since $\pi_N = c_{N-1} - \pi_{N-1}$.
    On the other hand, when we apply the \routing policy $\Vsig'$,
    \begin{align*}
        & \quad W_{N-2}^{(*,~\Vsig') }\left(c_{N-2}, ~\sigma_{N-2}, ~d_{N-2}, ~\{i_D, ~\sigma_N\}, ~\beta_{\min}^{N-2} \right) \\
        & \overset{(i)}{=} \E_{d_{N-1} \sim \D_{\sigma_N}} \left[W_{N-1}^{(*,~\Vsig') }\left(c_{N-2} - \pi_{N-2}^{*,~w}, ~\sigma_N, ~d_{N-1}, ~\{i_D\}, ~\beta_{\min}^{N-2} \land \frac{\pi_{N-2}^{*,~w}}{d_{N-2}} \right) \right],\\
        & \overset{(ii)}{=} \E_{d_{N-1} \sim \D_{\sigma_N}} \left[\max \limits_{0 \leq \pi_{N-1} \leq c_{N-1} \land d_N} 
        \left[ \beta_{\min}^{N-2} \land \frac{\pi_{N-2}^{*,~w}}{d_{N-2}} \land \frac{\pi_{N-1}}{d_{N-1}} \land \frac{c_{N-1} - \pi_{N-1}}{\mu_D} \right] \right],\\
        &\overset{(iii)}{=} \E_{d_{N-1} \sim \D_{\sigma_N}} \left[\max \limits_{0 \leq \pi_{N-1} \leq c_{N-1}} 
        \left[ \beta_{\min}^{N-2} \land \frac{\pi_{N-2}^{*,~w}}{d_{N-2}} \land \frac{\pi_{N-1}}{d_{N-1}} \land \frac{c_{N-1} - \pi_{N-1}}{\mu_D} \right] \right],
    \end{align*}
    where equality $(i)$ is from \Cref{eqn:m_bellman} and the fact that $\sigma_{N-2} = \sigma^{\prime}_{N-2}$ so the state at $N-2$ and $\pi_{N-2}^{*,~w}$ are the same, equality $(ii)$ is from \Cref{eqn:m_dp} and $c_{N-1} = c_{N-2} - \pi_{N-2}^{*,~w}$ and equality $(iii)$ is due to $\hist^{N-2} \leq 1$ and removing the constraint $d_N$ does not affect the objective value. Trivially, $$W_{N-2}^{(*,~\Vsig') }\left(c_{N-2}, ~\sigma_{N-2}, ~d_{N-2}, ~\{i_D, ~\sigma_N\}, ~\beta_{\min}^{N-2} \right) \geq W_{N-2}^{(*,~\Vsig) }\left(c_{N-2}, ~\sigma_{N-2}, ~d_{N-2}, ~\{i_D, ~\sigma_N\}, ~\beta_{\min}^{N-2} \right)$$ since $\Vsig'$ maximizes over {\em every} sample path. 

    Note that when $N = 2$, the base case is $n = 0$ and thus there is no step case. Indeed, we show $W_0^{(*, ~\Vsig')}\left(c \right) \geq W_0^{(*, ~\Vsig)}\left(c \right)$ by taking $\hist^{N-2} = 1$.

    \noindent \textbf{Step Case:} $n = i+1 \rightarrow i$.
    First note that by induction hypothesis, if a \routing policy $\Vsig$ puts the deterministic node after stage $i$ but not the end, i.e., for some $i+1 \leq j < N$, $\sigma_j = i_D$, one can switch $\sigma_N$ with $i_D$ so that 
    $$W_i^{(*,~\Vsig')}\left(\Vu_i, ~\beta_{\min}^{i-1} \right) \overset{(i)}{=} \E_{\vec{d}_{i:j}} \left [W_j^{(*,~\Vsig')}\left(\Vu_j, ~\beta_{\min}^{j-1} \right) \right] \overset{(ii)}{\geq} \E_{\vec{d}_{i:j}} \left [W_j^{(*,~\Vsig)}\left(\Vu_j, ~\beta_{\min}^{j-1} \right) \right] \overset{(iii)}{=} W_i^{(*,~\Vsig)}\left(\Vu_i, ~\beta_{\min}^{i-1} \right),$$
    where equality $(i)$ and equality $(iii)$ are applying \Cref{eqn:m_dp} multiple times, and equality $(ii)$ is from induction hypothesis. Thus, we only have two candidates for the optimal static routing policy: putting the deterministic node at the end, denoted by $\Vsig'$ ($\sigma^{\prime}_N = i_D$), or putting the deterministic node at the beginning, denoted by $\Vsig$ ($\sigma_{1} = i_D$). 

    Let $\Vpi^{*}(\Vsig) = \left(\pi_i^*, ~\pi_{i+1}^*, ~\dots, ~\pi_N^* \right)$ be the optimal allocation policy for the DP model $W_i^{(*, ~\Vsig)} \left(\Vu_i, ~\beta_{\min}^{i-1} \right)$. Since node $D$ is deterministic, $\pi_D^*$ is a fixed value, and $\Vpi^{*}(\Vsig)$ is feasible to the DP model $W_i^{(\Vpi, ~\Vsig')} \left(\Vu_i, ~\beta_{\min}^{i-1} \right)$. Therefore, $$ W_i^{(*, ~\Vsig')} \left(\Vu_i, ~\beta_{\min}^{i-1} \right) \geq W_i^{(\Vpi^*(\Vsig), ~\Vsig')} \left(\Vu_i, ~\beta_{\min}^{i-1} \right) \geq W_i^{(*, ~\Vsig)} \left(\Vu_i, ~\beta_{\min}^{i-1} \right).$$
    Thus $\Vsig'$ is better than $\Vsig$, showing that $\Vsig'$ is optimal.
\end{proof}

\DeterministicPPA*
\begin{proof}
We proceed in two steps. First, we show that positioning the deterministic node in the middle of the route is suboptimal, as moving it to the end improves the objective. Second, we show that placing the deterministic node at the beginning is also suboptimal.

\textbf{Case I ($\sigma^*_2 \neq i_D$):} Let $\Vsig$ be any \routing policy (either dynamic or static) such that for some state $(\Vu_1, ~\hist^0)$, it visits the deterministic node in the middle, i.e., $\sigma_2(\Vu_1, ~\hist^0) = i_D$. Let $\Vsig'$ be the \routing policy constructed from $\Vsig$ such that preserves the original visitation order except moving $i_D$ to the end for these states $(\Vu_1, ~\hist^0)$ and sample paths following them, i.e., $\sigma'_1 = \sigma_1$, $\sigma'_2 = \sigma_3$ and $\sigma'_3 = i_D$. We will show that for any state $(\Vu_1, ~\hist^0)$, we have 
$$W_1^{ (\Vsig', ~\Vpi ^ { P P A }) } \left(\Vu_1, ~\beta_{\min}^{0} \right) \geq W_1^{ (\Vsig, ~\Vpi ^ { P P A }) } \left(\Vu_1, ~\beta_{\min}^{0} \right),$$
and therefore $\Vsig$ is not optimal. Specifically, we show that the minimum fill rate over the same sample path under $\Vsig'$ is higher than under $\Vsig$. Since $\sigma'_1 = \sigma_1$ and we apply the PPA policy for both \routing policy, $\beta^{(\Vsig', ~\Vpi ^ { P P A })}_{\sigma'_1} = \beta^{(\Vsig, ~\Vpi ^ { P P A })}_{\sigma_1}$. Moreover, starting from the same $c$ results in the same $c_2$ for both routes.

Let $(d_1, ~i_D, ~d_3)$ be the sample path generated in $\Vsig$. The corresponding sample path in $\Vsig'$ is $(d_1, ~d_3,~i_D)$. By \Cref{eqn:PPA}, when the \routing policy is $\Vsig'$, we have
\begin{align}
    \beta^{(\Vsig', ~\Vpi ^ { P P A })}_{\sigma_2'} &= \beta^{(\Vsig', ~\Vpi ^ { P P A })}_{\sigma_3} = \frac{c_2}{d_3 + \mu_D} \land 1,\label{eqn:beta_2_sig_prime} \\
    \beta^{(\Vsig', ~\Vpi ^ { P P A })}_{i_D} &= \frac{c_2 - \left( \frac{d_3}{d_3 + \mu_D}c_2 \land d_3 \right)}{\mu_D} \land 1 = \frac{c_2}{d_3 + \mu_D} \land 1. \label{eqn:beta_D_sig_prime}
\end{align}
When the \routing policy is $\Vsig$, we have
\begin{align}
    \beta^{(\Vsig, ~\Vpi ^ { P P A })}_{i_D} &= \frac{c_2}{\mu_D + \mu_{\sigma_3}} \land 1,\label{eqn:beta_D_sig} \\
    \beta^{(\Vsig, ~\Vpi ^ { P P A })}_{\Vsig_{3}} &= \frac{c_2 - \left( \frac{\mu_D}{\mu_D + \mu_{\sigma_3}}c_2 \land \mu_D \right)}{d_3} \land 1 =  \begin{cases} 
    \frac{c_2}{\mu_D + \mu_{\sigma_3}} \cdot \frac{\mu_{\sigma_3}}{d_3} \land 1, &\text{if $c_2 \leq \mu_D + \mu_{\sigma_3}$},\\
    \frac{c_2 - \mu_D}{d_3} \land 1, &\text{if $c_2 > \mu_D + \mu_{\sigma_3}$}.
    \end{cases} \label{eqn:beta_3_sig}
\end{align}

Note that these fill rates are dependent on both capacity $c_2$ and demand realization $d_3$. Therefore, we are going to compare the minimum fill rate in different capacity regimes and demand realizations scenarios. Specifically, $\beta^{(\Vsig', ~\Vpi ^ { P P A })}_{\sigma_2'} = \beta^{(\Vsig', ~\Vpi ^ { P P A })}_{i_D}$ always hold from \Cref{eqn:beta_2_sig_prime} and \Cref{eqn:beta_D_sig_prime}, thus, it is sufficient to compare $\beta^{(\Vsig', ~\Vpi ^ { P P A })}_{\sigma_2'}$ with $\beta^{(\Vsig, ~\Vpi ^ { P P A })}_{i_D} \land \beta^{(\Vsig, ~\Vpi ^ { P P A })}_{\sigma_3}$. Moreover, when $c_2 > d_3 + \mu_D$, we observe that fill rates from stage $2$ under both \routing policies are one and thus equal. Hence, we assume that $c_2 \leq d_3 + \mu_D$ for the rest of the discussion.

When $c_2 \leq \mu_{\sigma_3} + \mu_D$, there are two demand realization possibilities: $d_{3} \leq  \mu_{\sigma_3}$ and $d_{3} >  \mu_{\sigma_3}$.
Indeed, when $d_{3} \leq  \mu_{\sigma_3}$, we have $$\beta^{(\Vsig, ~\Vpi ^ { P P A })}_{i_D} \land \beta^{(\Vsig, ~\Vpi ^ { P P A })}_{\sigma_3} \overset{(i)}{=} \frac{c_2}{\mu_D + \mu_{\sigma_3}} \wedge \frac{c_2}{\mu_D + \mu_{\sigma_3}} \cdot \frac{\mu_{\sigma_3}}{d_{3}} \overset{(ii)}{=} \frac{c_2}{\mu_D + \mu_{\sigma_3}} \overset{(iii)}{\leq} \frac{c_2}{d_3 + \mu_D} \overset{(iv)}{=} \beta^{(\Vsig', ~\Vpi ^ { P P A })}_{\sigma_2'},$$
where equality $(i)$ is from \Cref{eqn:beta_D_sig} and \Cref{eqn:beta_3_sig} when $c_2 \leq \mu_{\sigma_3} + \mu_D$, equality $(ii)$ and inequality $(iii)$ is due to $d_{3} \leq  \mu_{\sigma_3}$, and equality $(iv)$ is from \Cref{eqn:beta_2_sig_prime}.

Moreover, when $d_3 >  \mu_{\sigma_3}$, we have $$\beta^{(\Vsig, ~\Vpi ^ { P P A })}_{i_D} \land \beta^{(\Vsig, ~\Vpi ^ { P P A })}_{\sigma_3} = \frac{c_2}{\mu_D + \mu_{\sigma_3}} \wedge \frac{c_2}{\mu_D + \mu_{\sigma_3}} \cdot \frac{\mu_{\sigma_3}}{d_{3}} \overset{(i)}{=} \frac{c_2}{\mu_D + \mu_{\sigma_3}} \cdot \frac{\mu_{\sigma_3}}{d_{3}} \overset{(ii)}{<} \frac{c_2}{d_3 + \mu_D} = \beta^{(\Vsig', ~\Vpi ^ { P P A })}_{\sigma_2'},$$ 
where equality $(i)$ is due to $d_3 >  \mu_{\sigma_3}$, inequality $(ii)$ is equivalent to $\frac{\mu_{\sigma_3}}{d_{3}} < \frac{\mu_{\sigma_3} + \mu_D}{d_{3} + \mu_D}$, which is true when $d_3 >  \mu_{\sigma_3}$ and $\mu_D > 0$.

Lastly we consider the case where $\mu_D + \mu_{\sigma_3} < c_2 \leq d_3 + \mu_D$.  Here we have $$\beta^{(\Vsig, ~\Vpi ^ { P P A })}_{i_D} \land \beta^{(\Vsig, ~\Vpi ^ { P P A })}_{\sigma_3} \overset{(i)}{=} \frac{c_2 - \mu_D}{d_3} \overset{(ii)}{\leq} \frac{c_2}{d_3 + \mu_D} \overset{(iii)}{=} \beta^{(\Vsig', ~\Vpi ^ { P P A })}_{\sigma_2'},$$
where equality $(i)$ is from \Cref{eqn:beta_D_sig} and \Cref{eqn:beta_3_sig} when $c_2 > \mu_D + \mu_{\sigma_3}$, inequality $(ii)$ is true from the capacity condition and $\mu_D > 0$, and equality $(iii)$ is from \Cref{eqn:beta_2_sig_prime} under the capacity condition.
As a result, $\Vsig'$ is a better (up to tie-breaking) \routing policy.

\textbf{Case II ($\sigma^*_1 \neq i_D$):} Let $\Vsig$ be any \routing policy (either dynamic or static) and $\sigma_1 = i_D$. Again we let $\Vsig'$ be the \routing policy such preserves the original visitation order except moving $i_D$ to the end for every sample path, i.e., $\sigma'_1 = \sigma_2$, $\sigma'_2 = \sigma_3$ and $\sigma'_3 = i_D$. We show that for any initial capacity $c$, we have $$ W_0^{(\Vsig', ~\Vpi^{PPA})}\left(c \right) \geq W_0^{(\Vsig, ~\Vpi^{PPA})}\left(c \right).$$ To this end, we prove that for any sample path, $\Vsig'$ achieves a better minimum fill rate than $\Vsig$ and then take the expectation. 

First, notice that by \Cref{eqn:PPA}, the minimum fill rate over stage $i$ and $i+1$ is determined by the demand realization at stage $i+1$. When $d_{i+1}$ is below the mean, the $i$-th stage's fill rate is lower; Otherwise, the $i+1$-th stage's fill rate is lower. We now consider a sample path $(\mu_D, ~d_2, ~d_3)$ under $\Vsig$ which corresponds to a sample path $(d_2, ~d_3, ~\mu_D)$ under $\Vsig'$. There are four possibilities for the demand realizations: 
\begin{enumerate*}[label=(\alph*)]
    \item $d_2 \leq \mu_{\sigma_2}$, $d_3 \leq \mu_{\sigma_3}$;
    \item $d_2 > \mu_{\sigma_2}$, $d_3 \leq \mu_{\sigma_3}$;
    \item $d_2 \leq \mu_{\sigma_2}$, $d_3 > \mu_{\sigma_3}$;
    \item $d_2 > \mu_{\sigma_2}$, $d_3 > \mu_{\sigma_3}$.
\end{enumerate*}

Note that $\beta^{(\Vsig', ~\Vpi ^ { P P A })}_{\Vsig'_1} = \beta^{(\Vsig, ~\Vpi ^ { P P A })}_{\sigma_1}$ is no longer true here, and fill rates under $\Vsig$ becomes

\begin{align}
    &\beta_{i_D}^{(\Vsig,~\pi^{PPA})} = \begin{cases}
    \frac{1}{\mu_D + \mu_{\sigma_2} + \mu_{\sigma_3}}c, &\text{if $c \leq \mu_{\sigma_2} + \mu_{\sigma_3} + \mu_D$,}\\
    1, &\text{otherwise.}
    \end{cases}\\
    &\beta_{\sigma_2}^{(\Vsig,~\pi^{PPA})} = \begin{cases}
    \frac{1}{\mu_D + \mu_{\sigma_2} + \mu_{\sigma_3}}c \cdot \frac{\mu_{\sigma_2} + \mu_{  \sigma_3}}{d_2 + \mu_{\sigma_3}}, &\text{if $c \leq \left( \mu_{\sigma_2} + \mu_{\sigma_3} + \mu_D \right) \cdot \frac{d_2 + \mu_{\sigma_3}}{\mu_{\sigma_2} + \mu_{\sigma_3}}$,}\\
    \frac{c-\mu_D}{d_2+\mu_{\sigma_3}}, &\text{if $\left( \mu_{\sigma_2} + \mu_{\sigma_3} + \mu_D \right) \cdot \frac{d_2 + \mu_{\sigma_3}}{\mu_{\sigma_2} + \mu_{\sigma_3}} < c \leq \mu_{D}+d_2+\mu_{\sigma_3}$,}\\
    1, &\text{otherwise.}
    \end{cases}\\
    &\beta_{\sigma_3}^{(\Vsig,~\pi^{PPA})} = \begin{cases}
    \frac{1}{\mu_D + \mu_{\sigma_2} + \mu_{\sigma_3}}c \cdot \frac{\mu_{\sigma_2} + \mu_{\sigma_3}}{d_2 + \mu_{\sigma_3}} \cdot \frac{\mu_{\sigma_3}}{d_3}, &\text{if $ c \leq \left( \mu_{\sigma_2} + \mu_{\sigma_3} + \mu_D \right) \cdot \frac{d_2 + \mu_{\sigma_3}}{\mu_{\sigma_2} + \mu_{\sigma_3}}$,}\\
    \frac{\frac{\mu_{\sigma_2} + \mu_{\sigma_3}}{\mu_D + \mu_{\sigma_2} + \mu_{\sigma_3}}c - d_2}{d_3} \land 1, &\text{if $\left( \mu_{\sigma_2} + \mu_{\sigma_3} + \mu_D \right) \cdot \frac{d_2 + \mu_{\sigma_3}}{\mu_{\sigma_2} + \mu_{\sigma_3}} < c \leq \mu_{\sigma_2} + \mu_{\sigma_3} + \mu_D$,}\\
    \frac{c - \mu_D - d_2}{d_3} \land 1,  &\text{if $c > \mu_{\sigma_2} + \mu_{\sigma_3} + \mu_D$.}\\
\end{cases}
\end{align}
Since we need to consider $\beta^{(\Vsig', ~\Vpi ^ { P P A })}_{\Vsig'_1}$ now under $\Vsig'$, we present another form of all fill rates,
\begin{align}
    &\beta_{\Vsig'_1}^{(\Vsig',~\pi^{PPA})} 
    = \frac{1}{d_2 + \mu_{\sigma_3} + \mu_D} c \land 1,\\
    &\beta_{\Vsig'_2}^{(\Vsig',~\pi^{PPA})} = \beta_{i_D}^{(\Vsig',~\pi^{PPA})} = \begin{cases}
    \frac{1}{d_2+ \mu_{\sigma_3} + \mu_D}c \cdot \frac{\mu_{\sigma_3} + \mu_D}{d_3 + \mu_D}, 
    &\text{if $c \leq d_2 + \mu_{\sigma_3} + \mu_D$,} \\
    \frac{c - d_2}{d_3 + \mu_{D}} \land 1, &\text{if $c > d_2 + \mu_{\sigma_3} + \mu_D$.}
\end{cases}
\end{align}

\textbf{Case (a) ($d_2 \leq \mu_{\sigma_2}$, $d_3 \leq \mu_{\sigma_3}$):} Since demands are both below the means, based on our discussion before the minimum fill rate under $\Vsig$ is $\beta_{i_D}^{(\Vsig,~\pi^{PPA})}$. Respectively, the minimum fill rate under $\Vsig'$ is $$\beta_{\Vsig'_1}^{(\Vsig',~\pi^{PPA})} = \frac{1}{d_2 + \mu_{\sigma_3} + \mu_D} c \land 1 \overset{(i)}{\geq} \frac{1}{\mu_D + \mu_{\sigma_2} + \mu_{\sigma_3}} c \land 1 = \beta_{i_D}^{(\Vsig,~\pi^{PPA})},$$
where inequality $(i)$ is from $d_2 \leq \mu_{\sigma_2}$.

\textbf{Case (b) ($d_2 > \mu_{\sigma_2}$, $d_3 \leq \mu_{\sigma_3}$):} Since demand at the second stage is above the mean, the minimum fill rate under $\Vsig$ is $\beta_{\sigma_2}^{(\Vsig,~\pi^{PPA})}$, and the minimum fill rate under $\Vsig'$ is $\beta_{\Vsig'_1}^{(\Vsig',~\pi^{PPA})}$.

\noindent When $c \leq \mu_D + \mu_{\sigma_2} + \mu_{\sigma_3}$ (and thus $c < d_2 + \mu_{\sigma_3} + \mu_D$),
$$\beta_{\Vsig'_1}^{(\Vsig',~\pi^{PPA})} = \frac{1}{d_2 + \mu_{\sigma_3} + \mu_D} c \overset{(i)}{\geq} \frac{1}{\mu_D + \mu_{\sigma_2} + \mu_{\sigma_3}} c \cdot \frac{\mu_{\sigma_2} + \mu_{\sigma_3}}{d_2 + \mu_{\sigma_3}} = \beta_{\sigma_2}^{(\Vsig,~\pi^{PPA})},$$
where inequality $(i)$ is equivalent to $\frac{\mu_{\sigma_2} + \mu_{\sigma_3} + \mu_D}{d_2 + \mu_{\sigma_3} + \mu_D} \geq \frac{\mu_{\sigma_2} + \mu_{\sigma_3}}{d_2 + \mu_{\sigma_3}}$ and is true when $d_2 > \mu_{\sigma_2}$ and $\mu_D > 0$.

When $\mu_D + \mu_{\sigma_2} + \mu_{\sigma_3} < c \leq d_2 + \mu_{\sigma_3} + \mu_D$, 
$$\beta_{\Vsig'_1}^{(\Vsig',~\pi^{PPA})} = \frac{1}{d_2 + \mu_{\sigma_3} + \mu_D} c \overset{(i)}{\geq} \frac{c - \mu_D}{d_2 + \mu_{\sigma_3}}= \beta_{\sigma_2}^{(\Vsig,~\pi^{PPA})},$$
where inequality $(i)$ is owing to $c \leq d_2 + \mu_{\sigma_3} + \mu_D$.

When $c > d_2 + \mu_{\sigma_3} + \mu_D$, both are equal to one.

\textbf{Case (c) ($d_2 \leq \mu_{\sigma_2}$, $d_3 > \mu_{\sigma_3}$):} Since $c$ is the reverse case of $(b)$, the minimum fill rate under $\Vsig$ is $\beta_{i_D}^{(\Vsig,~\pi^{PPA})} \land \beta_{\sigma_3}^{(\Vsig,~\pi^{PPA})}$. 

When $c \leq (\mu_D + \mu_{\sigma_2} + \mu_{\sigma_3})\frac{d_2 + \mu_{\sigma_3}}{\mu_{\sigma_2}+ \mu_{\sigma_3}}$, from the assumption on the distribution support, there can be two realizations: $d_2 = \mu_{\sigma_2} - k$ or $d_2 = \mu_{\sigma_2}$. If $d_2 = \mu_{\sigma_2}$, we have
\begin{align*}
    \beta_{i_D}^{(\Vsig,~\pi^{PPA})} \land \beta_{\sigma_3}^{(\Vsig,~\pi^{PPA})} &= \beta_{\sigma_3}^{(\Vsig,~\pi^{PPA})} = \frac{1}{\mu_D + \mu_{\sigma_2} + \mu_{\sigma_3}}c \cdot \frac{\mu_{\sigma_2} + \mu_{\sigma_3}}{d_2 + \mu_{\sigma_3}} \cdot \frac{\mu_{\sigma_3}}{d_3} = \frac{1}{\mu_D + \mu_{\sigma_2} + \mu_{\sigma_3}}c \cdot \frac{\mu_{\sigma_3}}{d_3}\\
    & \overset{(i)}{\leq} \frac{1}{\mu_D + \mu_{\sigma_2} + \mu_{\sigma_3}}c \cdot \frac{\mu_{\sigma_3} + \mu_D}{d_3 + \mu_D} = \frac{1}{d_2 + \mu_D + \mu_{\sigma_3}}c \cdot \frac{\mu_{\sigma_3} + \mu_D}{d_3 + \mu_D} = \beta_{\Vsig'_2}^{(\Vsig',~\pi^{PPA})},
\end{align*}
where inequality $(i)$ is due to $d_3 > \mu_{\sigma_3}$ and $\mu_D > 0$. If $d_2 = \mu_{\sigma_2} - k$, we have
\begin{align*}
    \beta_{\sigma_3}^{(\Vsig,~\pi^{PPA})} &= \frac{1}{\mu_D + \mu_{\sigma_2} + \mu_{\sigma_3}}c \cdot \frac{\mu_{\sigma_2} + \mu_{\sigma_3}}{d_2 + \mu_{\sigma_3}} \cdot \frac{\mu_{\sigma_3}}{d_3} = \frac{1}{\mu_D + \mu_{\sigma_2} + \mu_{\sigma_3}}c \cdot \frac{\mu_{\sigma_2} + \mu_{\sigma_3}}{\mu_{\sigma_2} - k + \mu_{\sigma_3}} \cdot \frac{\mu_{\sigma_3}}{\mu_{\sigma_3} + k}\\
    &\overset{(i)}{\leq} \frac{1}{\mu_{\sigma_2} - k + \mu_D + \mu_{\sigma_3}}c \cdot \frac{\mu_{\sigma_3} + \mu_D}{\mu_{\sigma_3} + k + \mu_D} = \frac{1}{\mu_{\sigma_2} - k + \mu_D + \mu_{\sigma_3}}c \cdot \frac{\mu_{\sigma_3} + \mu_D}{d_3 + \mu_D} = \beta_{\Vsig'_2}^{(\Vsig',~\pi^{PPA})},
\end{align*}
where inequality $(i)$ can be shown to be equivalent to $(\mu_{\sigma_2} + \mu_{\sigma_3} + \mu_D)(\mu_{\sigma_2} k - k^2) + \mu_{\sigma_3}(\mu_{\sigma_2} k - k^2) \geq 0$, which is trivially true since $\mu_{\sigma_2} - k \geq 0$. Therefore, $\beta_{i_D}^{(\Vsig,~\pi^{PPA})} \land \beta_{\sigma_3}^{(\Vsig,~\pi^{PPA})} \leq \beta_{\Vsig'_2}^{(\Vsig',~\pi^{PPA})}$.

When $(\mu_D + \mu_{\sigma_2} + \mu_{\sigma_3})\frac{d_2 + \mu_{\sigma_3}}{\mu_{\sigma_2}+ \mu_{\sigma_3}} < c \leq d_2 + \mu_{\sigma_3} + \mu_D$ (note that when $d_2 = \mu_{\sigma_2}$, this interval does not exist), we have 
\begin{align*}
    \beta_{\sigma_3}^{(\Vsig,~\pi^{PPA})} &= \frac{\frac{\mu_{\sigma_2} + \mu_{\sigma_3}}{\mu_D + \mu_{\sigma_2} + \mu_{\sigma_3}}c - d_2}{d_3} \land 1 \overset{(i)}{\leq} \frac{1}{d_2 + \mu_{\sigma_3} + \mu_D}c \cdot \frac{\mu_{\sigma_3} + \mu_D}{d_3 + \mu_D} = \beta_{\Vsig'_2}^{(\Vsig',~\pi^{PPA})},
\end{align*}
where inequality $(i)$ comes from the fact that $$\frac{\frac{\mu_{\sigma_2} + \mu_{\sigma_3}}{\mu_D + \mu_{\sigma_2} + \mu_{\sigma_3}}c - d_2}{d_3} \leq \frac{1}{d_2 + \mu_{\sigma_3} + \mu_D}c \cdot \frac{\mu_{\sigma_3} + \mu_D}{d_3 + \mu_D},$$ which is equivalent to $$\left(\frac{\mu_{\sigma_2} + \mu_{\sigma_3}}{(\mu_{\sigma_2} + \mu_{\sigma_3} + \mu_D)(d_2 + \mu_{\sigma_3} + \mu_D)} d_2 + \frac{\mu_D(\mu_{\sigma_3} + \mu_D)}{(\mu_{\sigma_2} + \mu_{\sigma_3} + \mu_D)(d_2 + \mu_{\sigma_3} + \mu_D) (d_3 + \mu_D)}d_2 \right)c \leq d_2,$$
and we can find an upper bound of the left-hand side
\begin{align*}
    &\quad \left(\frac{\mu_{\sigma_2} + \mu_{\sigma_3}}{(\mu_{\sigma_2} + \mu_{\sigma_3} + \mu_D)(d_2 + \mu_{\sigma_3} + \mu_D)} d_2 + \frac{\mu_D(\mu_{\sigma_3} + \mu_D)}{(\mu_{\sigma_2} + \mu_{\sigma_3} + \mu_D)(d_2 + \mu_{\sigma_3} + \mu_D) (d_3 + \mu_D)}d_2 \right)c \\
    &\overset{(i)}{\leq}\left(\frac{\mu_{\sigma_2} + \mu_{\sigma_3}}{(\mu_{\sigma_2} + \mu_{\sigma_3} + \mu_D)(d_2 + \mu_{\sigma_3} + \mu_D)} d_2 + \frac{\mu_D(d_3 + \mu_D)}{(\mu_{\sigma_2} + \mu_{\sigma_3} + \mu_D)(d_2 + \mu_{\sigma_3} + \mu_D) (d_3 + \mu_D)}d_2 \right)c \\
    &= \frac{\mu_{\sigma_2} + \mu_{\sigma_3} + \mu_D}{(\mu_{\sigma_2} + \mu_{\sigma_3} + \mu_D)(d_2 + \mu_{\sigma_3} + \mu_D)}d_2 c \overset{(ii)}{\leq} d_2,
\end{align*}
where inequality $(i)$ is due to $\mu_{\sigma_3} \leq d_3$ and inequality $(ii)$ comes from $c \leq d_2 + \mu_{\sigma_3} + \mu_D$.

When $d_2 + \mu_{\sigma_3} + \mu_D < c \leq \mu_{\sigma_2} + \mu_{\sigma_3} + \mu_D$, we have 
\begin{align*}
    &\quad \frac{\frac{\mu_{\sigma_2} + \mu_{\sigma_3}}{\mu_D + \mu_{\sigma_2} + \mu_{\sigma_3}}c - d_2}{d_3} - \frac{c-d_2}{d_3 + \mu_D},\\
    &= \frac{1}{d_3+\mu_D} \left[ \left(\frac{\mu_{\sigma_2} + \mu_{\sigma_3}}{\mu_D + \mu_{\sigma_2} + \mu_{\sigma_3}}c - d_2 \right) \frac{d_3 + \mu_D}{d_3} -(c - d_2) \right],\\
    &= \frac{1}{d_3+\mu_D} \left[\frac{\mu_D(\mu_{\sigma_2}+\mu_{\sigma_3} - d_3)}{d_3(\mu_{\sigma_2} + \mu_{\sigma_3} + \mu_D)}c - \frac{\mu_D}{d_3} d_2 \right] \overset{(i)}{\leq} 0,
\end{align*}
where inequality $(i)$ is from $c \leq \mu_{\sigma_2} + \mu_{\sigma_3} + \mu_D$ and the fact that either $d_2 + d_3 = \mu_{\sigma_2} + \mu_{\sigma_3}$ or $d_2 = \mu_{\sigma_2}$, $d_3 = \mu_{\sigma_3} + k$.
Therefore, 
$$\beta_{\sigma_3}^{(\Vsig,~\pi^{PPA})} = \frac{\frac{\mu_{\sigma_2} + \mu_{\sigma_3}}{\mu_D + \mu_{\sigma_2} + \mu_{\sigma_3}}c - d_2}{d_3} \land 1 \leq \frac{c-d_2}{d_3 + \mu_D} \land 1 = \beta_{\Vsig'_2}^{(\Vsig',~\pi^{PPA})}.$$
When $c > \mu_{\sigma_2} + \mu_{\sigma_3} + \mu_D$, we have
\begin{align*}
    \beta_{i_D}^{(\Vsig,~\pi^{PPA})} \land \beta_{\sigma_3}^{(\Vsig,~\pi^{PPA})} & \leq \beta_{\sigma_3}^{(\Vsig,~\pi^{PPA})} = \frac{c - \mu_D - d_2}{d_3} \land 1,\\
    &\overset{(i)}{\leq} \frac{c - \mu_D - d_2 + \mu_D}{d_3 + \mu_D} \land 1 = \frac{c - d_2}{d_3 + \mu_D} \land 1 = \beta_{\Vsig'_2}^{(\Vsig',~\pi^{PPA})},
\end{align*}
where inequality $(i)$ is owing to $\mu_D > 0$ and when $\frac{c - \mu_D - d_2}{d_3} \land 1 = \frac{c - \mu_D - d_2}{d_3}$, we know $\frac{c - \mu_D - d_2}{d_3} \leq 1$.
\textbf{Case (d) ($d_2 > \mu_{\sigma_2}$, $d_3 > \mu_{\sigma_3}$):} In this case, fill rates decrease in stage, so the minimum fill rate under $\Vsig$ is $\beta_{\sigma_3}^{(\Vsig,~\pi^{PPA})}$, and the minimum fill rate under $\Vsig'$ is $\beta_{\Vsig'_2}^{(\Vsig',~\pi^{PPA})}$.

When $c \leq d_2 + \mu_{\sigma_3} + \mu_D$, 
\begin{align*}
    \beta_{\sigma_3}^{(\Vsig,~\pi^{PPA})} &= \frac{1}{\mu_D + \mu_{\sigma_2} + \mu_{\sigma_3}}c \cdot \frac{\mu_{\sigma_2} + \mu_{\sigma_3}}{d_2 + \mu_{\sigma_3}} \cdot \frac{\mu_{\sigma_3}}{d_3} \overset{(i)}{\leq} \frac{1}{\mu_D + \mu_{\sigma_2} + \mu_{\sigma_3}}c \cdot \frac{\mu_{\sigma_2} + \mu_{\sigma_3} + \mu_D}{d_2 + \mu_{\sigma_3} + \mu_D} \cdot \frac{\mu_{\sigma_3}}{d_3},\\
    &= \frac{1}{d_2 + \mu_{\sigma_3} + \mu_D}c \cdot \frac{\mu_{\sigma_3}}{d_3} \overset{(ii)}{\leq} \frac{1}{d_2 + \mu_{\sigma_3} + \mu_D}c \cdot \frac{\mu_{\sigma_3}+ \mu_D}{d_3 + \mu_D} = \beta_{\Vsig'_2}^{(\Vsig',~\pi^{PPA})}.
\end{align*}
where inequality $(i)$ is from $d_2 > \mu_{\sigma_2}$, inequality $(ii)$ is true when $d_3 > \mu_{\sigma_3}$ and $\mu_D > 0$.

When $d_2 + \mu_{\sigma_3} + \mu_D < c < d_2 + d_3 + \mu_D$,
$$\beta_{\Vsig'_2}^{(\Vsig',~\pi^{PPA})} = \frac{c - d_2}{d_3 + \mu_D} \overset{(i)}{\geq} \frac{c - \mu_D - d_2}{d_3} = \beta_{\sigma_3}^{(\Vsig,~\pi^{PPA})},$$
where inequality $(i)$ is because $c < d_2 + d_3 + \mu_D$.

When $c \geq d_2 + d_3 + \mu_D$, both are equal to one.
\end{proof}

\DecreasingVarTwoNode*
\begin{proof}
    Without loss of generality, we assume that node one has a higher variance, i.e., $p_1 < p_2$ (since the larger the central probability mass is, the smaller the variance is). Let $\Vsig = (1,2)$ and $\Vsig' = (2,1)$. We want to show that for any capacity $c$, $W_0^{(\Vpi^{PPA},~\Vsig)}(c) \geq W_0^{(\Vpi^{PPA},~\Vsig')}(c)$ and hence $\Vsig$ is optimal. To this end, for every capacity level and $d_1 \in \{\mu - k,~\mu, ~\mu + k\}$, we find the closed-form of $W_1^{(\Vpi^{PPA},~\Vsig)}(c,~d_1,~1)$. Then we take the expectation to derive $W_0^{(\Vpi^{PPA},~\Vsig)}(c)$ and show that the difference between $W_0^{(\Vpi^{PPA},~\Vsig)}(c)$ and $W_0^{(\Vpi^{PPA},~\Vsig')}(c)$ is always non-negative. 

    By \Cref{eqn:m_bellman} and \Cref{eqn:PPA}, when $d_1 = \mu - k$, we have
    \begin{equation*}
        W_1^{(\Vpi^{PPA},~\Vsig)}(c,~\mu-k,~1) = \E_{d_2} \left[\frac{1}{2\mu - k}c \land \frac{\frac{\mu}{2\mu - k}c}{d_2} \land 1 \right] =
        \begin{cases}
            \frac{1}{2\mu - k}c \cdot \frac{2\mu+ (1+p_2)k}{2(\mu+k)}, &\text{if $0 \leq c \leq 2\mu - k$,}\\
            \frac{1+p_2}{2} + \frac{c - \mu + k}{\mu + k} \cdot \frac{1-p_2}{2}, &\text{if $2\mu - k < c \leq 2\mu$,}\\
            1, &\text{if $c > 2\mu$}.
        \end{cases}
    \end{equation*}
    Similarly, when $d_1 = \mu$, we have 
    \begin{equation*}
        W_1^{(\Vpi^{PPA},~\Vsig)}(c,~\mu,~1) = \E_{d_2} \left[\frac{1}{2 \mu}c \land \frac{\frac{\mu}{2\mu}c}{d_2} \land 1 \right] =
        \begin{cases}
            \frac{1}{2\mu}c \cdot \frac{2\mu+ (1+p_2)k}{2(\mu+k)}, &\text{if $0 \leq c \leq 2\mu$,}\\
            \frac{1+p_2}{2} + \frac{c-\mu}{\mu + k} \cdot \frac{1 - p_2}{2}, &\text{if $2\mu < c \leq 2\mu + k$,}\\
            1, &\text{if $c > 2\mu + k$}.
        \end{cases}
    \end{equation*}
    Lastly, when $d_1 = \mu + k$, we have
    \begin{equation*}
        W_1^{(\Vpi^{PPA},~\Vsig)}(c,~\mu+k,~1) = \E_{d_2} \left[\frac{1}{2 \mu + k}c \land \frac{\frac{\mu}{2\mu + k}c}{d_2} \land 1 \right] =
        \begin{cases}
            \frac{1}{2\mu + k}c \cdot \frac{2\mu+ (1+p_2)k}{2(\mu+k)}, &\text{if $0 \leq c \leq 2\mu + k$,}\\
            \frac{1+p_2}{2} + \frac{c-\mu - k}{\mu + k} \cdot \frac{1 - p_2}{2}, &\text{if $2\mu + k < c \leq 2\mu + 2k$,}\\
            1, &\text{if $c > 2\mu + 2k$}.
        \end{cases}
    \end{equation*}
    Then by taking the expectation over $d_1$, we have
    {\small
    \begin{equation*}
        \begin{split}
            & \quad W_0^{(\Vpi^{PPA},~\Vsig)}(c) = \E_{d_1} \left[W_1^{(\Vpi^{PPA},~\Vsig)}(c,~d_2,~1) \right]\\
        &=
        \begin{cases}
            \frac{2\mu+ k(1+p_{2})}{2(\mu+k)} c \cdot \frac{4 \mu^2 - k^2 p_1}{2\mu (2\mu - k)(2\mu+k)}, &\text{if $0 \leq c \leq 2\mu - k$},\\
            \frac{c + 2k +1}{4(\mu+k)} + \frac{(2\mu + k) \left(\mu + k - \mu(c+2k) \right)}{4 \mu (\mu + k)(2\mu+ k)}p_1 + \frac{\mu(2\mu + k)(2\mu - c) + \mu k}{4 \mu (\mu + k)(2\mu+ k)}p_2 + \frac{k(\mu+k) - \mu(2\mu + k)(2\mu - c)}{4 \mu (\mu + k)(2\mu+ k)}p_1p_2, &\text{if $2\mu - k < c \leq 2\mu$},\\
            \frac{2(\mu + k) + c}{4(\mu + k)} + \frac{c - 2\mu}{4(\mu+k)}p_1 + \frac{ck}{4(\mu +k)(2\mu + k)}p_2 + \frac{2(2\mu + k -c)(2\mu + k) - ck}{4(\mu + k)(2\mu + k)}p_1 p_2, &\text{if $2\mu < c \leq 2\mu +k$},\\
            \frac{c+2(\mu+k)}{4(\mu+k)} + \frac{2(\mu+k) - c}{4(\mu+k)} p_1 + \frac{2(\mu+k) - c}{4(\mu+k)} p_2 + \frac{c - 2(\mu+k)}{4(\mu+k)} p_1 p_2, &\text{if $2\mu + k < c \leq 2\mu +2k$},\\
            1, &\text{if $c > 2\mu + 2k$}.
        \end{cases}
        \end{split}
    \end{equation*}}
    
    When $0 \leq c \leq 2\mu - k$, the decreasing variance is clearly optimal since the objective is increasing in $p_2$ and decreasing in $p_1$, which indicates that increasing the first node's variance and decreasing the second node's variance increases the objective value.

    When $2\mu - k < c \leq 2 \mu$, we calculate the objective difference between two \routing policy. Note that we obtain the expression of $W_0^{(\Vpi^{PPA},~\Vsig')}(c)$ by switching $p_1$ and $p_2$. Therefore, when calculating the difference, we can ignore the constant and the cross term $p_1 p_2$. Then we have
    \begin{align*}
        W_0^{(\Vpi^{PPA},~\Vsig)}(c) - W_0^{(\Vpi^{PPA},~\Vsig)'}(c) = \frac{(2\mu + k)(\mu + k)(2\mu -1) + \mu k}{4 \mu (\mu + k)(2\mu+ k)}(p_2 - p_1),
    \end{align*}
    which is positive when $\mu \geq \frac{1}{2}$. When $\mu < \frac{1}{2}$ this can be satisfied by multiplying the available capacity $c$ and the demands $d \in \text{supp}(\mathcal{D}_i)$ for all $i$ by a shared constant.  This does not affect the objective value.

    When $2\mu < c \leq 2 \mu + k$, the difference is 
    \begin{align*}
        W_0^{(\Vpi^{PPA},~\Vsig)}(c) - W_0^{(\Vpi^{PPA},~\Vsig')}(c) = \frac{2\mu (2\mu + k -c)}{4 \mu (\mu + k)(2\mu+ k)} (p_2 - p_1)\geq 0.
    \end{align*}
    Lastly, when $2\mu + k < c$, the two objectives obtain the same value and thus are both optimal.
\end{proof}

\IncreasingMeanTwoNode*
\begin{proof}
    Without loss of generality, we assume that the first node has a smaller mean, so $\mu_1 < \mu_2$.  Let $\Vsig = (1,2)$ denote the \routing policy that visits node one first before node two, and $\Vsig' = (2,1)$, the reverse order.
    We prove that for any capacity $c$, $W_0^{(\Vpi^{PPA},~\Vsig)}(c) \geq W_0^{(\Vpi^{PPA},~\Vsig')}(c)$ and hence $\Vsig$ is optimal. 

    First, note that by assumption, for any demand $d_1 \sim \D_1$, we can map it to a demand in the support of $\D_2$ by $f(d_1) = d_1 + \Delta$, where $\Delta = \mu_2 - \mu_1$. Since this mapping is one-to-one and onto, we can also map any demand $d_2 \sim \D_2$ to a demand in the support of $\D_1$ by $f^{-1}(d_2) = d_2 - \Delta$. We show that for any demand $d_1$, $$W_1^{(\Vpi^{PPA},~\Vsig)}(c,~d_1,~1) \geq W_1^{(\Vpi^{PPA},~\Vsig')}(c,~f(d_1), ~1).$$
    Then by taking the expectation on both sides, we prove that $W_0^{(\Vpi^{PPA},~\Vsig)}(c) \geq W_0^{(\Vpi^{PPA},~\Vsig')}(c)$.

    By \Cref{eqn:m_bellman} and \Cref{eqn:PPA}, we can rewrite the objective under $\Vsig$ as
    \begin{equation*}
        W_1^{(\Vpi^{PPA},~\Vsig)}(c,~d_1,~1) = \E_{d_2 \sim \D_2} \left[ W_2^{(\Vpi^{PPA},~\Vsig)}(c,~d_2,~\frac{\pi^{PPA}_1}{d_1} \land 1) \right] = \E_{d_2 \sim \D_2} \left[\frac{1}{d_1 + \mu_2}c \land \frac{1}{d_1 + \mu_2}c \cdot \frac{\mu_2}{d_2} \land 1 \right].
    \end{equation*}

    Similarly, we rewrite the objective under $\Vsig'$:
    \begin{equation*}
    \begin{split}
        W_1^{(\Vpi^{PPA},~\Vsig')}(c,~f(d_1),~1) &= \E_{d_2^{\prime} \sim \D_1} \left[ W_2^{(\Vpi^{PPA},~\Vsig')}(c,~d_2^{\prime},~\frac{\pi^{PPA}_1}{f(d_1)} \land 1) \right] = \E_{d_2^{\prime} \sim \D_1} \left[\frac{1}{f(d_1) + \mu_1}c \land \frac{\frac{\mu_1}{f(d_1) + \mu_1}c}{d_2^{\prime}} \land 1 \right],\\
        &= \E_{d_2^{\prime} \sim \D_1} \left[\frac{1}{d_1 + \Delta + \mu_1}c \land \frac{\frac{\mu_1}{d_1 + \Delta + \mu_1}c}{d_2^{\prime}} \land 1 \right] = \E_{d_2^{\prime} \sim \D_1} \left[\frac{1}{d_1 + \mu_2}c \land \frac{\frac{\mu_1}{d_1 + \mu_2}c}{d_2^{\prime}} \land 1 \right],\\
        &\overset{(i)}{=} \E_{f(d_2^{\prime}) \sim \D_2} \left[\frac{1}{d_1 + \mu_2}c \land \frac{\frac{\mu_1}{d_1 + \mu_2}c}{f^{-1}(f(d_2^{\prime}))} \land 1 \right], \\
        &= \E_{f(d_2^{\prime}) \sim \D_2} \left[\frac{1}{d_1 + \mu_2}c \land \frac{1}{d_1 + \mu_2}c \cdot \frac{\mu_1}{f(d_2^{\prime}) - \Delta} \land 1 \right],
    \end{split}
    \end{equation*}
    where equality $(i)$ is due to the change of variable.
    The difference between the two objectives is 
    \begin{align*}
        &\quad W_1^{(\Vpi^{PPA},~\Vsig)}(c,~d_1,~1) - W_1^{(\Vpi^{PPA},~\Vsig')}(c,~f(d_1),~1) \\
        &= \E_{d_2 \sim \D_2} \left[\frac{1}{d_1 + \mu_2}c \land \frac{1}{d_1 + \mu_2}c \cdot \frac{\mu_2}{d_2} \land 1 \right] - \E_{f(d_2^{\prime}) \sim \D_2} \left[\frac{1}{d_1 + \mu_2}c \land \frac{1}{d_1 + \mu_2}c \cdot \frac{\mu_1}{f(d_2^{\prime}) - \Delta} \land 1 \right],\\
        &\overset{(i)}{=} \E_{d_2 \sim \D_2} \left[ \left( \frac{1}{d_1 + \mu_2}c \cdot \frac{\mu_2}{d_2} \land 1 - \frac{1}{d_1 + \mu_2}c \cdot \frac{\mu_1}{d_2 - \Delta} \land 1 \right) \mathbf{1} \left\{ d_2 \geq \mu_2 \right\} \right] \overset{(ii)}{\geq} 0,
    \end{align*}
    where equality $(i)$ is from the fact that $f(d_2^{\prime})$ is a dummy variable and the minimum fill rates for both objectives are $\frac{1}{d_1 + \mu_2}c \land 1$ when $d_2 < \mu_2$, and equality $(ii)$ is because $\dfrac{\mu_2}{d_2} \geq \dfrac{\mu_2 - \Delta}{d_2 - \Delta} = \dfrac{\mu_1}{d_2 - \Delta}$ when $d_2 \geq \mu_2$.
\end{proof}

%% file: parts/appendix/distirbution_dict.tex
\section{List of Distributions} \label{appendix:distr_dict}

Here we list the demand support and probability mass for the distributions in \Cref{fig:distr_pmf}.

\begin{enumerate}
    \item Uniform: $D_1=\{1:1/5,~2:1/5,~3:1/5,~4:1/5,~5:1/5\}$.
    \item Symmetric concave: $D_2=\{1:1/10,~2:1/5,~3:2/5,~4:1/5,~5:1/10\}$.
    \item Symmetric convex: $D_3=\{1:2/5,~2:3/40,~3:1/20,~4:3/40,~5:2/5\}$.
    \item Decreasing: $D_4=\{1:2/5,~2:3/10,~3:1/5,~4:7/100,~5:3/100\}$.
    \item Increasing: $D_5=\{1:3/100,~2:7/100,~3:1/5,~4:3/10,~5:2/5\}$.
    \item Second peak: $D_6=\{1:3/10,~2:2/5,~3:1/5,~4:7/100,~5:3/100\}$.
    \item Second valley: $D_7=\{1:1/5,~2:7/100,~3:3/100,~4:3/10,~5:2/5\}$.
    \item W-shape: $D_8=\{1:2/5,~2:3/100,~3:3/10,~4:7/100,~5:1/5\}$.
\end{enumerate}